\documentclass[a4paper,10pt]{article}
\usepackage{amsmath,amssymb,amsfonts,amsthm}
\usepackage[all,cmtip]{xy}
\usepackage{arydshln}
\usepackage{a4wide}
\usepackage[usenames,dvipsnames]{color}
\usepackage[verbose,colorlinks=true,linktocpage=true,linkcolor=blue,citecolor=blue]{hyperref}
\usepackage{fouridx}

\theoremstyle{definition}
\newtheorem{definition}{Definition}[section]

\theoremstyle{plain}
\newtheorem{theorem}[definition]{Theorem}

\newtheorem{proposition}[definition]{Proposition}
\newtheorem{lemma}[definition]{Lemma}

\newtheorem{corollary}[definition]{Corollary}

\theoremstyle{remark}
\newtheorem{remark}[definition]{Remark}

\numberwithin{equation}{section}

\newcommand{\Y}{\mathrm{Y}}
\newcommand{\YR}{\mathrm{Y}}
\newcommand{\YD}{\widehat{\mathrm{Y}}}
\newcommand{\U}{\mathrm{U}}

\allowdisplaybreaks[4]

\begin{document}
\title{A Drinfeld Presentation of the Queer Super-Yangian}
\author{Zhihua Chang ${}^1$ and Yongjie Wang ${}^{2}$}
\maketitle

\begin{center}
\footnotesize
\begin{itemize}
\item[1] School of Mathematics, South China University of Technology, Guangzhou, 510640, China.
\item[2] School of Mathematics, Hefei University of Technology, Hefei, Anhui, 230009, China.
\end{itemize}
\end{center}

\begin{abstract} 
We introduce a Drinfeld presentation for the super-Yangian $\mathrm{Y}(\mathfrak{q}_n)$ associated with the queer Lie superalgebra $\mathfrak{q}_n$. The Drinfeld generators of $\mathrm{Y}(\mathfrak{q}_n)$ are obtained through a block Gauss decomposition of the generator matrix in its RTT presentation, and the Drinfeld relations are explicitly computed by utilizing a block version of its RTT relations. As a byproduct, we obtain a new expression for the central series of $\mathrm{Y}(\mathfrak{q}_n)$ in terms of Gauss generators.
\bigskip

\noindent\textit{MSC(2020):} 17B37, 16T20, 20G42.
\bigskip

\noindent\textit{Keywords:} Super-Yangian; Queer Superalgebra; Drinfeld Presentation; Gauss Decomposition; Center. 
\end{abstract}

\section{Introduction}

The Yangian first appeared in the study of quantum inverse scattering methods by Faddeev and the St. Petersburg school. It was named Yangian by V. Drinfeld in his paper in 1985 \cite{D85} in honor of the famous physicist Chen-Ning Yang. Yangian is the quantum deformation of the universal enveloping algebra of the current Lie algebra $\mathfrak{g}[u]$ associated with a finite-dimensional complex Lie algebra $\mathfrak{g}$. As a Hopf algebra, the Yangian $\mathrm{Y}(\mathfrak{g})$ of a complex simple Lie algebra $\mathfrak{g}$ possesses both algebraic and co-algebraic structures, which are realized through generators and defining relations \cite{D88}. We also refer to \cite{M07, N94} and the references therein for more results on Yangians.
\medskip

Yangians usually possess three distinct realizations: the Drinfeld-Jimbo realization, the Drinfeld new realization, and the RTT realization. These realizations are analogous to the Chevalley-Serre realization, the loop realization, and the matrix realization of affine Lie algebras, respectively. It is accepted among experts that the Yangians defined by these three realizations are isomorphic to each other, but their proofs are quite complicated. Inspired by the work of Ding and Frenkel in \cite{DF93} on the quantum affine Lie algebra of type $A$, Brundan and Kleshchev proved the equivalence between the RTT realization and the Drinfeld new realization of the Yangian $\mathrm{Y}(\mathfrak{gl}_n)$ by using the Gauss decomposition in \cite{BK05}. Recently, this technique has also been successfully employed to establish an isomorphism between the RTT realization and the Drinfeld new realization of Yangians in types $B, C$ and $D$, as well as twisted Yangians in \cite{JLM18, LWZ23, LWZ24}. Simultaneously, Guay, Regelskis, and Wendlandt have also provided an independent proof of the isomorphism among the three presentations of orthogonal and symplectic Yangians \cite{GRW18}.
\medskip

Associated with the generalization to supersetting, many researchers have investigated the super-Yangians associated with finite-dimensional complex Lie superalgebras in recent years.  In 1991, Nazarov defined the super-Yangian $\mathrm{Y}(\mathfrak{gl}_{m|n})$ of $\mathfrak{gl}_{m|n}$ and conjectured that its central generators can be represented by the quantum Berezinian in \cite{N91}, which has been proven by Gow \cite{G05, G07}. Zhang classified the finite-dimensional irreducible representations of the super-Yangian algebra $\mathrm{Y}(\mathfrak{gl}_{1|1})$ and determined their corresponding Drinfeld polynomials in \cite{Zh95, M23b}, subsequently completing the classification of the finite-dimensional irreducible representations of the super-Yangian $\mathrm{Y}(\mathfrak{gl}_{m|n})$ in \cite{Zh96}. Based on Brundan and Kleshchev's work, Gow and Peng obtained the Drinfeld presentation and parabolic presentation of the super-Yangian $\mathrm{Y}(\mathfrak{gl}_{m|n})$ in \cite{G07} and \cite{P11, P16}, respectively.
\medskip

The ortho-symplectic super-Yangian $\mathrm{Y}(\mathfrak{osp}_{m|2n})$ was introduced in \cite{AACFR03} via the RTT presentation, whose Drinfeld realization was recently obtained in \cite{M24a}. Consequently, Molev and Ragoucy have classified the finite-dimensional irreducible representations of $\mathrm{Y}(\mathfrak{osp}_{m|2n})$ in terms of Drinfeld polynomials in a series of papers \cite{M23a, M23b, M24a, MR24a, MR24b}.
Analogous to the twisted Yangian of types $B, C$ and $D$, Briot and Ragoucy have also introduced the twisted super-Yangian $\mathrm{Y}^{\mathrm{tw}}(\mathfrak{osp}_{m|2n})$ as fixed-point subalgebras of the super-Yangian $\mathrm{Y}(\mathfrak{gl}_{m|2n})$ under finite-order automorphisms. The finite-dimensional irreducible representations of $\mathrm{Y}^{\mathrm{tw}}(\mathfrak{osp}_{m|2n})$ were classified in terms of Drinfeld polynomials, as discussed in \cite{BR03}. These subalgebras are usually not Hopf algebras, but rather left coideal subalgebras of $\mathrm{Y}(\mathfrak{gl}_{m|2n})$. 
Recently, Lin, Zhang, and the second author of this article have investigated the quantum Berezinian of the twisted super-Yangian $\mathrm{Y}^{\mathrm{tw}}(\mathfrak{osp}_{m|2n})$ in \cite{LiWZ24}, and obtained the center of the twisted super-Yangian. This discovery enables the definition of the special twisted super-Yangian, which is isomorphic to the quotient of the twisted super-Yangian by its center. 
\medskip

Besides the basic classical Lie superalgebras, there are two families of strange Lie superalgebras, queer Lie superalgebras and periplectic Lie superalgebras, according to Kac's classification of the finite-dimensional complex simple Lie superalgebras \cite{Kac77}. 
Let $V=\mathbb{C}^{n|n}$ be a $2n$-dimensional vector superspace, and let $\zeta\in\mathrm{End}(V)$ be an odd linear map such that $\zeta^2=-1$. Then the queer Lie superalgebra $\mathfrak{q}_n$ is the centralizer subalgebra of $\zeta$ in the general linear superalgebra $\mathfrak{gl}(V)$. It is one of the strange Lie superalgebras, which can be realized as the following $n|n$-block matrices: 
$$\left\{\begin{pmatrix}
A&B\\
B&A
\end{pmatrix}| A \text{ and }B \text{ are }n\times n\text{ matrices}\right\}.$$
The queer Lie superalgebra has a unique structure that distinguishes it from other Lie superalgebras. For instance, the Cartan subalgebra $\mathfrak{h}$ is a solvable but non-abelian Lie superalgebra, which partly leads to the existence of its queer nature, including the Clifford module structure on the highest weight space of an irreducible highest weight $\mathfrak{q}_n$-mdoule $V(\lambda)$ \cite{CW12}. It does not possess an even nondegenerate invariant bilinear form, but admits an odd nondegenerate invariant bilinear form. The $r$-matrix in $\mathfrak{q}_n^{\otimes2}$ does not satisfy the classical Yang-Baxter equation. This highly peculiar phenomenon has resulted in its quantum deformation $\mathrm{U}_q(\mathfrak{q}_n)$, the affinization of the twisted affine queer superalgebra, and the queer super-Yangian $\mathrm{Y}(\mathfrak{q}_n)$, see \cite{CG12, LMZ25, N99, Ol92} and the references therein. 
\medskip

Due to some unique properties, the super-Yangian associated with a Lie superalgebra of type $P$ or $Q$ cannot be defined as a deformation of the universal enveloping algebra of the polynomial current Lie superalgebra in the class of Hopf algebras. However, Nazarov introduced the super-Yangians of the strange Lie superalgebras in \cite{N92}, which are deformations of the universal enveloping algebra of the twisted polynomial current Lie superalgebra; more details can be found in \cite{N99,N24}.
\medskip

Nazarov investigated the center and Drinfeld double of the super-Yangians $\mathrm{Y}(\mathfrak{q}_n)$ and $\mathrm{Y}(\mathfrak{p}_n)$ in \cite{N99} and \cite{N24}, respectively. He also constructed a class of finite-dimensional irreducible representations of $\mathrm{Y}(\mathfrak{q}_n)$ via the representations of degenerate affine Sergeev algebras. Furthermore, the centralizer construction of $\mathrm{Y}(\mathfrak{q}_n)$ was obtained by Nazarov and Sergeev in \cite{NS06}. Thus, these super-Yangians were also realized as limits of certain centralizers in the universal enveloping algebras of type $Q$. 

Analogously to the role of Clifford algebras in the representation theory of queer Lie superalgebras, the queer super-Yangian $\mathrm{Y}(\mathfrak{q}_1)$ plays a significant role in investigating the queer super-Yangian $\mathrm{Y}(\mathfrak{q}_n)$. Numerous related results revolve around this special case, including the quantum Berezinian \cite{N22} and its connection with finite $W$-superalgebras for any principal nilpotent element \cite{Po22, Po23, PS16, PS17, PS21}. However, based on the authors' knowledge, the classification of finite-dimensional representations of the queer super-Yangian has not been established up to this point. One of the main obstacles to this is the absence of the Drinfeld presentation. In this paper, we present a Drinfeld realization of the super-Yangian $\mathrm{Y}(\mathfrak{q}_n)$ and express the central series in terms of Gauss generators. Surprisingly, the central series $z(u)$ of the super-Yangian $\mathrm{Y}(\mathfrak{q}_n)$ is precisely the product of $n$ copies of the central series of $\mathrm{Y}(\mathfrak{q}_1)$, as detailed in  Theorem \ref{thm:ct} for more details. Consequently, we express the central series in terms of the quantum Berezinian defined by Nazarov in \cite{N22}.
\medskip

Our approach to obtaining a Drinfeld presentation of the super-Yangian $\mathrm{Y}(\mathfrak{q}_n)$ takes advantage of a block version of the Gauss decomposition, which is inspired by \cite{BK05}. We rearrange a matrix of type $Q$ such that the rows and columns of the matrix are indexed by the numbers $1,-1,2,-2,\ldots,n,-n$. This allows us to view the matrix as an $n\times n$-block matrix with each entry in $Q(1)$. Using $\mathfrak{q}_2$ as an illustrative example, we consider the matrix
$$\left(\begin{array}{cc|cc}
a_{11}&a_{12}&b_{-1,1}&b_{-1,2}\\
a_{21}&a_{22}&b_{-2,1}&b_{-2,2}\\
\hline
b_{-1,1}&b_{-1,2}&a_{11}&a_{12}\\
b_{-2,1}&b_{-2,2}&a_{21}&a_{22}
\end{array}\right)
$$
which can be rearranged as 
$$\left(\begin{array}{cc|cc}
a_{11}&b_{-1,1}&a_{12}&b_{-1,2}\\
b_{-1,1}&a_{11}&b_{-1,2}&a_{12}\\ 
\hline
a_{21}&b_{-2,1}&a_{22}&b_{-2,2}\\
b_{-2,1}&a_{21}&b_{-2,2}&a_{22}
\end{array}\right).$$

The advantage of this arrangement is that the quasi-diagonal block matrices form a Cartan subalgebra $\mathfrak{h}$ of the queer Lie superalgebra $\mathfrak{q}_2$. Moreover, the positive (negative) part of $\mathfrak{q}_2$ with respect to $\mathfrak{h}$ consists of the upper (lower) quasi-triangular block matrices. Inspired by this point of view, we rearrange the generator matrix $T(u)$ of the queer super-Yangian $\Y(\mathfrak{q}_n)$ such that it can be viewed as an $n\times n$ block matrix, whose entries are $2\times 2$-matrix of a certain form. Then the block Gauss decomposition will yield a family of appropriate Drinfeld generators.
\medskip

In order to effectively calculate the relations among these Drinfeld generators, we reformulate the RTT relation for the queer super-Yangian in block form. Combining this with the Gauss decomposition and the embedding theorem of the Yangian $\mathrm{Y}(\mathfrak{gl}_n)$, we deduce our main theorem (Theorem \ref{thm:isomorphism}) which asserts the equivalence between the Drinfeld presentation and the RTT presentation of $\YR(\mathfrak{q}_n)$.
\medskip

We have observed that Stukopin introduced a Drinfeld presentation of $\mathrm{Y}(\mathfrak{q}_1)$ in \cite{S20} and conjectured a Drinfeld realization of the general case $\mathrm{Y}(\mathfrak{q}_n)$ in \cite{S18}. His approach takes a complete Gauss decomposition of the generator matrix in the RTT presentation of $\Y(\mathfrak{q}_n)$, which leads to a triangular decomposition of $\Y(\mathfrak{q}_n)$ with a purely even diagonal part. Our approach employs a block Gauss decomposition instead and yields a different triangular decomposition, in which the diagonal part contains odd elements as well. As a corollary, we derive the central series $z(u)$ obtained by Nazarov \cite[Proposition 3.1]{N99} in terms of Drinfeld generators. Based on these facts, we believe that such a new triangular decomposition is advantageous for investigating the highest weight modules of $\Y(\mathfrak{q}_n)$.
\medskip

Further directions: The Drinfeld presentation of $\Y(\mathfrak{q}_n)$ leads to its triangular decomposition, which motivates subsequent investigations on its highest weight representations. An immediate problem is to classify finite-dimensional irreducible highest weight modules of $\mathrm{Y}(\mathfrak{q}_n)$ in terms of Drinfeld polynomials. We will address this problem in upcoming papers. 
\medskip

Brundan and Kleshchev established an isomorphism between the truncated shifted Yangian $\mathrm{Y}_{n,l}(\sigma)$ and the finite $W$-algebra associated to any nilpotent element of the general linear Lie algebra $\mathfrak{gl}_N$ in the literature. This allows for the derivation of generators and relations for the finite $W$-algebras associated with nilpotent matrices and investigating the highest weight theory for shifted Yangians and finite $W$-algebras, see \cite{BK06, BK08}. This profound connection has recently been extended to the case of other classical Lie algebras by Lu, Peng, Tappeiner, Toply, and Wang in \cite{LPTTW25}.
Meanwhile, the theory has been generalized to the finite $W$-superalgebra associated to any even nilpotent matrix in the general linear Lie superalgebra by Peng \cite{P21}.  The following question for the queer super-Yangian $\mathrm{Y}(\mathfrak{q}_n)$ is natural and fundamental compared to Yangian $\mathrm{Y}(\mathfrak{gl}_n)$: What is the relationship between the parabolic presentation of $\mathrm{Y}(\mathfrak{q}_n)$ and the finite $W$-superalgebra of type $Q$. 
\medskip

We acknowledge that Poletaeva and Serganova have made significant contributions in this area. More details can be found in \cite{Po22, Po23, PS17, PS21}, which primarily focus on the specific case of $\mathfrak{q}_1$. However, the more general case remains unknown. 
Since the parabolic realization plays a crucial role in establishing the connection between finite $W$-algebras and truncated shifted Yangians, we will proceed to construct the parabolic realization of $\Y(\mathfrak{q}_n)$. We anticipate that our Drinfeld realization and its parabolic counterpart will provide a more comprehensive investigation of the finite $W$-superalgebra of type $Q$. Notably, the parabolic presentations for twisted Yangians of types AI and AII, as well as for the extended Yangian  of types $B$ and $C$, have already been established in \cite{LPTTW25} and \cite{CJLM25}, respectively. 
\medskip

The paper is organized as follows. In Section \ref{se:Yqdef}, we review some basic facts related to the super-Yangian $\mathrm{Y}(\mathfrak{q}_n)$, and rewrite the RTT relation for $\mathrm{Y}(\mathfrak{q}_n)$ in block form. In Section \ref{se:Gaussdec}, we establish the Gauss decomposition and obtain the Drinfeld generators for $\Y(\mathfrak{q}_n)$. An embedding theorem for the super-Yangian $\mathrm{Y}(\mathfrak{q}_n)$ is proved in Section \ref{se:Embedding}, which reduces the calculation of the Drinfeld relations to the lower-rank case.  In Sections \ref{se:q2relation} and \ref{se:serrerelation}, we explicitly calculate the relations among the Drinfeld generators. We state and prove our main theorem about the equivalence between the Drinfeld presentation and the RTT presentation in Section \ref{se:Maintheorem}. In the last Section \ref{se:qCenter}, we transform the central series $z(u)$ obtained by Nazarov in terms of Drinfeld generators. 
\medskip

{\bf Notations and terminologies:}
Throughout the paper,  the symbols $\mathbb{Z}$ and $\mathbb{N}$ represent the sets of integers and non-negative integers, respectively. All vector spaces and algebras are assumed to be defined over the field $\mathbb{C}$ of complex numbers. The Kronecker delta $\delta_{ij}$ is equal to $1$ if $i=j$ and $0$ otherwise.
We write $\mathbb{Z}_2=\left\{\bar{0},\bar{1}\right\}$. For a homogeneous element $x$ of an associative or Lie superalgebra, we use $|x|$ to denote the degree of $x$ with respect to the $\mathbb{Z}_2$-grading. When we write $|x|$ for an element $x$,  we always assume that $x$ is a homogeneous element and automatically extend the relevant formulas by linearity (when applicable). All modules of Lie superalgebras are assumed to be $\mathbb{Z}_2$-graded. The tensor product of two superalgebras $A$ and $B$ carries a superalgebra structure given by 
$$(a_1 \otimes b_1)\cdot(a_2\otimes b_2) =(-1)^{|a_2||b_1|}a_1a_2\otimes b_1b_2,$$
for homogeneous elements $a_i\in A, b_i\in B$ with $i=1,2$. Any $N\times N$-matrix $X=[X_{ij}]$ with entries in an associative (super)algebra $\mathcal{A}$ will be regarded as the element 
$$X=\sum\limits_{i,j=1}\mathsf{E}_{ij}\otimes X_{ij}\in\mathrm{End}(\mathbb{C}^N)\otimes\mathcal{A},$$
where $\mathsf{E}_{ij}$ denotes the standard matrix unit. We will need tensor product superalgebras of the form
$\mathrm{End}(\mathbb{C}^N)^{\otimes m}\otimes \mathcal{A}$. For any $1\leqslant a\leqslant m$, we will denote by $X^a$ the element associated with $a$-th  copy of $\mathrm{End}(\mathbb{C}^N)$ so that 
$$X^a=\sum\limits_{i,j=1}^N1^{\otimes (a-1)}\otimes \mathsf{E}_{ij}\otimes 1^{\otimes (m-1)}\otimes X_{ij}\in\mathrm{End}(\mathbb{C}^N)^{\otimes m}\otimes \mathcal{A}.$$ 

\vskip 0.3cm
{\bf  Acknowledgment}.  The authors are partially supported by the NSF of China (Grant 12471025). The second author is also supported by the Anhui Provincial Natural Science Foundation 2308085MA01, and the Fundamental Research Funds for the Central Universities JZ2025HGTG0251. 

\section{The Queer Super-Yangian}
\label{se:Yqdef}

In this section, we will review some fundamental facts for the twisted queer current Lie superalgebra $\widehat{\mathfrak{q}}^{tw}_n[u]$ and the queer super-Yangian $\mathrm{Y}(\mathfrak{q}_n)$ to fix our notations. We refer to \cite{N99} for more details.   

Let $n$ be a positive integer, and $I_{n|n}$ be the set $\{-n,\ldots,-1,1,\ldots,n\}$, whose element $i$ is of parity $|i|:=\bar{0}$ if $i>0$ and $\bar{1}$ if $i<0$, respectively. Let $\mathbb{C}^{n|n}$ be the $\mathbb{Z}_2$-graded vector space with the standard basis consisting of $v_i$ of parity $|i|$ for $i\in I_{n|n}$. As usual, the algebras $\mathrm{End}(\mathbb{C}^{n|n})$ and $\mathfrak{gl}_{n|n}$ denote the associative superalgebra and the Lie superalgebra of all $2n\times 2n$-matrices, respectively, in which the $(i,j)$-matrix unit $\mathsf{E}_{ij}$ is of parity $|i|+|j|$. \textit{The queer Lie superalgebra $\mathfrak{q}_n$} is the Lie sub-superalgebra of $\mathfrak{gl}_{n|n}$ consisting of elements that are fixed by the involutive automorphism $\eta(\mathsf{E}_{ij})=\mathsf{E}_{-i,-j}$. It is generated by
$$\mathsf{g}_{ij}:=\mathsf{E}_{ij}+\mathsf{E}_{-i,-j}\text{ for }i,j\in I_{n|n},$$
which satisfy the relations $\mathsf{g}_{ij}=\mathsf{g}_{-i,-j}$, and
\begin{equation}
[\mathsf{g}_{ij},\mathsf{g}_{kl}]=\delta_{kj}\mathsf{g}_{il}+\delta_{j,-k}\mathsf{g}_{i,-l}-(-1)^{(|i|+|j|)(|k|+|l|)}\delta_{li}\mathsf{g}_{kj}-(-1)^{(|i|+|j|)(|k|+|l|)}\delta_{l,-i}\mathsf{g}_{k,-j},
\label{eq:qncommutator}
\end{equation}
for $i,j\in I_{n|n}$. The relations in \eqref{eq:qncommutator} are also equivalently written as a single ternary relation in $\mathrm{End}\left(\mathbb{C}^{n|n}\right)\otimes\mathrm{End}\left(\mathbb{C}^{n|n}\right)\otimes\mathrm{U}(\mathfrak{q}_n)$:
\begin{align*}
    &\left[\mathsf{G}^1,\mathsf{G}^2\right]=\mathsf{G}^2(\mathsf{P}+\mathsf{Q})-(\mathsf{P}+\mathsf{Q})\mathsf{G}^2,
\end{align*}
 where $\mathsf{G}=\sum\limits_{i,j\in I_{n|n}}(-1)^{|j|}\mathsf{E}_{ij}\otimes\mathsf{g}_{ji}$, $\mathsf{P}=\mathop{\sum}\limits_{i,j}(-1)^{|j|}\mathsf{E}_{ij}\otimes \mathsf{E}_{ji}$, and $\mathsf{Q}=\mathop{\sum}\limits_{i,j}(-1)^{|j|}\mathsf{E}_{ij}\otimes \mathsf{E}_{-j,-i}$.
 
The involution $\eta$ can be extended to the current Lie superalgebra $\mathfrak{gl}_{n|n}[\xi]$ in the usual sense. Then the \textit{twisted queer current Lie superalgebra} is defined to be
$$\widehat{\mathfrak{q}}^{tw}_n=\{X(\xi)\in\mathfrak{gl}_{n|n}[\xi]|\ \eta(X(\xi))=X(-\xi)\}.$$
The Lie superalgebra $\widehat{\mathfrak{q}}_n^{tw}$ is generated by elements
$\mathsf{g}_{ij}^{(r)}:=\mathsf{E}_{ij} \xi^r+(-1)^r\mathsf{E}_{-i,-j} \xi^{r}$ of parity $|i|+|j|$ for $i,j\in I_{n|n}$ and $r\in\mathbb{N}$. They satisfy the relations $\mathsf{g}_{ij}^{(r)}=(-1)^r\mathsf{g}_{-i,-j}^{(r)}$ and 
\begin{align*}
\left[\mathsf{g}_{ij}^{(r)}, \mathsf{g}_{kl}^{(s)}\right]=
&\left(\delta_{jk}\mathsf{g}_{il}^{(r+s)}+\delta_{j,-k}\mathsf{g}_{-i,l}^{(r+s)}(-1)^r\right)-\left(\delta_{il}\mathsf{g}_{kj}^{(r+s)}+\delta_{-i,l}\mathsf{g}_{k,-j}^{(r+s)}(-1)^r\right)(-1)^{(|i|+|j|)(|k|+|l|)},
\end{align*}
for $i,j,k,l\in I_{n|n}$ and $r,s\in\mathbb{N}$. 
Moreover, we write 
$$\mathsf{g}_{ij}(u)=\delta_{ij}+(-1)^{|j|}\sum\limits_{r\geqslant0}\mathsf{g}_{ij}^{(r)}u^{-r-1}\in\widehat{\mathfrak{q}}_n^{tw}[[u^{-1}]]$$ 
for $i,j\in I_{n|n}$, and $\mathsf{G}(u)=\sum\limits_{i,j\in I_{n|n}}(-1)^{|i|}\mathsf{E}_{ji}\otimes\mathsf{g}_{ij}(u)$. Then $\mathsf{g}_{-i,-j}(u)=\mathsf{g}_{ij}(-u)$ for $i,j\in I_{n|n}$ and
$$\left[\mathsf{G}^1(u),\mathsf{G}^2(v)\right]=
\mathsf{G}^2(v)\left(\frac{\mathsf{P}}{u-v}+\frac{\mathsf{Q}}{u+v}\right)-\left(\frac{\mathsf{P}}{u-v}+\frac{\mathsf{Q}}{u+v}\right)\mathsf{G}^2(v).$$
holds as an equality of formal Laurent series in $u^{-1}, v^{-1}$ with coefficients in $\mathrm{End}\left(\mathbb{C}^{n|n}\right)\otimes\mathrm{End}\left(\mathbb{C}^{n|n}\right)\otimes\mathrm{U}\left(\widehat{\mathfrak{q}}_n^{tw}\right)$ after multiplying each side by $u^2-v^2$.

\textit{The queer super-Yangian $\mathrm{Y}(\mathfrak{q}_n)$} is a $\mathbb{Z}_2$-graded unital associative superalgebra generated by $t_{ij}^{(r)}$ of parity $|i|+|j|$ with $i,j\in I_{n|n}$ and $r=1,2,\ldots$, subject to certain relations. In order to explicitly write them down, we assemble the generators into formal series
$$t_{ij}(u):=\sum\limits_{r\geqslant0}t_{ij}^{(r)}u^{-r}\in \mathrm{Y}(\mathfrak{q}_n)[[u^{-1}]], $$
where $t_{ij}^{(0)}:=\delta_{ij}$. They are further encapsulated into a $2n\times2n$-matrix
$$T(u)=\sum\limits_{i,j\in I_{n|n}}\mathsf{E}_{ij}\otimes t_{ij}(u).$$ 
Let
\begin{equation*}
R(u,v):=1
-\frac{1}{u-v}\sum\limits_{i,j\in I_{n|n}}(-1)^{|j|}\mathsf{E}_{ij}\otimes \mathsf{E}_{ji}
-\frac{1}{u+v}\sum\limits_{i,j\in I_{n|n}}(-1)^{|j|}\mathsf{E}_{ij}\otimes \mathsf{E}_{-j,-i}.
\end{equation*}
Then the defining relations of $\mathrm{Y}(\mathfrak{q}_n)$ can be written as
\begin{eqnarray}
t_{i,j}(-u)&=&t_{-i,-j}(u),\qquad\quad\text{ for }i,j\in I_{n|n},\label{eq:tijsym}\\
R(u,v)T^1(u)T^2(v)&=&T^2(v)T^1(u)R(u,v),\label{eq:RTT}
\end{eqnarray}
where \eqref{eq:RTT} holds in $\mathrm{End}(\mathbb{C}^{n|n})^{\otimes 2}\otimes \mathrm{Y}(\mathfrak{q}_n)[[u^{-1}, v^{-1}]]$. The relation \eqref{eq:RTT} is also equivalent to
\begin{equation}
\begin{aligned}
\theta(i,k,l)[t_{ij}(u),t_{kl}(v)]
=&\frac{1}{u-v}(t_{kj}(u)t_{il}(v)-t_{kj}(v)t_{il}(u))\\
&-\frac{(-1)^{|k|+|l|}}{u+v}(t_{-k,j}(u)t_{-i,l}(v)-t_{k,-j}(v)t_{i,-l}(u)),
\end{aligned}
\label{eq:RTTexp}
\end{equation}
where $\theta(i,k,l)=(-1)^{|i||k|+|k||l|+|l||i|}$. 

We also denote the inverse matrix of $T(u)$ by $\widetilde{T}(u)=\sum\limits_{i,j\in I_{n|n}}\mathsf{E}_{ij}\otimes \widetilde{t}_{ij}(u)$, which satisfies
\begin{equation}
T^1(u)R^{12}(u,v)\widetilde{T}^2(v)=\widetilde{T}^2(v)R^{12}(u,v)T^1(u).\label{eq:TRT}
\end{equation}
It is equivalent to 
 \begin{align*}
 &\theta(i,j,l)\left[t_{ij}(u),\widetilde{t}_{kl}(v)\right]\\
 =&\frac{1}{u-v}\left(\delta_{jk}\sum\limits_{p\in I_{n|n}}(-1)^{|p|}\theta(p,i,l)t_{ip}(u)\tilde{t}_{pl}(v)-\delta_{il}\sum\limits_{p\in I_{n|n}}(-1)^{|p|}\theta(p,j,k)\widetilde{t}_{kp}(v)t_{pj}(u)\right)\\
 &+\frac{1}{u+v}\left((-1)^{|j|}\delta_{j,-k}\sum\limits_{p\in I_{n|n}}\theta(p,i,l)t_{ip}(u)\widetilde{t}_{-p,l}(v)-(-1)^{|i|}\delta_{i,-l}\sum\limits_{p\in I_{n|n}}\theta(p,j,k)\widetilde{t}_{k,-p}(v)t_{pj}(u)\right).
 \end{align*}
 \begin{lemma}\cite[Relations (2.12) and (2.13)]{N99}\label{lem:embeddingandevaluation}
 The assignment $$\mathrm{ev}:\ t_{ij}(u)\mapsto\delta_{ij}-\mathsf{g}_{ji}u^{-1}(-1)^{|j|}, i,j\in I_{n|n}$$
 defines a surjective homomorphism $\Y(\mathfrak{q}_n)\twoheadrightarrow\mathrm{U}(\mathfrak{q}_n)$, and 
 $$\iota:\ \mathbf{g}_{ji}\mapsto-t_{ij}^{(1)}(-1)^{|j|}, i,j\in I_{n|n}$$
 defines an embedding $\mathrm{U}(\mathfrak{q}_n)\hookrightarrow\Y(\mathfrak{q}_n)$.
 \end{lemma}

\begin{lemma}\label{lem:antiinvolution}
The assignment $t_{ij}(u)\mapsto t_{ji}(u), i,j\in I_{n|n}$ defines an anti-involution $\omega$ of the associative algebra $\Y(\mathfrak{q}_n)$, that is a $\mathbb{C}$-linear map such that $\omega(ab)=\omega(b)\omega(a)$ for $a,b\in\Y(\mathfrak{q}_n)$.
\end{lemma}
\begin{proof}
$\omega$ is well-defined since it preserves the relations \eqref{eq:tijsym} and \eqref{eq:RTTexp}.
\end{proof}

There are two different ascending $\mathbb{Z}$-filtrations of $\mathrm{Y}(\mathfrak{q}_n)$ as defined in \cite{N99}. The first assigns degree $r$ to each  generator $t_{ij}^{(r)}$ for $r\geqslant 1$, yielding the associated graded algebra denoted by $\mathrm{gr}\mathrm{Y}(\mathfrak{q}_n)$. This algebra is supercommutative with free generators $t_{ij}^{(r)}$ and $t_{-i,j}^{(r)}$ (using the same notation without ambiguity), where $i,j= 1,\ldots,n$ and $r\in\mathbb{N}$. It leads to a PBW theorem for $\Y(\mathfrak{q}_n)$ as demonstrated in \cite[Corollary 2.4]{N99}.

The second filtration is defined by declaring that the generator $t_{ij}^{(r)}$  is of degree $r-1$ for each $r\geqslant 1$. The corresponding graded algebra, denoted by $\mathrm{gr}^{\prime}\mathrm{Y}(\mathfrak{q}_n)$,  is isomorphic to the universal enveloping algebra $\mathrm{U}(\widehat{\mathfrak{q}}^{tw}_n)$ of the twisted current Lie superalgebra $\widehat{\mathfrak{q}}_n^{tw}$, see \cite[Proposition 2.1]{N99}. Let $\pi: \Y(\mathfrak{q}_n)\rightarrow\mathrm{gr}^{\prime}\Y(\mathfrak{q}_n)$ be the canonical homomorphism. The isomorphism is explicitly given in the following theorem:

\begin{theorem} \cite[Theorem 2.3 ]{N99}\label{thm:Iso}
There is a superalgebra isomorphism from  $\mathrm{gr}^{\prime}\mathrm{Y}(\mathfrak{q}_n)$ to $\mathrm{U}(\widehat{\mathfrak{q}}^{tw}_n)$ defined by 
 $$\pi\left(t_{ij}^{(r+1)}\right)\mapsto -(-1)^{|j|}\mathsf{g}_{ji}^{(r)},$$
 for $i,j\in I_{n|n}$ and $r\in\mathbb{N}$.
\end{theorem}

\section{A Block Gauss decomposition}\label{se:Gaussdec}

In this section, we will find a family of generators for the super-Yangian $\Y(\mathfrak{q}_n)$ through the Gaussian factorization of the generator matrix $T(u)$ in an appropriate block form. 

We reorder the generator matrix $T(u)$ of $\Y(\mathfrak{q}_n)$ such that its rows and columns are labeled by $1,-1,2,-2,\ldots,n,-n$. Then $T(u)$ can be partitioned into an $n\times n$ block matrix such that each block $T_{ab}(u)$ is a $2\times 2$-matrix with entries in $\mathrm{Y}(\mathfrak{q}_n)[[u^{-1}]]$, that is, 
\begin{align}\label{eq:qnTu}
T(u)=
\begin{pmatrix}
T_{11}(u)&\ldots &T_{1n}(u)\\
\vdots& &\vdots\\
T_{n1}(u)&\ldots & T_{nn}(u)
\end{pmatrix}
=\sum\limits_{a,b=1}^n\mathsf{E}_{ab}\otimes T_{ab}(u), 
\end{align}
where each block
\begin{align*}
T_{ab}(u)
=\sum\limits_{i,j\in I_{1|1}}\varepsilon_{ij}\otimes t_{ia,jb}(u),
\end{align*}
is a $2\times 2$-matrix satisfying $t_{-a,-b}(u)=t_{ab}(-u)$ and $t_{a,-b}(u)=t_{-a,b}(-u)$ for $a,b=1,2,\ldots,n$. Such $2\times2$-matrices possess the following nice properties:
\begin{lemma}\label{le:q1}
Let $\mathcal{A}=\mathcal{A}_{\bar{0}}\oplus\mathcal{A}_{\bar{1}}$ be an arbitrary associative superalgebra with unit and $\mathcal{A}[[u^{-1}]]$ be the superalgebra of formal series with coefficients in $\mathcal{A}$. We say that a $2\times 2$-matrix 
$$X(u)=\sum\limits_{i,j\in I_{1|1}}\varepsilon_{ij}\otimes x_{ij}(u)$$ 
with entries in $\mathcal{A}[[u^{-1}]]$ is of YQ form if $x_{-i,-j}(u)=x_{ij}(-u)\in\mathcal{A}_{\bar{0}}[[u^{-1}]]$ for $i,j=-1,1$. Then these matrices fulfill the following properties:
\begin{enumerate}
\item If $X(u)$ and $Y(u)$ are $2\times2$-matrices of YQ form, so is $X(u)Y(u)$.
\item Let $X(u)$ be a $2\times2$-matrix of YQ form. If $x_{11}(u)$ is invertible and $x_{-1,1}(u)\in u^{-1}\mathcal{A}[[u^{-1}]]$, then $X(u)$ is an invertible matrix, whose inverse is also of YQ form. 
\end{enumerate}
\end{lemma}

\begin{proof}
The proof of the lemma is straightforward. We omit it here.
\end{proof}

Now, the generator matrix $T(u)$ of $\Y(\mathfrak{q}_n)$ is regarded as an $n\times n$-block matrix, in which each block $T_{ab}(u)$ is a $2\times 2$-matrix of YQ form with entries in $\Y(\mathfrak{q}_n)[[u^{-1}]]$. By \cite{GR97}, $T(u)$ admits the following block Gauss decomposition:
\begin{equation}
T(u)=F(u)H(u)E(u),\label{eq:Gaussfac}
\end{equation}
where $F(u)$ (resp. $H(u)$ and $E(u)$) is a lower triangular unipotent (resp. diagonal and upper triangular unipotent) $n\times n$-block matrix such that all blocks are $2\times 2$-matrices in $\mathrm{Y}(\mathfrak{q}_n)[[u^{-1}]]$. Namely, if we denote the $(a,b)$-block of $F(u)$ and $E(u)$ by $F_{ab}(u)$ and $E_{ab}(u)$, respectively, then
$$F_{ab}(u)=E_{ba}(u)=0,\text{if }a<b, \text{ and }E_{aa}(u)=F_{aa}(u)=1,a=1,2,\ldots,n.$$
We also denote $H(u)=\mathrm{diag}\left(H_1(u), H_2(u),\ldots,H_n(u)\right)$. 

The factorization \eqref{eq:Gaussfac} is uniquely determined by $T(u)$. In fact, $F(u), H(u)$ and $E(u)$ can be described using quasi-determinants \cite{GR97}. Suppose that $A, B, C$ and $D$ are $m\times m, m\times n, n\times m$ and $n\times n$ matrices with entries in some ring, respectively, and $A$ is invertible, then the quasi-determinant is 
\begin{equation}
\label{eq:qdet}
\begin{vmatrix}A&B\\C&\boxed{D}\end{vmatrix}:=D-CA^{-1}B.
\end{equation}
For $1\leqslant i,j\leqslant n$ and $r<i,j$, we denote 
$$\Delta_{r,a,b}(T(u))
:=\begin{vmatrix}
	T_{11}(u)&T_{12}(u)&\cdots&T_{1r}(u)&T_{1b}(u)\\
	T_{21}(u)&T_{22}(u)&\cdots&T_{2r}(u)&T_{2b}(u)\\
	\vdots&\vdots&&\vdots\\
	T_{r1}(u)&T_{r2}(u)&\cdots&T_{rr}(u)&T_{rb}(u)\\
	T_{a1}(u)&T_{a2}(u)&\cdots&T_{ar}(u)&\boxed{T_{ab}(u)}
	\end{vmatrix},$$
and set $\Delta_{0,1,1}(T(u))=H_1(u)$.
Then 
\begin{align}\label{re:generatorsHEF}
\begin{cases}
H_a(u)=\Delta_{a-1,a,a}(T(u)), &a=1,\ldots,n,\\
E_{ba}(u)=H_b(u)^{-1}\Delta_{b-1,b,a}(T(u)),&1\leqslant b<a\leqslant n,\\
F_{ab}(u)=\Delta_{b-1,a,b}(T(u))H_b(u)^{-1},&1\leqslant b<a\leqslant n.
\end{cases}
\end{align}

By Lemma~\ref{le:q1}, all blocks $H_a(u), a=1,2,\ldots,n$, $E_{ab}(u)$, $F_{ba}(u)$ for $1\leqslant a<b\leqslant n$ are $2\times2$-matrices of YQ form with entries in $\Y(\mathfrak{q}_n)[[u^{-1}]]$. We simply write $E_b(u)=E_{b,b+1}(u)$ and $F_b(u)=F_{b+1,b}(u)$ for $b=1,2,\ldots,n-1$. Then the entries of matrices $H_a(u)$, $E_b(u)$ and $F_b(u)$ can be written as follows:
\begin{align*}
H_a(u)=&\varepsilon_{11}\otimes h_a(u)+\varepsilon_{-1,-1}\otimes h_a(-u)
+\varepsilon_{-1,1}\otimes \bar{h}_a(u)+\varepsilon_{1,-1}\otimes\bar{h}_a(-u),\\
E_b(u)=&\varepsilon_{11}\otimes e_b(u)+\varepsilon_{-1,-1}\otimes e_b(-u)
+\varepsilon_{-1,1}\otimes \bar{e}_b(u)+\varepsilon_{1,-1}\otimes\bar{e}_b(-u),\\
F_b(u)=&\varepsilon_{11}\otimes f_b(u)+\varepsilon_{-1,-1}\otimes f_b(-u)
+\varepsilon_{-1,1}\otimes \bar{f}_b(u)+\varepsilon_{1,-1}\otimes\bar{f}_b(-u),
\end{align*}
for $a=1,2,\ldots,n$ and $b=1,2,\ldots,n-1$. Then we can prove the following theorem:

\begin{theorem}
	\label{thm:generator}
	The associative superalgebra $\Y(\mathfrak{q}_n)$ is generated by all coefficients of the formal series $h_a(u)$, $\bar{h}_a(u)$, $e_b(u)$, $\bar{e}_b(u)$, $f_b(u)$ and $\bar{f}_b(u)$ for $a=1,2,\ldots,n$ and $b=1,2,\ldots,n-1$.
\end{theorem}

\begin{proof}
By the Gauss decomposition \eqref{eq:Gaussfac},
\begin{align}
T_{aa}(u)=&H_a(u)+\sum_{p=1}^{a-1}F_{ap}(u)H_p(u)E_{pa}(u),&&a=1,2,\ldots,n,\label{re:TaaHFHE}\\
T_{ab}(u)=&H_a(u)E_{ab}(u)+\sum_{p=1}^{a-1}F_{ap}(u)H_p(u)E_{pb}(u),&&1\leqslant a<b\leqslant n,\label{re:TabHFHE}\\
T_{ba}(u)=&F_{ba}(u)H_a(u)+\sum_{p=1}^{a-1}F_{bp}(u)H_p(u)E_{pa}(u),&&1\leqslant a<b\leqslant n.\label{re:TbaHFHE}
\end{align}
Recall that $\pi:\Y(\mathfrak{q}_n)\rightarrow\mathrm{gr}^{\prime}\Y(\mathfrak{q}_n)$ is the canonical homomorphism given in Section~\ref{se:Yqdef}. The above identities show that $\pi\left(T_{aa}^{(r)}\right)=\pi\left(H_a^{(r)}\right)$ for $a=1,2,\ldots,n$, $r\geqslant1$,
$$\pi\left(T_{ab}^{(r)}\right)=\pi\left(E_{ab}^{(r)}\right), \text{ and }\pi\left(T_{ba}^{(r)}\right)=\pi\left(F_{ba}^{(r)}\right),$$
for $1\leqslant a<b\leqslant n, r\geqslant1$.

On the other hand, the block RTT relation \eqref{re:BTT} imply that
$$\left[\pi\left(T_{ab}^{(r+1),1}\right),\pi\left(T_{bc}^{(s+1),2}\right)\right]
=\left(P+(-1)^rQ\right)\pi\left(T_{ac}^{(r+s+1),2}\right),$$
for distinct $1\leqslant a,b,c\leqslant n$. Hence, $\Y(\mathfrak{q}_n)$ is generated by the coefficients of all entries in the matrices $H_a(u)$ for $a=1,2,\ldots,n$, $E_b(u)$ and $F_b(u)$ for $b=1,2,\ldots,n-1$. 
\end{proof}

\begin{proposition}\label{antiinvolutionomega}
    Let $\omega$ be the anti-involution given in Lemma~\ref{lem:antiinvolution}. Then 
    \begin{align}
     \omega\left(h_{c}(u)\right)=&h_{c}(u),& 
     \omega\left(e_{ab}(u)\right)=&f_{ba}(u),&
     \omega\left(f_{ba}(u)\right)=&e_{ab}(u),\label{re:inveven}\\
    \omega\left(\bar{h}_{c}(u)\right)=&\bar{h}_{c}(-u),&
    \omega\left(\bar{e}_{ab}(u)\right)=&\bar{f}_{ba}(-u),&
    \omega\left(\bar{f}_{ba}(u)\right)=&\bar{e}_{ab}(-u),\label{re:invodd}
    \end{align}
    for $1\leqslant a<b\leqslant n$ and $1\leqslant c\leqslant n$.
\end{proposition}
\begin{proof}
For a $2\times 2$-matrix $X(u)=\sum\limits_{i,j\in I_{1|1}}\varepsilon_{ij}\otimes x_{ij}(u)$ with $x_{ij}(u)\in\Y(\mathfrak{q}_n)[[u^{-1}]]$, we denote $$\omega(X(u))=\sum\limits_{i,j\in I_{1|1}}\varepsilon_{ji}\otimes \omega\left(x_{ij}(u)\right).$$
Then $\omega\left(T_{ab}(u)\right)=T_{ba}(u)$ for $a,b=1,2,\ldots,n$. It suffices to show that
\begin{equation}
\omega\left(H_c(u)\right)=H_c(u),\quad 
\omega\left(E_{ab}(u)\right)=F_{ba}(u),\quad
\omega\left(F_{ba}(u)\right)=E_{ab}(u),
\label{eq:invmat}
\end{equation}
for $1\leqslant a<b\leqslant n$ and $1\leqslant c\leqslant n$.

Since $\omega$ is an anti-automorphism of the associative algebra $\Y(\mathfrak{q}_n)$, we directly verify that $$\omega\left(X(u)Y(u)\right)=\omega\left(Y(u)\right)\omega\left(X(u)\right)$$ for any $2\times 2$-matrices $X(u), Y(u)$ of YQ form with entries in $\Y(\mathfrak{q}_n)$.

The block Gauss decomposition \eqref{eq:Gaussfac} implies that 
$$T_{11}(u)=H_1(u), T_{1b}(u)=H_1(u)E_{1b}(u),\qquad T_{b1}(u)=F_{b1}(u)H_1(u),$$
for $b=2,3,\ldots,n$. It follows that $\omega\left(H_1(u)\right)=\omega\left(T_{11}(u)\right)=T_{11}(u)=H_1(u)$. Then we deduce that
\begin{align*}
T_{b1}(u)=\omega\left(T_{1b}(u)\right)=\omega\left(E_{1b}(u)\right)\omega\left(H_1(u)\right)=\omega\left(E_{1b}(u)\right)H_1(u),
\end{align*}
which shows $\omega\left(E_{1b}(u)\right)=F_{b1}(u)$. Similarly, we have $\omega\left(F_{b1}(u)\right)=E_{1b}(u)$.

Then \eqref{eq:invmat} can be obtained inductively through \eqref{re:TaaHFHE}, \eqref{re:TabHFHE}, and \eqref{re:TbaHFHE}.
\end{proof}
\bigskip

In order to effectively calculate with matrix blocks, we reformulate the RTT relation~\ref{eq:RTT} in a block form. Indeed, the block matrix \eqref{eq:qnTu} can be understood via the identification $\mathrm{End}\left(\mathbb{C}^{n|n}\right)\cong\mathrm{End}\left(\mathbb{C}^n\right)\otimes\mathrm{End}\left(\mathbb{C}^{1|1}\right)$ by
$$\mathsf{E}_{i,j}\mapsto \mathsf{E}_{ij}\otimes \varepsilon_{11},\quad \mathsf{E}_{-i, j}\mapsto \mathsf{E}_{ij}\otimes \varepsilon_{-1,1},
\quad \mathsf{E}_{i,-j}\mapsto \mathsf{E}_{ij}\otimes \varepsilon_{1,-1},\quad \mathsf{E}_{-i,-j}\mapsto \mathsf{E}_{ij}\otimes \varepsilon_{-1,-1},$$
where $\mathsf{E}_{ij}$ is the matrix unit of $\mathrm{Mat}_{n|n}(\mathbb{C})$ or $\mathrm{Mat}_{n}(\mathbb{C})$ in an obvious way, and $\varepsilon_{kl}$, for $k,l\in\{\pm 1\}$, is the matrix unit of $\mathrm{Mat}_{1|1}(\mathbb{C})$. Then the associative superalgebra $\mathrm{End}\left(\mathbb{C}^{n|n}\right)\otimes\mathrm{End}\left(\mathbb{C}^{n|n}\right)$ could be identified with 
$$\mathrm{End}\left(\mathbb{C}^n\right)\otimes \mathrm{End}\left(\mathbb{C}^n\right)\otimes
\mathrm{End}\left(\mathbb{C}^{1|1}\right)\otimes \mathrm{End}\left(\mathbb{C}^{1|1}\right).$$
Under this identification, the $R(u,v)$-matrix can be rewritten as  
$$R(u,v)=1-\sum\limits_{i,j=1}^n\mathsf{E}_{ij}\otimes \mathsf{E}_{ji}\otimes K(u,v),$$
where 
\begin{align}\label{re:Kuv}
K(u,v)=&\frac{1}{u-v}\sum_{r,s=-1}^1(-1)^{|s|}\varepsilon_{rs}\otimes \varepsilon_{sr}
+\frac{1}{u+v}\sum\limits_{r,s=-1}^1(-1)^{|s|}\varepsilon_{rs}\otimes \varepsilon_{-s,-r}\nonumber\\
=&\frac{P}{u-v}+\frac{Q}{u+v}=\frac{P}{u-v}-\frac{PJ^1J^2}{u+v},
\end{align}
where $J^1=\sum\limits_{i=-1}^1(-1)^{|i|}E_{i,-i}\otimes 1$, similar for $J^2$. We remark that $P^{12}J^1=J^2P^{12}$, but $P^{12}J^2\ne J^1P^{12}$.  We remark that the element $Q$ is $-Q$ in Nazarov's paper \cite{N22}. 
\begin{lemma}\cite[Section IV]{N22}\label{re:PQrelation}
The following identities hold in $\mathrm{End}(\mathbb{C}^{1|1})^{\otimes 3}$:
\begin{align*}
&P^{12}P^{13}=P^{23}P^{12}=P^{13}P^{23},&&P^{12}P^{23}=P^{13}P^{12}=P^{23}P^{13},\\
&P^{12}Q^{13}=Q^{23}P^{12}=Q^{13}Q^{23},&&P^{12}Q^{23}=Q^{13}P^{12}=Q^{23}Q^{13},\\
&Q^{12}Q^{13}=-P^{23}Q^{12}=-Q^{13}P^{23},&&Q^{12}P^{23}=-Q^{13}Q^{12}=P^{23}Q^{13},\\
&Q^{12}P^{13}=-Q^{23}Q^{12}=-P^{13}Q^{23},&&Q^{12}Q^{23}=-P^{13}Q^{12}=Q^{23}P^{13},\\
&P^{12}P^{13}P^{23}=P^{23}P^{13}P^{12}, &&P^{12}Q^{13}Q^{23}=Q^{23}Q^{13}P^{12},\\
&Q^{12}Q^{13}P^{23}=P^{23}Q^{13}Q^{12}, &&Q^{12}P^{13}Q^{23}=Q^{23}P^{13}Q^{12},\\
&P^{12}P^{13}Q^{23}=P^{23}Q^{13}P^{12}=Q^{12}P^{13}P^{23}=Q^{23}Q^{13}Q^{12},\\
&Q^{23}P^{13}P^{12}=P^{12}Q^{13}P^{23}=P^{23}P^{13}Q^{12}=Q^{12}Q^{13}Q^{23}.
\end{align*}
\end{lemma}

\begin{lemma}
\label{lem:PQtransform}
Let $a\geqslant 2$ be a positive integer and $X(u)$ be a matrix of YQ form with entries in $\mathcal{A}[[u]]$. Then the following identities hold in $\mathrm{End}(\mathbb{C}^{1|1})^{\otimes a}\otimes \Y(\mathfrak{q}_2)[[u^{-1}]]$:
    \begin{align}
        &P^{ij}X^{i}(u)=X^j(u)P^{ij},\quad Q^{ij}X^i(u)=X^j(-u)Q^{ij},\label{re:PQtransform1}\\
        &P^{ij}X^{k}(u)=X^k(u)P^{ij},\quad Q^{ij}X^k(u)=X^k(u)Q^{ij},\label{re:PQtransform2}
    \end{align}
    where $1\leqslant i, j, k\leqslant a$ with $k\ne i, j$.
\end{lemma}

Then the RTT relation \eqref{eq:RTT} is equivalently written as
\begin{equation}\label{re:BTT}
\left[T_{ab}^{1}(u), T_{cd}^{2}(v)\right]
=K(u,v)T_{cb}^{1}(u)T_{ad}^{2}(v)-T_{cb}^{2}(v)T_{ad}^{1}(u)K(u,v),
\end{equation}
for $a,b,c,d=1,2,\ldots,n$.

Similarly, the inverse matrix $\widetilde{T}(u)$ of $T(u)$ is also regarded as an $n\times n$-block matrix 
\begin{equation}
\widetilde{T}(u)=\sum\limits_{a,b=1}^n\mathsf{E}_{ab}\otimes \widetilde{T}_{ab}(u),
\label{eq:Ttildeblock}
\end{equation}
in which each block $\widetilde{T}_{ab}(u)$ is a $2\times 2$-matrix of YQ form with entries in $\Y(\mathfrak{q}_n)[[u^{-1}]]$ by Lemma~\ref{le:q1}. The identity \eqref{eq:TRT} is also equivalently written as 
\begin{align}\label{re:BTWT}
\left[T_{ab}^{1}(u), \tilde{T}_{cd}^{2}(v)\right]
=&\delta_{bc}\sum\limits_{p}T_{ap}^{1}(u)K(u,v)\tilde{T}_{pd}^{2}(v)-\delta_{ad}\sum\limits_{p}\tilde{T}_{cp}^{2}(v)K(u,v)T^{1}_{pb}(u).
\end{align}

\section{An Embedding Theorem}\label{se:Embedding}
In order to investigate the Drinfeld presentation of super-Yangain $\mathrm{Y}(\mathfrak{q}_n)$, we need to establish an embedding theorem of super-Yangian $\mathrm{Y}(\mathfrak{q}_n)$ that reduces the calculation of relations among Drinfeld generators in $\Y(\mathfrak{q}_n)$ to low-rank cases.  

For each positive integer $m$, there is a homomorphism of associative superalgebras 
$$\varphi_m: \Y(\mathfrak{q}_m)\rightarrow\Y(\mathfrak{q}_{m+n}),\quad T_{ab}(u)\mapsto T_{ab}(u),\quad a,b=1,2,\ldots,m.$$
One can define another family of embedding homomorphisms as in \cite{BK05}. Let
\begin{align*}
&H^{[m]}(u)=\mathrm{diag}\left(H_{m+1}(u),\ldots,H_{m+n}(u)\right),\\
&E^{[m]}(u)=\left(E_{ab}(u)\right)_{m+1\leqslant  a,b\leqslant m+n},\\
&F^{[m]}(u)=\left(F_{ab}(u)\right)_{m+1\leqslant a,b\leqslant m+n},
\end{align*}
be the lower-right $n\times n$-submatrix of $H(u)$, $E(u)$, and $F(u)$, respectively. Although the Gauss decomposition asserts $T(u)=F(u)H(u)E(u)$, the product $F^{[m]}(u)H^{[m]}(u)E^{[m]}(u)$ of lower-right submatrices is not equal to the lower-right submatrix $T^{[m]}(u)$ of $T(u)$. Nevertheless, it gives the following embedding homomorphism:

\begin{theorem}\label{thm:embedding}
For each positive integer $m$, there is a homomorphism of associative superalgebras
$$\psi_m:\Y(\mathfrak{q}_{n})\rightarrow\Y(\mathfrak{q}_{m+n}),\qquad 
T(u)\mapsto F^{[m]}(u)H^{[m]}(u)E^{[m]}(u).$$
\end{theorem} 
\begin{proof}
It suffices to show that the matrix $F^{[m]}(u)H^{[m]}(u)E^{[m]}(u)$ satisfies the RTT relation associated with the $R$-matrix of $\Y(\mathfrak{q}_{n})$.

The inverse $\widetilde{T}(u)$ of $T(u)$ can be written as the block matrix
$$\widetilde{T}(u)=\begin{pmatrix}\widetilde{A}(u)&\widetilde{B}(u)\\ \widetilde{C}(u)&\widetilde{D}(u)\end{pmatrix},$$
such that $\tilde{D}(u)$ is a $2n\times 2n$-matrix. 
Then it follows from the RTT relation \eqref{eq:RTT} that
$$\widetilde{T}^1(u)\widetilde{T}^2(u)R(u,v)=R(u,v)\widetilde{T}^2(u)\widetilde{T}^1(u),$$
which implies that
$$\widetilde{D}^1(u)\widetilde{D}^2(v)R_{\mathfrak{q}_{n}}(u,v)=R_{\mathfrak{q}_{n}}(u,v)\widetilde{D}^2(v)\widetilde{D}^1(u).$$

On the other hand, the Gauss decomposition \eqref{eq:Gaussfac} yields that
$$\widetilde{T}(u)
=\begin{pmatrix} *&*\\O&\widetilde{E}^{[m]}(u)\end{pmatrix}
\begin{pmatrix} *&O\\O&\widetilde{H}^{[m]}(u)\end{pmatrix}
\begin{pmatrix} *&O\\ *&\widetilde{F}^{[m]}(u)\end{pmatrix}
=\begin{pmatrix} *&*\\ *&\widetilde{E}^{[m]}(u)\widetilde{H}^{[m]}(u)\widetilde{F}^{[m]}(u)\end{pmatrix}
,$$
where $\widetilde{E}^{[m]}(u), \widetilde{H}^{[m]}(u), \widetilde{F}^{[m]}(u)$ are the inverses of $E^{[m]}(u), H^{[m]}(u)$ and $F^{[m]}(u)$, respectively. Hence,
$$\widetilde{D}(u)^{-1}=\left(\widetilde{E}^{[m]}(u)\widetilde{H}^{[m]}(u)\widetilde{F}^{[m]}(u)\right)^{-1}=F^{[m]}(u)H^{[m]}(u)E^{[m]}(u).$$
Hence, $\psi_m(T(u))=\widetilde{D}(u)^{-1}$ satisfies the RTT relation
$$R_{\mathfrak{q}_{n}}(u,v)\psi_m(T^1(u))\psi_m(T^2(v))=\psi_m(T^2(v))\psi_m(T^1(u))R_{\mathfrak{q}_{n}}(u,v).$$
We complete the proof.
\end{proof}

\begin{lemma}\label{lem:com}
The subalgebras $\varphi_m\left(\Y(\mathfrak{q}_m)\right)$ and $\psi_{m}\left(\Y(\mathfrak{q}_{n})\right)$ mutually commute in $\Y(\mathfrak{q}_{m+n})$. In particular,
$$\left[T_{ab}^1(u), \psi_{m}(T_{cd}^2(u))\right]=0,$$
for $1\leqslant a,b\leqslant m$ and $1\leqslant c,d\leqslant n$.
\end{lemma}
\begin{proof}
For $1\leqslant a,b\leqslant m$ and $1\leqslant c,d\leqslant n$, the identity \eqref{re:BTWT} implies that $$\left[T_{ab}^1(u), \widetilde{T}_{m+c,m+d}^2(v)\right]=0.$$
Using the same notation as in the proof of Theorem~\ref{thm:embedding}, the inverse matrix $T(u)^{-1}$ of $T(u)$ can be written as a block matrix $$T(u)^{-1}=\begin{pmatrix}\widetilde{A}(u)&\widetilde{B}(u)\\ \widetilde{C}(u)&\widetilde{D}(u)\end{pmatrix},$$
where $\widetilde{D}(u)=\left(T_{cd}(u)\right)_{m+1\leqslant c,d\leqslant m+n}$ is a $2n\times 2n$-matrix. Now, $\psi_{m}\left(T(u)\right)=\widetilde{D}(u)^{-1}$, which is generated by $\widetilde{T}_{m+c,m+d}(u)$ for $1\leqslant c,d\leqslant n$. So $\varphi_m\left(\Y(\mathfrak{q}_m)\right)$ and $\psi_{m}\left(\Y(\mathfrak{q}_{n})\right)$ commute in $\Y(\mathfrak{q}_{m+n})$.  
\end{proof}

\begin{lemma}
The following identities hold in $\mathrm{End}\left(\mathbb{C}^{1|1}\right)\otimes \mathrm{End}\left(\mathbb{C}^{1|1}\right)\otimes \Y(\mathfrak{q}_{m+n})[[u^{-1},v^{-1}]]$:
\begin{align}
\left[E_{ma}^1(u),\psi_{m}\left(T_{bc}^2(v)\right)\right]
=&\psi_{m}\left(T_{b,a-m}^2(v)\right)
\left(K(u,v)E_{m,c+m}^2(v)-E_{m,c+m}^1(u)K(u,v)\right),\label{eq:qnpsi1}\\
\left[F_{am}^1(u),\psi_{m}\left(T_{bc}^2(v)\right)\right]
=&\left(F_{b+m,m}^2(v)K(u,v)-K(u,v)F_{b+m,m}^1(v)\right)
\psi_{m}(T_{a-m,c}^2(v)),\label{eq:qnpsi2}
\end{align}
for $m<a$ and $1\leqslant b,c\leqslant n$.
\end{lemma}
\begin{proof}
 We prove the identity \eqref{eq:qnpsi1}. First, note that
the homomorphisms $\psi_m$ satisfy the following commutative diagram:
$$\xymatrix{
	\Y(\mathfrak{q}_{n-1})\ar[r]^{\psi_1}\ar[dr]_{\psi_{m+1}}&
	\Y(\mathfrak{q}_{n})\ar[d]^{\psi_{m}}\\
	&\Y(\mathfrak{q}_{n+m})}$$
We have $\psi_m\circ\psi_1(T_{bc}(u))=\psi_{m+1}(T_{bc}(u))$ for $1\leqslant b,c\leqslant n$. By the definition of $\psi_m$, we have $\psi_m(E_{1,a-m}(u))=E_{ma}(u)$. We only need to verify the identities in the case where $m=1$.

By Gauss decomposition \eqref{eq:Gaussfac}, $T_{11}(u)=H_1(u)$, $T_{1a}(u)=H_1(u)E_{1a}(u)$ and 
\begin{align*}
T_{b+1,c+1}(v)=\sum\limits_{r=1}^{\min(b+1,c+1)}F_{b+1,r}(v)H_r(v)E_{r,a+1}(v)
=F_{b+1,1}(v)H_1(v)E_{1,c+1}(v)+\psi_1\left(T_{bc}(v)\right).
\end{align*}
Since $T_{11}(u)$ commutes with $\psi_1(T_{bc})$ by Lemma~\ref{lem:com}, we compute that
\begin{align*}
&H_1^1(u)\left[E_{1a}^1(u),\psi_1\left(T_{bc}(v)\right)\right]\\
=&\left[T_{1a}^1(u),T_{b+1,c+1}^2(v)-F_{b+1,1}^2(v)H_1^2(v)E_{1,c+1}^2(v)\right]\\
=&\left[T_{1a}^1(u), T_{b+1,c+1}^2(v)\right]
-\left[T_{1a}^1(u), T_{b+1,1}^2(v)\right]E_{1,c+1}^2(v)\\
&-F_{b+1,1}^2(v)\left[T_{1a}^1(u),T_{1,c+1}^2(v)\right]
+F_{b+1,1}^2(v)\left[T_{1a}^1(u),T_{11}^2(v)\right]E_{1,c+1}^2(v).
\end{align*}
Using the RTT relation \eqref{eq:RTT}, we further deduce that
\begin{align*}
&H_1^1(u)\left[E_{1a}^1(u),\psi_2\left(T_{bc}(v)\right)\right]\\
=&K(u,v)T_{b+1,a}^1(u)T_{1,c+1}^2(v)-T_{b+1,a}^2(v)T_{1,c+1}^1(u)K(u,v)\\
&-K(u,v)T_{b+1,a}^1(u)T_{1,1}^2(v)E_{1,c+1}^2(v)
+T_{b+1,a}^2(v)T_{1,1}^1(u)K(u,v)E_{1,c+1}^2(v)\\
&-F_{b+1,1}^2(v)K(u,v)T_{1,a}^1(u)T_{1,c+1}^2(v)
+F_{b+1,1}^2(v)T_{1,a}^2(v)T_{1,c+1}^1(u)K(u,v)\\
&+F_{b+1,1}^2(v)K(u,v)T_{1,a}^1(u)T_{11}^2(v)E_{1,c+1}^2(v)
-F_{b+1,1}^2(v)T_{1,a}^2(v)T_{1,1}^1(u)K(u,v)E_{1,c+1}^2(v)\\
=&\left(T_{b+1,a}^2(v)-F_{b+1,1}^2(v)T_{1,a}^2(v)\right)T_{1,1}^1(u)
\left(K(u,v)E_{1,c+1}^2(v)-E_{1,c+1}^1(u)K(u,v)\right)\\
=&\psi_1\left(T_{b,a-1}^2(v)\right)H_1^1(u)\left(K(u,v)E_{1,c+1}^2(v)-E_{1,c+1}^1(u)K(u,v)\right).
\end{align*} 
Since $H_1(u)$ commutes with $\psi_1\left(T_{b,a-1}^2(v)\right)$ and $H_1(u)$ is invertible, we obtain \eqref{eq:qnpsi1}.  The proof of the relation \eqref{eq:qnpsi2} is similar.
\end{proof}

\section{Drinfeld relations in $\mathrm{Y}(\mathfrak{q}_2)$}\label{se:q2relation}

This section is devoted to discussing the super-Yangian $\Y(\mathfrak{q}_2)$. According to the block Gauss decomposition, the generator matrix $T(u)$ of $\Y(\mathfrak{q}_2)$ is factorized as
\begin{equation}
\label{q2:eq:gaussfac}
T(u)=\begin{pmatrix}1&\\F_1(u)&1\end{pmatrix}
\begin{pmatrix}H_1(u)&O\\O&H_2(u)\end{pmatrix}
\begin{pmatrix}1&E_1(u)\\O&1\end{pmatrix},
\end{equation}
where $H_1(u), H_2(u), E_1(u)$ and $F_1(u)$ are all $2\times2$-matrices of YQ form with enties in $\Y(\mathfrak{q}_2)[[u^{-1}]]$.
Theorem~\ref{thm:generator} asserts that $\Y(\mathfrak{q}_2)$ is generated by the coefficients of entries in the matrices $H_1(u)$, $H_2(u)$, $E_1(u)$ and $F_1(u)$. We will figure out relations among these generators. 

\begin{lemma}
The following identities hold in $\mathrm{End}\left(\mathbb{C}^{1|1}\right)\otimes\mathrm{End}\left(\mathbb{C}^{1|1}\right)\otimes\Y(\mathfrak{q}_2)[[u^{-1}, v^{-1}]]$:
\begin{align}
\left[H_a^{1}(u), H_a^{2}(v)\right]
=&K(u,v)H_a^{1}(u)H_a^{2}(v)-H_a^{2}(v)H_a^{1}(u)K(u,v), \quad a=1,2,
\label{q2:eq:HaHa}\\
\left[H_1^{1}(u), H_2^{2}(v)\right]
=&0.\label{q2:eq:H1H2}
\end{align}
\end{lemma}
\begin{proof}
Recall that the homomorphisms $\varphi_1:\Y(\mathfrak{q}_1)\rightarrow\Y(\mathfrak{q}_2)$ and $\psi_1:\Y(\mathfrak{q}_1)\rightarrow\Y(\mathfrak{q}_2)$ satisfy that
$$\varphi_1\left(T(u)\right)=H_1(u),\quad \psi_1\left(T(u)\right)=H_2(u),$$
where $T(u)$ is the generator matrix of $\Y(\mathfrak{q}_1)$. Now, the R-matrix for $\Y(\mathfrak{q}_1)$ is $R(u,v)=1-K(u,v)$, the RTT relation \eqref{re:BTT} for $\Y(\mathfrak{q}_1)$ can be written as
\begin{align*}
\left[T^{1}(u), T^{2}(v)\right]
=&K(u,v)T^{1}(u)T^{2}(v)-T^{2}(v)T^{1}(u)K(u,v),
\end{align*}
which yields \eqref{q2:eq:HaHa} for $a=1,2$ by applying the homomorphism $\varphi_1$ and $\psi_1$, respectively. 

The identity \eqref{q2:eq:H1H2} is obtained by Lemma~\ref{lem:com}.
\end{proof}

\begin{remark}
The identity \eqref{q2:eq:HaHa} is equivalent to 
$$R_{\mathfrak{q}_1}(u,v)H_a^1(u)H_a^2(v)=H_a^2(v)H_a^1(u)R_{\mathfrak{q}_1}(u,v).$$
It shows that there is a homomorphism of associative superalgebra 
$$\Y(\mathfrak{q}_1)\longrightarrow\Y(\mathfrak{q}_n),\quad  T(u)\mapsto H_a(u),$$
for each $a=1,2$.
\end{remark}

\begin{lemma}
The following identities hold in $\mathrm{End}\left(\mathbb{C}^{1|1}\right)\otimes\mathrm{End}\left(\mathbb{C}^{1|1}\right)\otimes\Y(\mathfrak{q}_2)[[u^{-1}, v^{-1}]]$:
\begin{align}
\left[H_1^{1}(u), E_1^{2}(v)\right]
=&H_1^1(u)\left(K(u,v)E_1^2(v)-E_1^1(u)K(u,v)\right),\label{q2:eq:H1E1}\\
\left[H_2^{1}(u), E_1^{2}(v)\right]
=&H_2^1(u)\left(K(u,v)E_1^1(u)-E_1^2(v)K(u,v)\right).\label{q2:eq:H2E1}
\end{align}
\end{lemma}
\begin{proof}
We consider the block RTT relation \eqref{re:BTT}
\begin{align*}
\left[T_{11}^{1}(u), \widetilde{T}_{12}^{2}(v)\right]
=T_{11}^{1}(u)K(u,v)\widetilde{T}_{12}^{2}(v)+T_{12}^{1}(u)K(u,v)\widetilde{T}_{22}^{2}(v),
\end{align*}
where $\widetilde{T}(u)=\left(\widetilde{T}_{ab}(u)\right)$ is the inverse of $T(u)$. Now, 
\begin{align*}
T_{11}(u)=&H_1(u),& T_{12}(u)=&H_1(u)E_1(u),\\
\widetilde{T}_{22}(v)=&\widetilde{H}_2(v),&
\widetilde{T}_{12}(v)=&-E_1(v)\widetilde{H}_2(v),
\end{align*}
where $\widetilde{H}_2(v)$ is the inverse of $H_2(v)$. We obtain that
\begin{align*}
\left[H_1^{1}(u),E_1^2(v)\widetilde{H}_2^2(v)\right]=&H_1^{1}(u)\left(K(u,v)E_1^{2}(v)-E_1^{1}(u)K(u,v)\right)\widetilde{H}_2^2(v),
\end{align*}
which yields \eqref{q2:eq:H1E1} since $H_1^1(u)$ commutes with $H_2^2(v)$.

For \eqref{q2:eq:H2E1}, we consider the block RTT relation \eqref{re:BTT}:
$$\left[T_{12}^1(u), \widetilde{T}_{22}^2(v)\right]=T_{11}^1(u)K(u,v)\widetilde{T}_{12}^2(v)+T_{12}^1(u)K(u,v)\widetilde{T}_{22}^2(v),
$$
which implies that
$$\left[E_1^1(u),\widetilde{H}_2^2(v)\right]
=\left(E_1^1(u)K(u,v)-K(u,v)E_1^2(v)\right)\widetilde{H}_2^2(v).$$
Since $\widetilde{H}_2(v)$ is the inverse of $H_2(v)$, we obtain that
$$\left[H_2^2(u),E_1^1(v)\right]=-H_2^2(u)\left[E_1^1(v),\widetilde{H}_2^2(u)\right]H_2^2(u)
=H_2^2(u)\left(E_1^1(v)K(v,u)-K(v,u)E_1^2(u)\right).$$
Then we conclude \eqref{q2:eq:H2E1} by Lemmas~\ref{re:PQrelation} and \ref{lem:PQtransform}.
\end{proof}

\begin{lemma}
The following identities holds in $\mathrm{End}\left(\mathbb{C}^{1|1}\right)\otimes\mathrm{End}\left(\mathbb{C}^{1|1}\right)\otimes\Y(\mathfrak{q}_2)[[u^{-1}, v^{-1}]]$:
\begin{equation}
\begin{aligned}
\left[E_1^{1}(u), F_1^{2}(v)\right]
=&\widetilde{H}_1^{1}(u)K(u,v)H_2^1(u)-H_2^2(v)K(u,v)\widetilde{H}_1^{2}(v).\label{q2:eq:E1F1}
\end{aligned}
\end{equation}
\end{lemma}
\begin{proof}
According to the block Gauss decomposition \eqref{q2:eq:gaussfac}, 
$$T_{11}(u)=H_1(u), T_{21}(u)=F_1(u)H_1(u), \text{ and }T_{12}(v)=H_1(v)E_1(v).$$
It yields that
\begin{align*}
\left[T_{21}^{1}(u), T_{12}^{2}(v)\right]=&\left[T_{21}^{1}(u), T_{11}^{2}(v)\right]E_1^{2}(v)+H_1^{2}(v)\left[F_{1}^{1}(u), E_{1}^{2}(v)\right]H_1^{1}(u)\\&+H_1^{2}(v)F_1^{1}(u)\left[H_1^{1}(u), E_{1}^{2}(v)\right].
\end{align*}
Then we deduce by the block RTT relation \eqref{re:BTT} and the identity \eqref{q2:eq:H1E1} that
\begin{align*}
&H_1^{2}(v)\left[F_1^{1}(u),\ E_1^{2}(v)\right]H_1^{1}(u)\\
=&K(u,v)T_{11}^{1}(u)T_{22}^{2}(v)-T_{11}^{2}(v)T_{22}^{1}(u)K(u,v)\\
&-K(u,v)H_1^{1}(u)F_1^{2}(v)H_1^{2}(v)E_1^{2}(v)+H_1^{2}(v)F_1^{1}(u)H_1^{1}(u)E_1^{1}(u)K(u,v)\\
=&K(u,v)H_1^1(u)H_2^2(v)-H_1^2(v)H_2^1(u)K(u,v).
\end{align*}
Since $H_1(u)$ is invertible, we further obtain that
\begin{align*}
\left[F_1^{1}(u),\ E_{1}^{2}(v)\right]
=&\widetilde{H}_1^2(v)K(u,v)H_2^2(v)-H_2^1(u)K(u,v)\widetilde{H}_1^1(u).
\end{align*}
Now, Lemmas~\ref{re:PQrelation} and \ref{lem:PQtransform} ensure that $\left[E_1^{1}(u), F_1^{2}(v)\right]=-P\left[F_1^1(v),\ E_1^{2}(u)\right]P$. Then the identity \eqref{q2:eq:E1F1} is derived from the above identity

\end{proof}

\begin{lemma}\label{lem:E11E12}
The following identities hold in $\mathrm{End}\left(\mathbb{C}^{1|1}\right)\otimes\mathrm{End}\left(\mathbb{C}^{1|1}\right)\otimes\Y(\mathfrak{q}_2)[[u^{-1}, v^{-1}]]$:
\begin{align}
&\left[E_1^{1}(u), E_1^{2}(v)\right]
=\frac{P}{u-v}E_1^1(v,u)E_1^2(v,u)+\frac{Q}{u+v}E_1^1(-v,u)E_1^2(v,-u),\label{q2:eq:E1E1}\\
&P\left[E_1^{1}(u),\ E_1^{2}(u)\right]-Q\left[E_1^{1}(u),\ E_1^{2}(-u)\right]=0,\label{eq:q2:E1uE1u}\\
&2u\left[E_1^{1}(u),\ E_1^{2}(u)\right]=-QE_1^1(u,-u)E_1^2(u,-u).\label{eq:q2:E1uE1u2}
\end{align}
where $E_1(v,u)=E_1(v)-E_1(u)$.
\end{lemma}
\begin{proof}
We consider the block RTT relation \eqref{eq:RTT}, and obtain that
\begin{align*}
\left[T_{12}^{1}(u),\ T_{12}^{2}(v)\right]
=&K(u,v)T_{12}^{1}(u)T_{12}^{2}(v)-T_{12}^{2}(v)T_{12}^{1}(u)K(u,v).
\end{align*}
Note that the Gauss decomposition \eqref{q2:eq:gaussfac} yields that $T_{12}(u)=H_1(u)E_1(u)$, we deduce that
\begin{align*}
&H_{1}^{2}(v)H_{1}^{1}(u)\left[E_1^{1}(u), E_1^{2}(v)\right]\\
=&\left[T_{12}^{1}(u), T_{12}^{2}(v)\right]
-\left[T_{12}^{1}(u),T_{11}^{2}(v)\right]E_1^{2}(v)
-T_{11}^{2}(v)\left[T_{11}^{1}(u), E_1^{2}(v)\right]E_1^{1}(u)\\
=&\frac{P}{u-v}H_1^{1}(v)E_1^{1}(v)H_1^{2}(u)\left(E_1^{2}(v)-E_1^{2}(u)\right)+\frac{Q}{u+v}H_1^{1}(-v)E_1^{1}(-v)H_1^{2}(-u)\left(E_1^{2}(v)-E_1^{2}(-u)\right)\\
&-\frac{P}{u-v}H_1^{1}(v)H_1^{2}(u)\left(E_1^{2}(v)-E_1^{2}(u)\right)E_1^{1}(u)-\frac{Q}{u+v}H_1^{1}(-v)H_1^{2}(-u)\left(E_1^{2}(v)-E_1^{2}(-u)\right)E_1^{1}(u).
\end{align*}
Now, the relation \eqref{q2:eq:H1E1} yields that
$$E_1^1(v)H_1^2(u)=H_1^2(u)E_1^1(v)-\frac{P}{u-v}H_1^1(u)\left(E_1^1(v)-E_1^1(u)\right)
+\frac{Q}{u+v}H_1^1(-u)\left(E_1^1(v)-E_1^1(-u)\right).$$
Since $H_1(u)$ and $H_1(v)$ are invertible, we obtain that
\begin{equation}
\begin{aligned}
R(u,v)\left[E_1^{1}(u), E_1^{2}(v)\right]
=&R(u,v)\left(\frac{P}{u-v}E_1^1(v,u)E_1^2(v,u)+\frac{Q}{u+v}E_1^1(-v,u)E_1^2(v,-u)\right)\\
&-\frac{P}{u-v}\left[E_1^{1}(u),E_1^{2}(u)\right]-\frac{Q}{u+v}\left[E_1^{1}(u),E_1^{2}(-u)\right]\\
&-\frac{PQ}{u^2-v^2}E_1^1(u,-u)E_1^2(u,-u),
\end{aligned}
\label{q2:eq:E1E1pf1}
\end{equation}
where $R(u,v)=1-\frac{P}{u-v}-\frac{Q}{u+v}$ is the R-matrix for $\Y(\mathfrak{q}_1)$.

Since $R(-u,-v)R(u,v)=1-\frac{1}{(u-v)^2}-\frac{1}{(u+v)^2}$, we further deduce by left multiplying $R(-u,-v)$ on both sides of the identity \eqref{q2:eq:E1E1pf1} that 
\begin{align*}
&\left(1-\frac{1}{(u-v)^2}-\frac{1}{(u+v)^2}\right)\left[E_1^{1}(u), E_1^{2}(v)\right]\\
=&\left(1-\frac{1}{(u-v)^2}-\frac{1}{(u+v)^2}\right)\left(\frac{P}{u-v}E_1^1(v,u)E_1^2(v,u)+\frac{Q}{u+v}E_1^1(-v,u)E_1^2(v,-u)\right)\\
&-R(-u,-v)\left(\frac{P}{u-v}\left[E_1^{1}(u),E_1^{2}(u)\right]+\frac{Q}{u+v}\left[E_1^{1}(u), E_1^{2}(-u)\right]\right)\\
&-R(-u,-v)\frac{PQ}{u^2-v^2}E_1^1(u,-u)E_1^2(u,-u). 
\end{align*}

The formal series $1-\frac{1}{(u-v)^2}-\frac{1}{(u+v)^2}$ has four zeros $v=\pm v_i, i=1,2,$ in $\mathbb{C}[[u^{-1}]]$, where
$$v_1=\left(1+u^{-2}+(1+4u^{-2})^{\frac{1}{2}}\right)^{-\frac{1}{2}},\quad v_2=\left(1+u^{-2}-(1+4u^{-2})^{\frac{1}{2}}\right)^{-\frac{1}{2}}.$$
Let $v=\pm v_i$ for $i=1,2$ in the above equality, we obtain four identities:
\begin{equation}
\begin{aligned}
&R(-u,\mp v_i)\frac{P}{u\mp v_i}\left[E_1^{1}(u),E_1^{2}(u)\right]
+R(-u,\mp v_i)\frac{Q}{u\pm v_i}\left[E_1^{1}(u), E_1^{2}(-u)\right]\\
&+R(-u,\mp v_i)\frac{PQ}{u^2-v_i^2}E_1^1(u,-u)E_1^2(u,-u)=0,
\end{aligned}
\label{q2:eq:E1E1pf3}
\end{equation}
for $i=1,2$. 

Moreover, Lemma~\ref{re:PQrelation} implies that
	\begin{align}
		&R(-u,-v_i)\frac{P}{u-v_i}-R(-u,v_i)\frac{P}{u+v_i}=\left(\frac{2v_i}{u^2-v^2_i}\right)\left(P+\frac{2u}{u^2-v^2_i}\right),\label{eq:Rzero1}\\
		&R(-u,-v_i)\frac{Q}{u+v_i}-R(-u,v_i)\frac{P}{u-v_i}=\left(\frac{-2v_i}{u^2-v^2_i}\right)\left(Q+\frac{2u}{u^2-v^2_i}\right),\label{eq:Rzero2}\\
		&R(-u,-v_i)\frac{PQ}{u^2-v_i^2}-R(-u,v_i)\frac{PQ}{u^2-v_1^2}=\left(\frac{2v_i}{u^2-v^2_i}\right)\frac{P+Q}{u^2-v_i^2}.\label{eq:Rzero3}
	\end{align}
Hence, for each $i=1,2$, the difference of the two identities in \eqref{q2:eq:E1E1pf3} corresponding to two signs yields that
\begin{align*}
&\left(P+\frac{2u}{u^2-v_i^2}\right)\left[E^{1}(u),E^{2}(u)\right]-\left(Q+\frac{2u}{u^2-v_i^2}\right)\left[E^{1}(u),E^{2}(-u)\right]\\
&+\frac{P+Q}{u^2-v_i^2}E_1^1(u,-u)E_1^2(u,-u)=0.
\end{align*}
Since $v_1\neq v_2$, we obtain \eqref{eq:q2:E1uE1u} and
$$2u\left[E^{1}(u),E^{2}(u)\right]-2u\left[E^{1}(u),E^{2}(-u)\right]+(P+Q)E_1^1(u,-u)E_1^2(u,-u)=0.$$
Then \eqref{eq:q2:E1uE1u2} is derived by left multiplying $Q$ on both sides of the above identity.

By \eqref{eq:q2:E1uE1u} and \eqref{eq:q2:E1uE1u2}, we also have
$$\frac{P}{u-v}\left[E_1^{1}(u),E_1^{2}(u)\right]+\frac{Q}{u+v}\left[E_1^{1}(u),E_1^{2}(-u)\right]
+\frac{PQ}{u^2-v^2}E_1^1(u,-u)E_1^2(u,-u)=0.$$
Therefore, the identity \eqref{q2:eq:E1E1pf1} is reduced to \eqref{q2:eq:E1E1} since $R(u,v)$ is invertible.
\end{proof}

\begin{remark}
  By comparing matrix entries in \eqref{q2:eq:E1E1}, the following identities hold in $\Y(\mathfrak{q}_2)[[u^{-1},v^{-1}]]$:
\begin{align*}
&\left[e_1(u),\ e_1(v)\right]=\frac{\left(e_1(v)-e_1(u)\right)^2}{u-v}-\frac{\left(\bar{e}_1(-v)-\bar{e}_1(u)\right)\left(\bar{e}_1(v)-\bar{e}_1(-u)\right)}{u+v},\\
&\left[e_1(u),\ \bar{e}_1(v)\right]=\frac{\left(\bar{e}_1(v)-\bar{e}_1(u)\right)\left(e_1(v)-e_1(u)\right)}{u-v}+\frac{\left(e_1(-v)-e_1(u)\right)\left(\bar{e}_1(v)-\bar{e}_1(-u)\right)}{u+v},\\
&\left[e_1(u),\ \bar{e}_1(v)\right]=\frac{\left(e_1(u)-e_1(v)\right)\left(\bar{e}_1(u)-\bar{e}_1(v)\right)}{u-v}+\frac{\left(\bar{e}_1(-u)-\bar{e}_1(v)\right)\left(e_1(u)-e_1(-v)\right)}{u+v},\\
&\left[\bar{e}_1(u),\ \bar{e}_1(v)\right]=-\frac{\left(\bar{e}_1(v)-\bar{e}_1(u)\right)^2}{u-v}-\frac{\left(e_1(-v)-e_1(u)\right)\left(e_1(v)-e_1(-u)\right)}{u+v}.
\end{align*}
Meanwhile, the identity \eqref{eq:q2:E1uE1u} is equivalent to the following identities in $\Y(\mathfrak{q}_2)[[u^{-1}]]$
\begin{align*}
&\left[e_1(u),\ e_1(u)\right]=\left[\bar{e}_1(-u),\ \bar{e}_1(u)\right],&&\left[e_1(u),\ e_1(-u)\right]=\left[\bar{e}_1(u),\ \bar{e}_1(u)\right],\\
&\left[e_1(u),\ \bar{e}_1(u)\right]=\left[\bar{e}_1(-u),\ e_1(u)\right],&&\left[e_1(u),\ \bar{e}_1(u)\right]=\left[e_1(-u),\ \bar{e}_1(u)\right].
\end{align*}  
\end{remark}

\section{Serre Relations}
\label{se:serrerelation}

This section is dedicated to Serre relations for the super-Yangian $\mathrm{Y}(\mathfrak{q}_n)$. By the embedding theorem \ref{thm:embedding}, it suffices to restrict to the case where $n=3$. 

\begin{lemma}
\label{lem:E1E2}
The following identity holds in $\mathrm{End}\left(\mathbb{C}^{1|1}\right)\otimes\mathrm{End}\left(\mathbb{C}^{1|1}\right)\otimes\mathrm{Y}(\mathfrak{q}_3)[[u^{-1},v^{-1}]]$:
\begin{equation}
\left[E^{1}_{1}(u),\ E^{2}_2(v)\right]=
K(u,v)E_{13}^2(v)-E_{13}^1(u)K(u,v)+\left(E_1^1(u)K(u,v)-K(u,v)E_1^2(v)\right)E_2^2(v),\label{re:q3:E1E2}
\end{equation}
which is equivalent to the following identities in $\mathrm{Y}(\mathfrak{q}_3)[[u^{-1},v^{-1}]]$:
\begin{align}
(u-v)\left[e_{1}(u),e_2(v)\right]
=&e_{13}(v)-e_{13}(u)+\left(e_1(u)-e_1(v)\right)e_2(v)
-\left(\bar{e}_1(-u){\color{red}-}\bar{e}_1(-v)\right)\bar{e}_2(v),
\label{eq:q3:e1e2}\\
(u+v)\left[e_{1}(u),\bar{e}_2(v)\right]
=&\bar{e}_{13}(v)-\bar{e}_{13}(-u)+\left(e_1(u)-e_1(-v)\right)\bar{e}_2(v)
+\left(\bar{e}_1(-u)-\bar{e}_1(v)\right)e_2(v),
\label{eq:q3:e1be2}\\
(u-v)\left[\bar{e}_1(u),e_2(v)\right]
=&\bar{e}_{13}(v)-\bar{e}_{13}(u)+\left(e_1(-u)-e_1(-v)\right)\bar{e}_2(v)
+\left(\bar{e}_1(u)-\bar{e}_1(v)\right)e_2(v),
\label{eq:q3:be1e2}\\
(u+v)\left[\bar{e}_1(u),\bar{e}_2(v)\right]
=&e_{13}(-u)-e_{13}(v)-\left(e_1(-u)-e_1(v)\right)e_2(v)
+\left(\bar{e}_1(u)-\bar{e}_1(-v)\right)\bar{e}_2(v).
\label{eq:q3:be1be2}
\end{align}
\end{lemma}
\begin{proof}
According to the block Gauss decomposition \eqref{eq:Gaussfac},
\begin{align*}
T_{11}(u)=&H_1(u),& T_{12}(u)=&H_1(u)E_1(u),& T_{13}(u)=&H_1(u)E_{13}(u),\\
\widetilde{T}_{33}(v)=&\widetilde{H}_3(v),&
\widetilde{T}_{23}(v)=&-E_2(v)\widetilde{H}_3(v),&
\widetilde{T}_{13}(v)=&\left(E_1(v)E_2(v)-E_{13}(v)\right)\widetilde{H}_3(v).
\end{align*}
Note that $\left[T_{12}^1(u), \widetilde{H}_3^2(u)\right]=\left[T_{12}^1(u), \widetilde{T}_{33}^2(v)\right]=0$ by \eqref{re:BTWT}, and $\left[H_1^1(u), E_2^2(v)\right]=0$, we have
$$\left[T_{12}^1(u),\widetilde{T}_{23}^2(v)\right]=-H_1^1(u)\left[E_1^1(u), E_2^2(v)\right]\widetilde{H}_3^2(v).$$

On the other hand, the relation \eqref{re:BTWT} also yields that
\begin{align*}
\left[T_{12}^1(u),\widetilde{T}_{23}^2(v)\right]
=&T_{11}^1(u)K(u,v)\widetilde{T}_{13}^2(v)
+T_{12}^1(u)K(u,v)\widetilde{T}_{23}^2(v)
+T_{13}^1(u)K(u,v)\widetilde{T}_{33}^2(v)\\
=&H_1^1(u)K(u,v)\left(E_1^2(v)E_2^2(v)-E_{13}^2(v)\right)\widetilde{H}_3^2(v)\\
&-H_1^1(u)E_1^1(u)K(u,v)E_2^2(v)\widetilde{H}_3^2(v)
+H_1^1(u)E_{13}^1(u)K(u,v)\widetilde{H}_3^2(v).
\end{align*}
It implies \eqref{re:q3:E1E2} since $H_1(u)$ and $\widetilde{H}_3(v)$ are invertible.

Comparing the matrix entries on both sides of \eqref{re:q3:E1E2}, we obtain \eqref{eq:q3:e1e2}, \eqref{eq:q3:e1be2}, \eqref{eq:q3:be1e2}, and \eqref{eq:q3:be1be2}.
\end{proof}

\begin{corollary}
The following identities hold in $\Y(\mathfrak{q}_3)[[u^{-1},v^{-1}]]$:
\begin{align}
&\left[e_1(u),e_{2}(v)\right]=\left[\bar{e}_1(-u),\bar{e}_{2}(v)\right],
\label{eq:q3:ee}\\
&\left[e_1(u),\ \bar{e}_{2}(v)\right]=-\left[\bar{e}_1(-u),e_{2}(v)\right],
\label{eq:q3:ebe}\\
u&\left[\mathring{e}_1(u), e_2(v)\right]
-v\left[e_1(u), \mathring{e}_2(v)\right]=e_1(u)e_2(v)-\bar{e}_1(-u)\bar{e}_2(v),
\label{eq:q3:e1circle}\\
u&\left[\mathring{\bar{e}}_1(u), e_2(v)\right]
-v\left[\bar{e}_1(u), \mathring{e}_2(v)\right]=\bar{e}_1(u)e_2(v)+e_1(-u)\bar{e}_2(v),
\label{eq:q3:be1circle}
\end{align}
where $\mathring{e}_i(u)=\sum\limits_{r\geqslant2}e_i^{(r)}u^{-r}$ and $\mathring{\bar{e}}_i(u)=\sum\limits_{r\geqslant2}\bar{e}_i^{(r)}u^{-r}$ for $i=1,2$.
\end{corollary}
\begin{proof}
Identities \eqref{eq:q3:ee} and \eqref{eq:q3:ebe} can easily be observed in \eqref{eq:q3:e1e2}-\eqref{eq:q3:be1be2}. 

In order to show \eqref{eq:q3:e1circle}, we read the coefficients of $u^0$ in \eqref{eq:q3:e1e2}, and obtain that
$$\left[e_1^{(1)},e_2(v)\right]
=e_{13}(v)-e_1(v)e_2(v)-\bar{e}_1(v)\bar{e}_2(v).$$
While the coefficients of $v^0$ in \eqref{eq:q3:e1e2} yields that
$$\left[e_1(u),e_2^{(1)}\right]=e_{13}(u).$$
Hence, the identity \eqref{eq:q3:e1e2} implies that
$$(u-v)\left[e_1(u),e_2(v)\right]
=\left[e_1^{(1)},e_2(v)\right]-\left[e_1(u),e_2^{(1)}\right]+e_1(u)e_2(v)-\bar{e}_1(-u)\bar{e}_2(v),$$
which yields \eqref{eq:q3:e1circle}. The identity \eqref{eq:q3:be1circle} can be verified by a similar argument.
\end{proof}

\begin{lemma}
The following identities hold in $\mathrm{End}\left(\mathbb{C}^{1|1}\right)\otimes\mathrm{End}\left(\mathbb{C}^{1|1}\right)\otimes\mathrm{Y}(\mathfrak{q}_3)[[u^{-1}, v^{-1}]]$:
\begin{align}
\left[E^{1}_1(u),E^{2}_{13}(v)-E^{2}_{1}(v)E^{2}_2(v)\right]=&-\left[E^{1}_1(u),E^{2}_2(v)\right]E^{1}_1(u),\label{re:q3:E1E13}\\
\left[E^{1}_{13}(u),E^{2}_2(v)\right]=&E^{2}_2(v)\left[E^{1}_1(u),E^{2}_2(v)\right].\label{re:q3:E13E2}
\end{align}
\end{lemma}
\begin{proof}
By the block Gauss decomposition \eqref{eq:Gaussfac}, 
$$T_{11}(u)=H_1(u), T_{12}(u)=T_{11}(u)E_1(u),\text{ and }\widetilde{T}_{13}(v)=-\left(E_{13}(v)-E_1(v)E_2(v)\right)\widetilde{H}_3(v).$$ 
Then it follows from \eqref{re:BTWT} that
$$0=\left[T_{12}^{1}(u),\ \widetilde{T}^{2}_{13}(v)\right]
=\left[T_{11}^1(u),\widetilde{T}_{13}^2(v)\right]E_1^1(u)+H_1^1(u)\left[E_1^1(u),\widetilde{T}_{13}^2(v)\right].$$

On one hand, we deduce from $\left[E_1^1(u), H_3^2(v)\right]=0$ that
$$H_1^1(u)\left[E_1^1(u),\widetilde{T}_{13}^2(v)\right]=-H_1^1(u)\left[E_1^1(u),E_{13}^2(v)-E_1^2(v)E_2^2(v)\right]\widetilde{H}_3^2(v).$$
On the other hand, we compute by using \eqref{re:BTWT} and \eqref{re:q3:E1E2} that    
\begin{align*}
\left[T_{11}^{1}(u),\ \widetilde{T}^{2}_{13}(v)\right]
=&T_{11}^{1}(u)K(u,v)\widetilde{T}^{2}_{13}(v)+T_{12}^{1}(u)K(u,v)\tilde{T}^{2}_{23}(v)+T_{13}^{1}(u)K(u,v)\widetilde{T}^{2}_{33}(v)\\
=&-H_1^1(u)K(u,v)\left(E_{13}^2(v)-E_1^2(v)E_2^2(v)\right)\widetilde{H}_3^2(v)\\
&-H_1^1(u)E_1^1(u)K(u,v)E_2^2(v)\widetilde{H}_3^2(v)+H_1^1(u)E_{13}(u)K(u,v)\widetilde{H}_3^2(v)\\
=&-H_1^1(u)\left[E_1^1(u), E_2^2(v)\right]\widetilde{H}_3^2(v).    
\end{align*}
Hence,
$$H_1^1(u)\left[E_1^1(u),E_{13}^2(v)-E_1^2(v)E_2^2(v)\right]\widetilde{H}_3^2(v)
=-H_1^1(u)\left[E_1^1(u), E_2^2(v)\right]\widetilde{H}_3^2(v)E_1^1(u),$$
which implies \eqref{re:q3:E1E13} since $H_3(v)$ commutes with $E_1(u)$ and $H_1(u), H_3(v)$ are invertible.

For \eqref{re:q3:E13E2}, we deduce from \eqref{re:BTWT} that
$$0=\left[T_{13}^{1}(u),\widetilde{T}^{2}_{23}(v)\right]
=-\left[T_{13}^1(u),E_2^2(v)\right]\widetilde{H}_3^2(v)
-E_2^2(v)\left[T_{13}^1(u),\widetilde{T}_{33}^2(v)\right],$$
and
\begin{align*}
	\left[T_{13}^{1}(u),\ \tilde{T}^{2}_{33}(v)\right]=&T_{11}^{1}(u)K(u,v)\tilde{T}^{2}_{13}(v)+T_{12}^{1}(u)K(u,v)\tilde{T}^{2}_{23}(v)
	+T_{13}^{1}(u)K(u,v)\tilde{T}^{2}_{33}(v).
		\end{align*}
Then it follows from \eqref{re:q3:E1E2} that
$$\left[T_{13}^{1}(u),\tilde{T}^{2}_{33}(v)\right]=-H_1^1(u)\left[E_1^1(u), E_2^2(v)\right]\widetilde{H}_3^2(v).$$

Since $H_1(u)$ commutes with $E_2(v)$, we further deduce that
$$H_1^1(u)\left[E_{13}^1(u),E_2^2(v)\right]\widetilde{H}_3^2(v)
=-E_2^2(v)\left[T_{13}^1(u),\widetilde{T}_{33}^2(v)\right]
=H_1^1(u)E_2^2(v)\left[E_1^1(u), E_2^2(v)\right]\widetilde{H}_3^2(v),$$
which yields \eqref{re:q3:E13E2} since $H_1(u)$ and $H_3(v)$ are invertible.
\end{proof}

\begin{lemma}
\label{lem:q3:E1E13}
The following identities hold in $\mathrm{End}\left(\mathbb{C}^{1|1}\right)\otimes\mathrm{End}\left(\mathbb{C}^{1|1}\right)\otimes\mathrm{Y}(\mathfrak{q}_3)[[u^{-1}, v^{-1}]]$:
\begin{align}
&\left[E_1^1(u),E_{13}^2(v)\right]
=\frac{P}{u-v}\left(E_1^1(v)-E_1^1(u)\right)\left(E_{13}^2(v)-E_{13}^2(u)\right)\nonumber\\
&\qquad\qquad\qquad\qquad+\frac{Q}{u+v}\left(E_1^1(-v)-E_1^1(u)\right)\left(E_{13}^2(v)-E_{13}^2(-u)\right),\label{eq:q3:E1uE13v}\\
&P\left[E_1^1(u),E_{13}^2(u)\right]
-Q\left[E_1^1(u),E_{13}^2(-u)\right]=0,\label{eq:q3:E1uE13u1}\\
&2u\left[E_1^1(u),E_{13}^2(u)\right]=-Q\left(E_1^1(u)-E_1^1(-u)\right)\left(E_{13}^2(u)-E_{13}^2(-u)\right).
\label{eq:q3:E1uE13u2}
\end{align}
\end{lemma}
\begin{proof}
We deduce by \eqref{re:q3:E1E13} that
\begin{align*}
\left[E_1^1(u),E_{13}^2(v)\right]
=&\left[E_1^1(u),E_{13}^2(v)-E_1^2(v)E_2^2(v)+E_1^2(v)E_2^2(v)\right]\\
=&-\left[E_1^1(u),E_2^2(v)\right]E_1^1(u)
+E_1^2(v)\left[E_1^1(u),E_2^2(v)\right]
+\left[E_1^1(u),E_1^2(v)\right]E_2^2(v).
\end{align*}
By using \eqref{re:q3:E1E2}, we compute that
\begin{align*}
\left[E_1^1(u),E_{13}^2(v)\right]
=&-\frac{P}{u-v}\left(E_{13}^2(v)-E_{13}^2(u)\right)E_1^1(u)
+\frac{P}{u-v}\left(E_1^2(v)-E_1^2(u)\right)E_2^2(v)E_1^1(u)\\
&-\frac{Q}{u+v}\left(E_{13}^2(v)-E_{13}^2(-u)\right)E_1^1(u)
+\frac{Q}{u+v}\left(E_1^2(v)-E_1^2(-u)\right)E_2^2(v)E_1^1(u)\\
&+\frac{P}{u-v}E_1^1(v)\left(E_{13}^2(v)-E_{13}^2(u)\right)
-\frac{P}{u-v}E_1^1(v)\left(E_1^2(v)-E_1^2(u)\right)E_2^2(v)\\
&+\frac{Q}{u+v}E_1^1(-v)\left(E_{13}^2(v)-E_{13}^2(-u)\right)
-\frac{Q}{u+v}E_1^1(-v)\left(E_1^2(v)-E_1^2(-u)\right)E_2^2(v)\\
&+\left[E_1^1(u),E_1^2(v)\right]E_2^2(v).
\end{align*}
Note that
\begin{align*}
\left(E_{13}^2(v)-E_{13}^2(u)\right)E_1^1(u)
=&E_1^1(u)\left(E_{13}^2(v)-E_{13}^2(u)\right)-\left[E_1^1(u),E_{13}^2(v)-E_{13}^2(u)\right],
\end{align*}
and
\begin{align*}
E_2^2(v)E_1^1(u)=&E_1^1(u)E_2^2(v)-\left[E_1^1(u),E_2^2(v)\right]\\
=&E_1^1(u)E_2^2(v)-\frac{P}{u-v}\left(E_{13}^2(v)-E_{13}^2(u)\right)
-\frac{Q}{u+v}\left(E_{13}^2(v)-E_{13}^2(-u)\right)\\
&+\frac{P}{u-v}\left(E_1^2(v)-E_1^2(u)\right)E_2^2(v)
+\frac{Q}{u+v}\left(E_1^2(v)-E_1^2(-u)\right)E_2^2(v).
\end{align*}
We further deduce that
\begin{align*}
R(u,v)\left[E_1^1(u),E_{13}^2(v)\right]
=&-\frac{P}{u-v}\left[E_1^1(u),E_{13}^2(u)\right]
-\frac{Q}{u+v}\left[E_1^1(u),E_{13}^2(-u)\right]\\
&-\frac{PQ}{u^2-v^2}\left(E_1^1(u)-E_1^1(-u)\right)\left(E_{13}^2(u)-E_{13}^2(-u)\right)\\
&+R(u,v)\frac{P}{u-v}\left(E_1^1(v)-E_1^1(u)\right)\left(E_{13}^2(v)-E_{13}^2(u)\right)\\
&+R(u,v)\frac{Q}{u+v}\left(E_1^1(-v)-E_1^1(u)\right)\left(E_{13}^2(v)-E_{13}^2(-u)\right)\\
&-R(u,v)\frac{P}{u-v}\left(E_1^1(v)-E_1^1(u)\right)\left(E_1^2(v)-E_1^2(u)\right)E_2^2(v)\\
&-R(u,v)\frac{Q}{u+v}\left(E_1^1(-v)-E_1^1(u)\right)\left(E_1^2(v)-E_1^2(-u)\right)E_2^2(v)\\
&+R(u,v)\left[E_1^1(u),E_1^2(v)\right]E_2^2(v)\\
&+\frac{P}{u-v}\left[E_1^1(u),E_1^2(u)\right]E_2^2(v)
+\frac{Q}{u+v}\left[E_1^2(u),E_1^2(-u)\right]E_2^2(v)\\
&+\frac{PQ}{u^2-v^2}\left(E_1^1(u)-E_1^1(-u)\right)\left(E_1^2(u)-E_1^2(-u)\right)E_2^2(v),
\end{align*}
where $R(u,v)=1-\frac{P}{u-v}-\frac{Q}{u+v}$.

Now, the identity \eqref{q2:eq:E1E1} shows that
\begin{align*}
\left[E_1^1(u),E_1^2(v)\right]
=&\frac{P}{u-v}\left(E_1^1(v)-E_1^1(u)\right)\left(E_1^2(v)-E_1^2(u)\right)\\
&+\frac{Q}{u+v}\left(E_1^1(-v)-E_1^1(u)\right)\left(E_1^2(v)-E_1^2(-u)\right).
\end{align*}
Meanwhile, the identities \eqref{eq:q2:E1uE1u} and \eqref{eq:q2:E1uE1u2} ensure that
\begin{align*}
&\frac{P}{u-v}\left[E_1^1(u),E_1^2(u)\right]
+\frac{Q}{u+v}\left[E_1^2(u),E_1^2(-u)\right]\\
=&-\frac{PQ}{u^2-v^2}\left(E_1^1(u)-E_1^1(-u)\right)\left(E_1^2(u)-E_1^2(-u)\right).
\end{align*}
Thus we obtain that
\begin{equation}
\begin{aligned}
R(u,v)\left[E_1^1(u),E_{13}^2(v)\right]
=&-\frac{P}{u-v}\left[E_1^1(u),E_{13}^2(u)\right]
-\frac{Q}{u+v}\left[E_1^1(u),E_{13}^2(-u)\right]\\
&-\frac{PQ}{u^2-v^2}\left(E_1^1(u)-E_1^1(-u)\right)\left(E_{13}^2(u)-E_{13}^2(-u)\right)\\
&+R(u,v)\frac{P}{u-v}\left(E_1^1(v)-E_1^1(u)\right)\left(E_{13}^2(v)-E_{13}^2(u)\right)\\
&+R(u,v)\frac{Q}{u+v}\left(E_1^1(-v)-E_1^1(u)\right)\left(E_{13}^2(v)-E_{13}^2(-u)\right).
\end{aligned}
\label{eq:q3:E1uE13vpf}
\end{equation}
By left multiplying $R(-u,-v)$ on both sides of \eqref{eq:q3:E1uE13vpf}, we have
\begin{align*}
&\left(1-\frac{1}{(u-v)^2}-\frac{1}{(u+v)^2}\right)\left[E_1^1(u),E_{13}^2(v)\right]\\
=&-R(-u,-v)\frac{P}{u-v}\left[E_1^1(u),E_{13}^2(u)\right]
-R(-u,-v)\frac{Q}{u+v}\left[E_1^1(u),E_{13}^2(-u)\right]\\
&-R(-u,-v)\frac{PQ}{u^2-v^2}\left(E_1^1(u)-E_1^1(-u)\right)\left(E_{13}^2(u)-E_{13}^2(-u)\right)\\
&+\left(1-\frac{1}{(u-v)^2}-\frac{1}{(u+v)^2}\right)\frac{P}{u-v}\left(E_1^1(v)-E_1^1(u)\right)\left(E_{13}^2(v)-E_{13}^2(u)\right)\\
&+\left(1-\frac{1}{(u-v)^2}-\frac{1}{(u+v)^2}\right)\frac{Q}{u+v}\left(E_1^1(-v)-E_1^1(u)\right)\left(E_{13}^2(v)-E_{13}^2(-u)\right),
\end{align*}
since $R(-u,-v)R(u,v)=1-\frac{1}{(u-v)^2}-\frac{1}{(u+v)^2}$.

The formal series $1-\frac{1}{(u-v)^2}-\frac{1}{(u+v)^2}$ has four zeros $v=\pm v_i\in\mathbb{C}[[u^{-1}]]$ for $i=1,2$, where
$$v_1=\left(1+u^{-2}+(1+4u^{-2})^{\frac{1}{2}}\right)^{-\frac{1}{2}},\quad v_2=\left(1+u^{-2}-(1+4u^{-2})^{\frac{1}{2}}\right)^{-\frac{1}{2}}.$$
Let $v=\pm v_i$ in the above equality. We obtain the following four identities in $\mathrm{End}\left(\mathbb{C}^{1|1}\right)\otimes\mathrm{End}\left(\mathbb{C}^{1|1}\right)\otimes\mathrm{Y}(\mathfrak{q}_3)[[u^{-1}]]$:
\begin{align*}
&R(-u,\mp v_i)\frac{P}{u\mp v_i}\left[E_1^1(u),E_{13}^2(u)\right]
+R(-u,\mp v_i)\frac{Q}{u\pm v_i}\left[E_1^1(u),E_{13}^2(-u)\right]\\
=&-R(-u,\mp v_i)\frac{PQ}{u^2-v_i^2}\left(E_1^1(u)-E_1^1(-u)\right)\left(E_{13}^2(u)-E_{13}^2(-u)\right)
\end{align*}
for $i=1,2$. By \eqref{eq:Rzero1}, \eqref{eq:Rzero2} and \eqref{eq:Rzero3} in the proof of Lemma~\ref{lem:E11E12}, we take the difference of two equalities corresponding to two signs for each $i=1,2$, and deduce that
\begin{align*}
&\left(P+\frac{2u}{u^2-v_i^2}\right)\left[E_1^1(u),E_{13}^2(u)\right]
-\left(Q+\frac{2u}{u^2-v_i^2}\right)\left[E_1^1(u),E_{13}^2(-u)\right]\\
=&-\frac{P+Q}{u^2-v_i^2}\left(E_1^1(u)-E_1^1(-u)\right)\left(E_{13}^2(u)-E_{13}^2(-u)\right).
\end{align*}
Since $v_1\neq v_2$, we obtain \eqref{eq:q3:E1uE13u1} and 
$$2u\left(\left[E_1^1(u),E_{13}^2(u)\right]-\left[E_1^1(u),E_{13}^2(-u)\right]\right)
=-(P+Q)\left(E_1^1(u)-E_1^1(-u)\right)\left(E_{13}^2(u)-E_{13}^2(-u)\right),$$
which yields \eqref{eq:q3:E1uE13u2} by left multiplying $Q$. Consequently,
\begin{align*}
&\frac{P}{u-v}\left[E_1^1(u),E_{13}^2(u)\right]+\frac{Q}{u+v}\left[E_1^1(u),E_{13}^2(-u)\right]\\
=&-\frac{PQ}{u^2-v^2}\left(E_1^1(u)-E_1^1(-u)\right)\left(E_{13}^2(u)-E_{13}^2(-u)\right).
\end{align*}
Therefore, the identity \eqref{eq:q3:E1uE13vpf} reduces to the identity \eqref{eq:q3:E1uE13v} since $R(u,v)$ is invertible.
\end{proof}

\begin{lemma}\label{lem:q3:E13E2}
The following identities hold in $\mathrm{End}\left(\mathbb{C}^{1|1}\right)\otimes\mathrm{End}\left(\mathbb{C}^{1|1}\right)\otimes\mathrm{Y}(\mathfrak{q}_3)[[u^{-1}, v^{-1}]]$:
\begin{align}
&\begin{aligned}
\left[\mathcal{E}_{13}^1(u),E_2^2(v)\right]
=&\left(\mathcal{E}_{13}^1(v)-\mathcal{E}_{13}^1(u)\right)\left(E_2^2(v)-E_2^2(u)\right)\frac{P}{u-v}\\
&+\left(\mathcal{E}_{13}^1(-v)-\mathcal{E}_{13}^1(u)\right)\left(E_2^2(v)-E_2^2(-u)\right)\frac{Q}{u+v},
\end{aligned}
\label{eq:q3:E13uuE2v}\\
&\left[\mathcal{E}_{13}^1(v), E_2^2(v)\right]P+\left[\mathcal{E}_{13}^1(-v), E_2^2(v)\right]Q=0,
\label{eq:q3:E13vvE2v1}\\
&2v\left[\mathcal{E}_{13}^1(v),E_2^2(v)\right]
=-\left(\mathcal{E}_{13}^1(v)-\mathcal{E}_{13}^1(-v)\right)(E_2^2(v)-E_2^2(-v))Q,
\label{eq:q3:E13vvE2v2}
\end{align}
where $\mathcal{E}_{13}(u)=E_{13}(u)-E_1(u)E_2(u).$
\end{lemma}
\begin{proof}
    The proof is similar to the Lemma \ref{lem:q3:E1E13}. 
\end{proof}

\begin{proposition}
The following Serre relations hold in $\mathrm{End}\left(\mathbb{C}^{1|1}\right)^{\otimes3}\otimes\mathrm{Y}(\mathfrak{q}_3)[[u^{-1}, v^{-1},w^{-1}]]$:
\begin{equation}
\begin{aligned}
&\frac{P^{12}}{u-v}\left[E_a^1(u,v),\ \left[E_a^2(u,v),\ E_b^3(w)\right]\right]+\frac{Q^{12}}{u+v}\left[E_a^1(u,-v),\ \left[E_a^2(-u,v),\ E_b^3(w)\right]\right]=0, 
\end{aligned}\label{eq:q3:Serre}
\end{equation}
where $E_a(u,v)=E_a(u)-E_a(v)$, for $a=1,b=2$ or $a=2,b=1$.
\end{proposition}
\begin{proof}
We verify the case where $a=1$ and $b=2$. The other case can be checked similarly. For convenience, we set $E_{13}(u,v)=E_{13}(u)-E_{13}(v)$.

We observe from \eqref{re:q3:E1E2} that
$$\left[E_1^2(u),E_2^{(1),3}\right]=E_{13}^2(u)\left(P^{23}-Q^{23}\right),$$
where $E_2^{(1)}$ is the coefficient of $v^{-1}$ in $E_2(v)$. By using \eqref{eq:q3:E1uE13v}, we compute that
\begin{align*}
&\left[E_1^1(u,v),E_{13}^2(u,v)\right]-\left[E_1^1(u),E_{13}^2(u)\right]-\left[E_1^1(v),E_{13}^2(v)\right]\\
=&-\frac{Q^{12}}{u+v}E_1^1(-v,u)E_{13}^2(v,-u)-\frac{Q^{12}}{u+v}E_1^1(-u,v)E_{13}^2(u,-v),
\end{align*}
and, 
\begin{align*}
&\left[E_1^1(u,-v),E_{13}^2(-u,v)\right]-\left[E_1^1(u),E_{13}^2(-u)\right]+\left[E_1^1(-v),E_{13}^2(v)\right]\\
=&-\frac{P^{12}}{u-v}E_1^1(v,u)E_{13}^2(v,u)-\frac{P^{12}}{u-v}E_1^1(-u,-v)E_{13}^2(-u,-v).
\end{align*}

Hence,
\begin{align*}
&\frac{P^{12}}{u-v}\left[E_1^1(u,v),E_{13}^2(u,v)\right]
+\frac{Q^{12}}{u+v}\left[E_1^1(u,-v),E_{13}^2(-u,v)\right]\\    
=&\frac{P^{12}}{u-v}\left[E_1^1(u),E_{13}^2(u)\right]+\frac{Q^{12}}{u+v}\left[E_1^1(u),E_{13}^2(-u)\right]+\frac{P^{12}Q^{12}}{u^2-v^2}E_1^1(u,-u)E_{13}^2(u,-u)\\
&+\frac{P^{12}}{u-v}\left[E_1^1(v),E_{13}^2(v)\right]+\frac{Q^{12}}{u+v}\left[E_1^1(-v),E_{13}^2(v)\right]+\frac{P^{12}Q^{12}}{u^2-v^2}E_1^1(v,-v)E_{13}^2(v,-v),
\end{align*}
which vanishes by the identities \eqref{eq:q3:E1uE13u1} and \eqref{eq:q3:E1uE13u2}. Therefore,
\begin{equation}
\begin{aligned}
&\frac{P^{12}}{u-v}\left[E_1^1(u,v),\left[E_1^2(u,v),E_2^{(1),3}\right]\right]+\frac{Q^{12}}{u+v}\left[E_1^1(u,-v),\left[E_1^2(-u,v),E_2^{(1),3}\right]\right]=0.
\end{aligned}
\label{eq:serre:pf1}
\end{equation}

On the other hand, \eqref{q2:eq:H2E1} yields that
$$\left[E_2^{(1),3},\widetilde{H}_3^4(u)\right]=
(P^{34}+Q^{34})E_2^4(u)\widetilde{H}_3^4(u),$$
where $\widetilde{H}_3(u)$ is the inverse of $H_3(u)$. Since $H_3(u)$ commutes with $E_1(v)$, we obtain that
\begin{align*}
&\frac{P^{12}}{u-v}\left[E_1^1(u,v),\left[E_1^2(u,v,(P^{34}+Q^{34})E_2^4(u)\widetilde{H}_3^4(u)\right]\right]\\
+&\frac{Q^{12}}{u+v}\left[E_1^1(u,-v),\left[E_1^2(-u,v),(P^{34}+Q^{34})E_2^4(u)\widetilde{H}_3^4(u)\right]\right]=0,
\end{align*}
by taking commutator with $\widetilde{H}_3^4(u)$ on both sides of \eqref{eq:serre:pf1}. This proves the identity \eqref{eq:q3:Serre} since $P^{34}+Q^{34}$ and $\widetilde{H}_3^4(u)$ are invertible and commutes with $P^{12}, Q^{12}, E_1^1(u), E_1^2(v)$. 
\end{proof}

\begin{remark}
By comparing the matrix entries on both sides of the identity \eqref{eq:q3:Serre}, we deduce the following Serre relations:
  \begin{align*}
        &\frac{\left[e_a(u)-e_a(v),\left[e_a(u)-e_a(v),e_b(w)\right]\right]}{u-v}
+\frac{\left[\bar{e}_a(u)-\bar{e}_a(-v),\left[\bar{e}_a(-u)-\bar{e}_a(v), e_b(w)\right]\right]}{u+v}=0,\\
&\frac{\left[\bar{e}_a(u)-\bar{e}_a(v),\left[e_a(u)-e_a(v),e_b(w)\right]\right]}{u-v}
-\frac{\left[e_a(u)-e_a(-v),\left[\bar{e}_a(-u)-\bar{e}_a(v), e_b(w)\right]\right]}{u+v}=0,\\
&\frac{\left[e_a(u)-e_a(v),\left[\bar{e}_a(u)-\bar{e}_a(v),e_b(w)\right]\right]}{u-v}
+\frac{\left[\bar{e}_a(u)-\bar{e}_a(-v),\left[e_a(-u)-e_a(v), e_b(w)\right]\right]}{u+v}=0,\\
&\frac{\left[\bar{e}_a(u)-\bar{e}_a(v),\left[\bar{e}_a(u)-\bar{e}_a(v),e_b(w)\right]\right]}{u-v}
-\frac{\left[e_a(u)-e_a(-v),\left[e_a(-u)-e_a(v), e_b(w)\right]\right]}{u+v}=0.
    \end{align*}
\end{remark}

\begin{remark}
By comparing the coefficients of $u^{-1}$ on both sides of \eqref{eq:q3:Serre}, we obtain the following simplified Serre relation
\begin{equation}
P^{12}\left[E_a^1(v),\left[E_a^2(v),\ E_b^3(w)\right]\right]
+Q^{12}\left[E_a^1(-v),\left[E_a^2(v),\ E_b^3(w)\right]\right]=0.
\label{eq:serresim}
\end{equation}
We will show in Theorem~\ref{thm:isomorphism} that the isomorphism between the Drinfeld presentation and the RTT presentation can be established by merely using the simplified Serre relation~\ref{eq:serresim}. Hence, we only include these simplified Serre relations \eqref{eq:Dr:eserre1}-\eqref{eq:Dr:eserre4} in the defining relation for the Drinfeld presentation in Theorem~\ref{thm:isomorphism}.   
\end{remark}

\section{The Drinfeld Presentation}\label{se:Maintheorem}

In this section, we prove our main theorem, which establishes the isomorphism between the Drinfeld presentation and the R-matrix presentation for the queer super-Yangian.

\begin{theorem}\label{thm:isomorphism}
The super-Yangian $\YR(\mathfrak{q}_n)$ is generated by even elements $h_a^{(r)},  e_b^{(r)}, f_b^{(r)}$ and odd elements $\bar{h}_a^{(r)}, \bar{e}_b^{(r)}, \bar{f}_b^{(r)}$ for $a=1,2,\ldots, n$, $b=1,2,\ldots,n-1$, and $r\geqslant1$. We set
\begin{align*}
h_a(u)=&1+\sum\limits_{r\geqslant1}h_a^{(r)}u^{-r},&
e_b(u)=&\sum\limits_{r\geqslant1}e_b^{(r)}u^{-r},&
f_b(u)=&\sum\limits_{r\geqslant1}f_b^{(r)}u^{-r},\\
\bar{h}_a(u)=&\sum\limits_{r\geqslant1}\bar{h}_a^{(r)}u^{-r},&
\bar{e}_b(u)=&\sum\limits_{r\geqslant1}\bar{e}_b^{(r)}u^{-r},&
\bar{f}_b(u)=&\sum\limits_{r\geqslant1}\bar{f}_b^{(r)}u^{-r}.
\end{align*}
Then the mere relations satisfied by these generators are given as follows:
\begin{align}
&\left[h_a(u), h_b(v)\right]
=\left[h_a(u),\bar{h}_b(v)\right]=\left[\bar{h}_a(u), h_b(v)\right]=\left[\bar{h}_a(u),\bar{h}_b(v)\right]=0,
\qquad\text{ if }a\neq b, \label{eq:Dr:hahb}\\
&\left[h_a(u),h_a(v)\right]=\frac{u-v}{(u+v)(u-v-1)}\left(\bar{h}_a(u)\bar{h}_a(v)-\bar{h}_a(-v)\bar{h}_a(-u)\right),
\label{eq:Dr:haha}\\
&\left[h_a(u),\bar{h}_a(v)\right]
=\frac{\bar{h}_a(u)h_a(v)-\bar{h}_a(v)h_a(u)}{u-v}
+\frac{h_a(u)\bar{h}_a(v)-h_a(-v)\bar{h}_a(-u)}{u+v},
\label{eq:Dr:hahta}\\
&\left[\bar{h}_a(u),\bar{h}_a(v)\right]
=\frac{h_a(u)h_a(v)-h_a(-v)h_a(-u)}{u+v}
-\frac{\bar{h}_a(u)\bar{h}_a(v)-\bar{h}_a(v)\bar{h}_a(u)}{u-v},
\label{eq:Dr:htahta}\\
&\left[h_a(u),e_b(v)\right]
=\left[\bar{h}_a(u),e_b(v)\right]=\left[h_a(u),\bar{e}_b(v)\right]=\left[\bar{h}_a(u),\bar{e}_b(v)\right]=0, 
\qquad\text{if }a\neq b,b+1,\label{eq:Dr:haeb}\\
&\left[h_a(u),e_a(v)\right]=\left[\bar{h}_a(-u),\bar{e}_a(v)\right]
=h_a(u)\frac{e_a(v)-e_a(u)}{u-v}+\bar{h}_a(-u)\frac{\bar{e}_a(v)-\bar{e}_a(u)}{u-v},
\label{eq:Dr:haea}\\
&\left[h_a(u),\bar{e}_a(v)\right]
=-\left[\bar{h}_a(-u),e_a(v)\right]
=h_a(u)\frac{\bar{e}_a(v)-\bar{e}_a(-u)}{u+v}+\bar{h}_a(-u)\frac{e_a(v)-e_a(-u)}{u+v},
\label{eq:Dr:haeta}\\
&\left[\bar{h}_a(u),e_a(v)\right]
=\bar{h}_{a}(u)\frac{e_a(v)-e_a(u)}{u-v}+h_a(-u)\frac{\bar{e}_a(v)-\bar{e}_a(u)}{u-v},
\label{eq:Dr:htaea}\\
&\left[h_{a+1}(u),e_a(v)\right]
=h_{a+1}(u)\frac{e_a(u)-e_a(v)}{u-v}+\bar{h}_{a+1}(-u)\frac{\bar{e}_a(u)-\bar{e}_a(-v)}{u+v},
\label{eq:Dr:ha1ea}\\
&\left[h_{a+1}(u),\bar{e}_a(v)\right]
=h_{a+1}(u)\frac{\bar{e}_a(u)-\bar{e}_a(v)}{u-v}+\bar{h}_{a+1}(-u)\frac{e_a(u)-e_a(-v)}{u+v},
\label{eq:Dr:ha1eta}\\
&\left[\bar{h}_{a+1}(u),e_a(v)\right]
=\bar{h}_{a+1}(u)\frac{e_a(u)-e_a(v)}{u-v}-h_{a+1}(-u)\frac{\bar{e}_a(u)-\bar{e}_a(-v)}{u+v},
\label{eq:Dr:hta1ea}\\
&\left[\bar{h}_{a+1}(u),\bar{e}_a(v)\right]
=\bar{h}_{a+1}(u)\frac{\bar{e}_a(u)-\bar{e}_a(v)}{u-v}-h_{a+1}(-u)\frac{e_a(u)-e_a(-v)}{u+v},
\label{eq:Dr:hta1eta}\\
&\left[h_a(u),f_b(v)\right]
=\left[h_a(u),\bar{f}_b(v)\right]=\left[\bar{h}_a(u),f_b(v)\right]=\left[\bar{h}_a(u),\bar{f}_b(v)\right]=0, 
\qquad\text{ if }a\neq b,b+1, \label{eq:Dr:hafb}\\
&\left[h_a(u),f_a(v)\right]=-\left[\bar{h}_a(u),\bar{f}_a(-v)\right]
=\frac{f_a(u)-f_a(v)}{u-v}h_a(u)
-\frac{\bar{f}_a(-u)-\bar{f}_a(-v)}{u-v}\bar{h}_a(u),
\label{eq:Dr:hafa}\\
&\left[h_a(u),\bar{f}_a(v)\right]
=-\left[\bar{h}_a(u),f_a(-v)\right]
=\frac{\bar{f}_a(u)-\bar{f}_a(v)}{u-v}h_a(u)
+\frac{f_a(-u)-f_a(-v)}{u-v}\bar{h}_a(u),
\label{eq:Dr:hafta}\\
&\left[h_{a+1}(u),f_a(v)\right]
=\frac{f_a(v)-f_a(u)}{u-v}h_{a+1}(u)-\frac{\bar{f}_a(-u)-\bar{f}_a(v)}{u+v}\bar{h}_{a+1}(u),
\label{eq:Dr:ha1fa}\\
&\left[h_{a+1}(u),\bar{f}_a(v)\right]
=\frac{f_a(v)-f_a(u)}{u-v}\bar{h}_{a+1}(u)
-\frac{\bar{f}_a(-u)-\bar{f}_a(v)}{u+v}h_{a+1}(u),
\label{eq:Dr:ha1fta}\\
&\left[\bar{h}_{a+1}(u),f_a(v)\right]
=\frac{\bar{f}_a(v)-\bar{f}_a(u)}{u-v}h_{a+1}(u)
+\frac{f_a(-u)-f_a(v)}{u+v}\bar{h}_{a+1}(u),
\label{eq:Dr:hta1fa}\\
&\left[\bar{h}_{a+1}(u),\bar{f}_a(v)\right]
=\frac{\bar{f}_a(v)-\bar{f}_a(u)}{u-v}\bar{h}_{a+1}(u)+\frac{f_a(-u)-f_a(v)}{u+v}h_{a+1}(u),
\label{eq:Dr:hta1fta}\\
&\left[e_a(u), f_b(v)\right]=\left[e_a(u),\bar{f}_b(v)\right]
=\left[\bar{e}_a(u), f_b(v)\right]=\left[\bar{e}_a(u),\bar{f}_b(v)\right]=0,\qquad\text{ if }a\neq b,
\label{eq:Dr:eafb}\\
&\left[e_a(u),f_a(v)\right]
=\frac{h_a^{\prime}(u)h_{a+1}(u)-h_{a+1}(v)h_a^{\prime}(v)}{u-v}
+\frac{\bar{h}_a^{\prime}(-u)\bar{h}_{a+1}(u)+\bar{h}_{a+1}(-v)\bar{h}_a^{\prime}(v)}{u+v},
\label{eq:Dr:eafa}\\
&\left[e_a(u),\bar{f}_a(v)\right]
=\frac{h_a^{\prime}(u)\bar{h}_{a+1}(u)-\bar{h}_{a+1}(v)h_a^{\prime}(v)}{u-v}
+\frac{\bar{h}_a^{\prime}(-u)h_{a+1}(u)-h_{a+1}(-v)\bar{h}_a^{\prime}(v)}{u+v},
\label{eq:Dr:eafta}\\
&\left[\bar{e}_a(u),f_a(v)\right]
=\frac{\bar{h}_a^{\prime}(u)h_{a+1}(u)-h_{a+1}(v)\bar{h}_a^{\prime}(v)}{u-v}
-\frac{h_a^{\prime}(-u)\bar{h}_{a+1}(u)-\bar{h}_{a+1}(-v)h_a^{\prime}(v)}{u+v},
\label{eq:Dr:etafa}\\
&\left[\bar{e}_a(u),\bar{f}_a(v)\right]
=\frac{\bar{h}_a^{\prime}(u)\bar{h}_{a+1}(u)+\bar{h}_{a+1}(v)\bar{h}_a^{\prime}(v)}{u-v}
-\frac{h_a^{\prime}(-u)h_{a+1}(u)-h_{a+1}(-v)h_a^{\prime}(v)}{u+v},
\label{eq:Dr:etafta}\\
&\left[e_a(u),e_b(v)\right]
=\left[e_a(u),\bar{e}_b(v)\right]=\left[\bar{e}_a(u),e_b(v)\right]=\left[\bar{e}_a(u),\bar{e}_b(v)\right]=0,\qquad \text{ if }|a-b|>1,
\label{eq:Dr:eaeb}\\
&\left[e_a(u), e_a(v)\right]
=\frac{\left(e_a(v)-e_a(u)\right)^2}{u-v}
-\frac{\left(\bar{e}_a(-v)-\bar{e}_a(u)\right)\left(\bar{e}_a(v)-\bar{e}_a(-u)\right)}{u+v},
\label{eq:Dr:eaea}\\
&\left[e_a(u),\bar{e}_a(v)\right]
=\frac{\left(e_a(u)-e_a(v)\right)\left(\bar{e}_a(u)-\bar{e}_a(v)\right)}{u-v}
+\frac{\left(\bar{e}_a(-u)-\bar{e}_a(v)\right)\left(e_a(u)-e_a(-v)\right)}{u+v},
\label{eq:Dr:eaeta}\\
&\left[\bar{e}_a(u),\bar{e}_a(v)\right]
=-\frac{\left(\bar{e}_a(v)-\bar{e}_a(u)\right)^2}{u-v}-\frac{\left(e_a(-v)-e_a(u)\right)\left(e_a(v)-e_a(-u)\right)}{u+v}.
\label{eq:Dr:etaeta}\\
&\left[e_a(u),\ e_{a+1}(v)\right]=\left[\bar{e}_{a}(-u),\ \bar{e}_{a+1}(v)\right],
\label{eq:Dr:eaea1}\\
&\left[e_a(u),\ \bar{e}_{a+1}(v)\right]=-\left[\bar{e}_a(-u),e_{a+1}(v)\right],
\label{eq:Dr:eaeta1}\\
&u\left[\mathring{e}_a(u), e_{a+1}(v)\right]
-v\left[e_a(u), \mathring{e}_{a+1}(v)\right]=e_a(u)e_{a+1}(v)-\bar{e}_a(-u)\bar{e}_{a+1}(v),
\label{eq:Dr:ecircle}\\
&u\left[\mathring{\bar{e}}_a(u), e_{a+1}(v)\right]
-v\left[\bar{e}_a(u), \mathring{e}_{a+1}(v)\right]=\bar{e}_a(u)e_{a+1}(v)+e_a(-u)\bar{e}_{a+1}(v),
\label{eq:Dr:etcircle}\\
&\left[f_a(u),f_b(v)\right]
=\left[f_a(u),\bar{f}_b(v)\right]=\left[\bar{f}_a(u),f_b(v)\right]=\left[\bar{f}_a(u),\bar{f}_b(v)\right]=0, \qquad\text{ if }|a-b|>1,
\label{eq:Dr:fafb}\\
&\left[f_a(u), f_a(v)\right]
=-\frac{\left(f_a(v)-f_a(u)\right)^2}{u-v}
+\frac{\left(\bar{f}_a(-v)-\bar{f}_a(u)\right)\left(\bar{f}_a(v)-\bar{f}_a(-u)\right)}{u+v},
\label{eq:Dr:fafa}\\
&\left[f_a(u),\bar{f}_a(v)\right]
=-\frac{\left(\bar{f}_a(-u)-\bar{f}_a(v)\right)\left(f_a(u)-f_a(-v)\right)}{u+v}
-\frac{\left(f_a(u)-f_a(v)\right)\left(\bar{f}_a(u)-\bar{f}_a(v)\right)}{u-v},
\label{eq:Dr:fafta}\\
&\left[\bar{f}_a(u),\bar{f}_a(v)\right]
=\frac{\left(\bar{f}_a(v)-\bar{f}_a(u)\right)^2}{u-v}+\frac{\left(f_a(-v)-f_a(u)\right)\left(f_a(v)-f_a(u)\right)}{u+v},
\label{eq:Dr:ftafta}\\
&\left[f_{a}(u),f_{a+1}(v)\right]=-\left[\bar{f}_{a}(u),\ \bar{f}_{a+1}(-v)\right],
\label{eq:Dr:fafa1}\\
&\left[f_{a}(u),\ \bar{f}_{a+1}(v)\right]=-\left[\bar{f}_{a}(u),f_{a+1}(-v)\right],
\label{eq:Dr:ftafa1}\\
&u\left[\mathring{f}_a(u), f_{a+1}(v)\right]
-v\left[f_a(u), \mathring{f}_{a+1}(v)\right]=-f_{a+1}(v)f_a(u)+\bar{f}_{a+1}(-v)\bar{f}_a(u),
\label{eq:Dr:fcircle}\\
&u\left[\mathring{\bar{f}}_a(u), f_{a+1}(v)\right]
+v\left[\bar{f}_a(u), \mathring{f}_{a+1}(v)\right]=f_{a+1}(v)\bar{f}_a(u)+\bar{f}_{a+1}(-v)f_a(u),
\label{eq:Dr:ftcirlce}\\
&\left[e_a(u), \left[e_a(u), e_b(v)\right]\right]-\left[\bar{e}_a(-u), \left[\bar{e}_a(u), e_b(v)\right]\right]=0, \qquad
\text{ if }  |a-b|=1,\label{eq:Dr:eserre1}\\
&\left[e_a(-u), \left[e_a(u), e_b(v)\right]\right]+\left[\bar{e}_a(u), \left[\bar{e}_a(u), e_b(v)\right]\right]=0,\qquad
\text{ if }  |a-b|=1,\label{eq:Dr:eserre2}\\
&\left[e_a(-u), \left[\bar{e}_a(u), e_b(v)\right]\right]+\left[\bar{e}_a(u), \left[e_a(u), e_b(v)\right]\right]=0,\qquad
\text{ if }  |a-b|=1,\label{eq:Dr:eserre3}\\
&\left[e_a(u), \left[\bar{e}_a(u), e_b(v)\right]\right]-\left[\bar{e}_a(-u),\left[e_a(u), e_b(v)\right]\right]=0,\qquad
\text{ if }  |a-b|=1,\label{eq:Dr:eserre4}\\
&\left[f_a(u), \left[f_a(u), f_b(v)\right]\right]+\left[\bar{f}_a(u),\left[\bar{f}_a(-u), f_b(v)\right]\right]=0,\qquad
\text{ if }  |a-b|=1,\label{eq:Dr:fserre1}\\
&\left[f_a(u),\left[f_a(-u), f_b(v)\right]\right]-\left[\bar{f}_a(u), \left[\bar{f}_a(u), f_b(v)\right]\right]=0,\qquad
\text{ if }  |a-b|=1,\label{eq:Dr:fserre2}\\
&\left[f_a(u),\left[\bar{f}_a(u), f_b(v)\right]\right]+\left[\bar{f}_a(u), \left[f_a(-u), f_b(v)\right]\right]=0,\qquad
\text{ if }  |a-b|=1,\label{eq:Dr:fserre3}\\
&\left[f_a(u), \left[\bar{f}_a(-u), f_b(v)\right]\right]-\left[\bar{f}_a(u), \left[f_a(u), f_b(v)\right]\right]=0,
\qquad\text{ if }  |a-b|=1.\label{eq:Dr:fserre4}
\end{align}
\end{theorem}
\begin{proof}
Let $\YD(\mathfrak{q}_n)$  be the associative superalgebra defined by the generators and relations as stated in the theorem, while $\Y(\mathfrak{q}_n)$ denotes the R-matrix presented super-Yangian as before. Theorem~\ref{thm:generator} demonstrates that $\Y(\mathfrak{q}_n)$ is generated by the coefficients of $h_a(u), e_b(u), f_b(u)$ and $\bar{h}_a(u), \bar{e}_b(u), \bar{f}_b(u)$ for $a=1,2,\ldots, n$ and $b=1,2,\ldots, n-1$, which are obtained by the Gauss decomposition~\eqref{eq:Gaussfac} of the generator matrix $T(u)$ in \eqref{eq:qnTu}. We have verified in Sections~\ref{se:q2relation} and ~\ref{se:serrerelation} that these generators of $\Y(\mathfrak{q}_n)$ satisfy all relations \eqref{eq:Dr:hahb}-\eqref{eq:Dr:fserre4}. Hence, there exists a canonical surjective homomorphism $\theta:\YD(\mathfrak{q}_n)\rightarrow\Y(\mathfrak{q}_n)$ that sends the generators of $\YD(\mathfrak{q}_n)$ to the corresponding elements in $\Y(\mathfrak{q}_n)$. 

Next, we show $\theta$ is an isomorphism. We introduce a loop filtration on $\YD(\mathfrak{q}_n)$ by declaring the generators $h_a^{(r)}, e_b^{(r)}, f_b^{(r)}, \bar{h}_a^{(r)}, \bar{e}_b^{(r)}, \bar{f}_b^{(r)}$ to be of degree $r-1$. Let $\mathrm{L}_r\YD(\mathfrak{q}_n)$ (resp. $\mathrm{L}_r\Y(\mathfrak{q}_n)$) be the subspace of  $\YD(\mathfrak{q}_n)$ (resp. $\Y(\mathfrak{q}_n)$) spanned by elements of degree at most $r$. It is easy to observe that $$\theta\left(\mathrm{L}_r\YD(\mathfrak{q}_n)\right)\subseteq\mathrm{L}_r\Y(\mathfrak{q}_n)$$ 
for all $r\geqslant0$. Hence, $\theta$ induces a surjective homomorphism of the corresponding graded algebras:
$$\mathrm{gr}^{\prime}\theta:\mathrm{gr}^{\prime}\YD(\mathfrak{q}_n)\rightarrow
\mathrm{gr}^{\prime}\Y(\mathfrak{q}_n),$$
such that the following diagram commutes:
$$\xymatrix{
\YD(\mathfrak{q}_n)
\ar[r]^{\theta}\ar[d]_{\widehat{\pi}}&
\Y(\mathfrak{q}_n)
\ar[d]^{\pi}\\
\mathrm{gr}^{\prime}\YD(\mathfrak{q}_n)
\ar[r]_{\mathrm{gr}^{\prime}\theta}&
\mathrm{gr}^{\prime}\Y(\mathfrak{q}_n)
}$$
where $\widehat{\pi}$ and $\pi$ are the canonical surjective homomorphisms.

It suffices to show that $\mathrm{gr}^{\prime}\theta$ is an isomorphism. Since there is an isomorphism $\psi: \mathrm{gr}^{\prime}\Y(\mathfrak{q}_n)\rightarrow \U(\mathfrak{q}_n^{tw})$ as shown in Theorem~\ref{thm:Iso}, we shall find a surjective homomorphism $\widehat{\psi}:\U(\mathfrak{q}_n^{tw})\rightarrow\mathrm{gr}^{\prime}\YD(\mathfrak{q}_n)$ such that the composition
$$\U(\mathfrak{q}_n^{tw})\xrightarrow{\widehat{\psi}}
\mathrm{gr}^{\prime}\YD(\mathfrak{q}_n)\xrightarrow{\mathrm{gr}^{\prime}\theta}
\mathrm{gr}^{\prime}\Y(\mathfrak{q}_n)\xrightarrow{\psi}
\U(\mathfrak{q}_n^{tw})$$
is the identity map, which ensures that $\mathrm{gr}^{\prime}\theta$ is an isomorphism.

Let $\YD^0(\mathfrak{q}_n)$ be the subalgebra of $\YD(\mathfrak{q}_n)$ generated by $h_a^{(r)}, \bar{h}_a^{(r)}$ for $a=1,2,\ldots,n$ and $r\geqslant1$.  Let $\YD^+(\mathfrak{q}_n)$ (resp. $\YD^-(\mathfrak{q}_n)$) denote the subalgebras of $\YD(\mathfrak{q}_n)$ generated by $e_b^{(r)}, \bar{e}_b^{(r)}$ (resp. $f_b^{(r)}, \bar{f}_b^{(r)}$) for $b=1,2,\ldots,n-1$ and $r\geqslant1$. Similarly, we denote by $\U^0(\mathfrak{q}_n^{tw})$, $\U^+(\mathfrak{q}_n^{tw})$ and $\U^-(\mathfrak{q}_n^{tw})$ the subalgebras of $\U(\mathfrak{q}_n^{tw})$ generated by the entries of $\mathsf{G}_{ab}^{(r)}$ with $a=b$, $a<b$, and $a>b$, respectively. Then the PBW theorem for $\U(\mathfrak{q}_n^{tw})$ implies that the multiplication map
$$\U^-(\mathfrak{q}_n^{tw})\otimes \U^0(\mathfrak{q}_n^{tw})\otimes \U^+(\mathfrak{q}_n^{tw})\longrightarrow \U(\mathfrak{q}_n^{tw})$$
is an isomorphism. On the other hand, the relations \eqref{eq:Dr:haeb}-\eqref{eq:Dr:etafta} shows that the multiplication map
$$\YD^-(\mathfrak{q}_n)\otimes\YD^0(\mathfrak{q}_n)\otimes\YD^+(\mathfrak{q}_n)\longrightarrow\YD(\mathfrak{q}_n)$$
is surjective. Then
$$\mathrm{gr}^{\prime}\YD^-(\mathfrak{q}_n)\otimes\mathrm{gr}^{\prime}\YD^0(\mathfrak{q}_n)\otimes\mathrm{gr}^{\prime}\YD^+(\mathfrak{q}_n)\longrightarrow\mathrm{gr}^{\prime}\YD(\mathfrak{q}_n)$$
is also surjective.

We first consider the subalgebra $\mathrm{gr}^{\prime}\YD^0(\mathfrak{q}_n)$. 
Comparing the coefficients of $u^{-r}v^{-s}$ in both sides of \eqref{eq:Dr:haha}, we obtain
\begin{align*}
&\left[h_a^{(1)}, h_a^{(s)}\right]=0,\\
&\left[h_a^{(2)}, h_a^{(s)}\right]=\bar{h}_a^{(1)}\bar{h}_a^{(s)}+(-1)^s\bar{h}_a^{(s)}\bar{h}_a^{(1)},\\
&\left[h_a^{(r+2)},h_a^{(s)}\right]
-\left[h_a^{(r)},h_a^{(s+2)}\right]
-\left[h_a^{(r+1)},h_a^{(s)}\right]
-\left[h_a^{(r)},h_a^{(s+1)}\right]\\
=&\left(\bar{h}_a^{(r+1)}\bar{h}_a^{(s)}+(-1)^{r+s}\bar{h}_a^{(r)}\bar{h}_a^{(s+1)}\right)-\left(\bar{h}_a^{(r)}\bar{h}_a^{(s+1)}+(-1)^{r+s}\bar{h}_a^{(s+1)}\bar{h}_a^{(r)}\right),
\end{align*}
which hold in $\YD^0(\mathfrak{q}_n)$. They yield the following identities in $\mathrm{gr}^{\prime}\YD^0(\mathfrak{q}_n)$:
\begin{equation}
\left[\widehat{\pi}\left(h_a^{(r+1)}\right), \widehat{\pi}\left(h_a^{(s+1)}\right)\right]=0,\quad r,s\geqslant0. \label{eq:grhat:hh}
\end{equation}

Comparing the coefficients of $u^{-r}v^{-s}$ on both sides of \eqref{eq:Dr:hahta}, we obtain
\begin{align*}
&\left[h_a^{(1)}, \bar{h}_a^{(s)}\right]=0,\\
&\left[h_a^{(2)}, \bar{h}_a^{(s)}\right]=-2\bar{h}_a^{(s+1)}
+h_a^{(1)}\bar{h}_a^{(s)}-(-1)^sh_a^{(s)}\bar{h}_a^{(1)}+\bar{h}_a^{(1)}h_a^{(s)}-\bar{h}_a^{(s)}h_a^{(1)},\\
&\left[h_a^{(r+2)},\bar{h}_a^{(s)}\right]
-\left[h_a^{(r)},\bar{h}_a^{(s+2)}\right]=h_a^{(r+1)}\bar{h}_a^{(s)}-h_a^{(r)}\bar{h}_a^{(s+1)}+(-1)^{r+s}\left(h_a^{(s+1)}\bar{h}_a^{(r)}-h_a^{(s)}\bar{h}_a^{(r+1)}\right)\\
&\qquad\qquad\qquad\quad\qquad\qquad\quad\quad\quad +\bar{h}_a^{(r+1)}h_a^{(s)}-\bar{h}_a^{(s)}h_a^{(r+1)}+\bar{h}_a^{(r)}h_a^{(s+1)}-\bar{h}_a^{(s+1)}h_a^{(r)},
\end{align*}
which hold in $\YD^0(\mathfrak{q}_n)$. They yield the following identities in $\mathrm{gr}^{\prime}\YD^0(\mathfrak{q}_n)$:
\begin{equation}
\left[\widehat{\pi}\left(h_a^{(r+1)}\right), \widehat{\pi}\left(\bar{h}_a^{(s+1)}\right)\right]=\left((-1)^r-1\right)\widehat{\pi}\left(\bar{h}_a^{(r+s+1)}\right),\quad r,s\geqslant0. \label{eq:grhat:hht}
\end{equation}
A similar consideration on \eqref{eq:Dr:htahta} yields that
\begin{equation}
\left[\widehat{\pi}\left(\bar{h}_a^{(r+1)}\right), \widehat{\pi}\left(\bar{h}_a^{(s+1)}\right)\right]=\left((-1)^r+(-1)^s\right)\widehat{\pi}\left(h_a^{(r+s+1)}\right),
\quad r,s\geqslant0.
\label{eq:grhat:htht}
\end{equation}
According to the relations \eqref{eq:grhat:hh},~\eqref{eq:grhat:hht} and ~\eqref{eq:grhat:htht}, there is a surjective homomorphism $\widehat{\psi}^0: \U^0(\mathfrak{q}_n^{tw})\rightarrow\mathrm{gr}^{\prime}\YD^0(\mathfrak{q}_n)$ such that
$$\widehat{\psi}^0\left(\mathsf{g}_{a,a}^{(r)}\right)=\widehat{\pi}\left(h_a^{(r+1)}\right),\quad \widehat{\psi}^0\left(\mathsf{g}_{-a,a}^{(r)}\right)=\widehat{\pi}\left(\bar{h}_a^{(r+1)}\right),\quad a=1,2,\ldots,n.$$

Next, we consider the subalgebra $\mathrm{gr}^{\prime}\YD^+(\mathfrak{q}_n)$. In order to find a homomorphism from $\U^+\left(\mathfrak{q}_n^{tw}\right)$ to $\mathrm{gr}^{\prime}\YD(\mathfrak{q}_n)$, we inductively define the following elements in $\YD(\mathfrak{q}_n)$: 
\begin{align*}
e_{a,a+1}^{(r)}:=&e_a^{(r)},
&e_{a,b}^{(r)}:=&\left[e_{a,b-1}^{(r)},e_{b-1}^{(1)}\right], \text{ if }b>a+1,\\
\bar{e}_{a,a+1}^{(r)}:=&\bar{e}_a^{(r)},
&\bar{e}_{a,b}^{(r)}:=&\left[\bar{e}_{a,b-1}^{(r)},e_{b-1}^{(1)}\right], \text{ if }b>a+1,
\end{align*}
for $r\geqslant1$. Based on the relations~\eqref{eq:Dr:eaeb}-\eqref{eq:Dr:eserre4}, it can be verified that 
\begin{align}
\left[\widehat{\pi}\left(e_{ab}^{(r+1)}\right),
\widehat{\pi}\left(e_{cd}^{(s+1)}\right)\right]
=&\delta_{bc}\widehat{\pi}\left(e_{ad}^{(r+s+1)}\right)
-\delta_{ad}\widehat{\pi}\left(e_{cb}^{(r+s+1)}\right),
\label{eq:grhat:ee}\\
\left[\widehat{\pi}\left(e_{ab}^{(r+1)}\right),
\widehat{\pi}\left(\bar{e}_{cd}^{(s+1)}\right)\right]
=&\delta_{bc}(-1)^r\widehat{\pi}\left(\bar{e}_{ad}^{(r+s+1)}\right)
-\delta_{ad}\widehat{\pi}\left(\bar{e}_{cb}^{(r+s+1)}\right),
\label{eq:grhat:eet}\\
\left[\widehat{\pi}\left(\bar{e}_{ab}^{(r+1)}\right),
\widehat{\pi}\left(\bar{e}_{cd}^{(s+1)}\right)\right]
=&\delta_{bc}(-1)^r\widehat{\pi}\left(e_{ad}^{(r+s+1)}\right)
-\delta_{ad}(-1)^s\widehat{\pi}\left(e_{cb}^{(r+s+1)}\right),
\label{eq:grhat:etet}
\end{align}
for $1\leqslant a<b\leqslant n$, $1\leqslant c<d\leqslant n$ and $r,s\geqslant0$. The details for verifying \eqref{eq:grhat:ee} and \eqref{eq:grhat:eet} are given in Lemmas~\ref{lem:grYD} and Lemmas~\ref{lem:grYD1} below. Then \eqref{eq:grhat:etet} follows from \eqref{eq:Dr:eaea1} and \eqref{eq:grhat:ee}. Hence, there is a surjective homomorphism $\widehat{\psi}^+: \U^+(\mathfrak{q}_n^{tw})\rightarrow\mathrm{gr}^{\prime}\YD^+(\mathfrak{q}_n)$ such that 
$$\widehat{\psi}^+\left({\mathsf{g}_{ab}^{(r)}}\right)=\widehat{\pi}\left(e_{ab}^{(r+1)}\right), 
\text{ and }\widehat{\psi}^+\left({\mathsf{g}_{-a,b}^{(r)}}\right)=\widehat{\pi}\left(\bar{e}_{ab}^{(r+1)}\right),\quad 
1\leqslant a<b\leqslant n.$$
By Proposition \ref{antiinvolutionomega}, we know that there is a surjective homomorphism $\widehat{\psi}^-: \U^-(\mathfrak{q}_n^{tw})\rightarrow\mathrm{gr}^{\prime}\YD^-(\mathfrak{q}_n)$ such that 
$$\widehat{\psi}^-\left({\mathsf{g}_{b+1,b}^{(r)}}\right)=\widehat{\pi}\left(f_b^{(r+1)}\right), 
\text{ and }\widehat{\psi}^-\left({\mathsf{g}_{-b-1,b}^{(r)}}\right)=\widehat{\pi}\left(\bar{f}_b^{(r+1)}\right),\quad 
b=1,2,\ldots,n-1.$$

Combining $\widehat{\psi}^0, \widehat{\psi}^+, \widehat{\psi}^-$, we obtain a surjective homomorphism $\widehat{\psi}: \U(\mathfrak{q}_n^{tw})\rightarrow\mathrm{gr}^{\prime}\YD(\mathfrak{q}_n)$ such that $\psi\circ\mathrm{gr}^{\prime}\circ\widehat{\psi}$ is the identity map on $\U(\mathfrak{q}^{tw})$. This completes the proof.
\end{proof}

\begin{lemma}
\label{lem:grYD}
The identity \eqref{eq:grhat:ee} holds in $\mathrm{gr}^{\prime}\YD(\mathfrak{q}_n)$.
\end{lemma}
\begin{proof}
Since the bracket is antisymmetric, we may assume that $a\leqslant c$. Then it suffices to verify \eqref{eq:grhat:ee} in the following eight cases.

\medskip

{\bf Case 1:} $a<b<c<d$. The identity
$\left[\widehat{\pi}\left(e_{ab}^{(r+1)}\right), 
\widehat{\pi}\left(e_{cd}^{(s+1)}\right)\right]
=0$ in $\mathrm{gr}^{\prime}\YD(\mathfrak{q}_n)$ follows from the identity  $\left[e_{ab}^{(r+1)}, e_{cd}^{(s+1)}\right]=0$ by \eqref{eq:Dr:eaeb}.

\medskip

{\bf Case 2:} $a<b=c<d$. If $a=b-1<b=c<b+1=d$, we compare the coefficients of $u^{-r}v^{-s}$ in \eqref{eq:Dr:ecircle} and obtain the following relation in $\YD(\mathfrak{q}_n)$:
$$\left[e_{b-1}^{(r+1)},e_b^{(s)}\right]
-\left[e_{b-1}^{(r)},e_b^{(s+1)}\right]
=e_{b-1}^{(r)}e_b^{(s)}-(-1)^r\bar{e}_{b-1}^{(r)}\bar{e}_b^{(s)}.$$
It yields that
$$\left[\widehat{\pi}\left(e_{b-1}^{(r+1)}\right),\widehat{\pi}\left(e_b^{(s+1)}\right)\right]
=\left[\widehat{\pi}\left(e_{b-1}^{(r+2)}\right),\widehat{\pi}\left(e_b^{(s)}\right)\right]
=\left[\widehat{\pi}\left(e_{b-1}^{(r+s+1)}\right),\widehat{\pi}\left(e_b^{(1)}\right)\right]
=\widehat{\pi}\left(e_{b-1,b+1}^{(r+s+1)}\right).$$

If $a<a+1<b=c<b+1=d$. By Case 1, we compute that
\begin{align*}
\left[\widehat{\pi}\left(e_{ab}^{(r+t+1)}\right),
\widehat{\pi}\left(e_{b}^{(s+1)}\right)\right]
=&\left[\left[\widehat{\pi}\left(e_{a,b-1}^{(r+1)}\right),
\widehat{\pi}\left(e_{b-1}^{(t+1)}\right)\right],
\widehat{\pi}\left(e_b^{(s+1)}\right)\right]\\
=&\left[\widehat{\pi}\left(e_{a,b-1}^{(r+1)}\right),\left[
\widehat{\pi}\left(e_{b-1}^{(t+1)}\right),
\widehat{\pi}\left(e_b^{(s+1)}\right)\right]\right]\\
=&\left[\widehat{\pi}\left(e_{a,b-1}^{(r+1)}\right),
\widehat{\pi}\left(e_{b-1,b+1}^{(s+t+1)}\right)\right].
\end{align*}
It follows that
\begin{align*}
\left[\widehat{\pi}\left(e_{ab}^{(r+1)}\right),
\widehat{\pi}\left(e_{b}^{(s+1)}\right)\right]
=&\left[\widehat{\pi}\left(e_{a,b-1}^{(r+1)}\right),
\widehat{\pi}\left(e_{b-1,b+1}^{(s+1)}\right)\right]\\
=&\left[\widehat{\pi}\left(e_{ab}^{(r+s+1)}\right),
\widehat{\pi}\left(e_{b}^{(1)}\right)\right]
=\widehat{\pi}\left(e_{a,b+1}^{(r+s+1)}\right).
\end{align*}

In general, if $b+1<d$, we conclude by induction on $d-b$ that
\begin{align*}
\left[\widehat{\pi}\left(e_{ab}^{(r+1)}\right),
\widehat{\pi}\left(e_{bd}^{(s+1)}\right)\right]
=\widehat{\pi}\left(e_{ad}^{(r+s+1)}\right).
\end{align*}

\medskip

{\bf Case 3: } $d=c+1$ and $c<b$. We have to show that
\begin{equation}\label{eq:inj:eabec}
\left[\widehat{\pi}\left(e_{ab}^{(r+1)}\right),\widehat{\pi}\left(e_c^{(s+1)}\right)\right]
=0,
\end{equation}
provided that $a\leqslant c$ and $c<b$. We consider the following subcases. 

Subcase 3-1: $a=c$ and $b=c+1$. In this situation, we have
\begin{equation}
\left[\widehat{\pi}\left(e_{c}^{(r+1)}\right),
\widehat{\pi}\left(e_{c}^{(s+1)}\right)\right]=0,
\label{eq:inj:ecec}
\end{equation}
by comparing the coefficients of $u^{-r}v^{-s}$ on both sides of \eqref{eq:Dr:eaea}.

Subcase 3-2: $a<c$ and $b=c+1$. By comparing the coefficients of $u^{-2}v^{-r-1}$ on both sides of the Serre relations \eqref{eq:Dr:eserre1} and \eqref{eq:Dr:eserre2}, we obtain
\begin{align*}
&\left[\left[e_{c-1}^{(r+1)}, e_c^{(1)}\right], e_c^{(1)}\right]-\left[\left[e_{c-1}^{(r+1)}, \bar{e}_c^{(1)}\right], \bar{e}_c^{(1)}\right]=0,\\
&\left[\left[e_{c-1}^{(r+1)}, e_c^{(1)}\right], e_c^{(1)}\right]+\left[\left[e_{c-1}^{(r+1)}, \bar{e}_c^{(1)}\right], \bar{e}_c^{(1)}\right]=0,
\end{align*}
which yields that $\left[\left[e_{c-1}^{(r+1)}, e_c^{(1)}\right], e_c^{(1)}\right]=0$. Hence,
$$\left[\widehat{\pi}\left(e_{c-1,c+1}^{(r+1)}\right),\widehat{\pi}\left(e_c^{(1)}\right)\right]=\left[\left[\widehat{\pi}\left(e_{c-1}^{(r+1)}\right), \widehat{\pi}\left(e_c^{(1)}\right)\right], \widehat{\pi}\left(e_c^{(1)}\right)\right]=0,$$
for $r\geqslant0$. If $s>0$, it follows from Case 2 and \eqref{eq:inj:ecec} that 
\begin{align*}
\left[\widehat{\pi}\left(e_{c-1,c+1}^{(r+1)}\right), \widehat{\pi}\left(e_c^{(s+1)}\right)\right]
=&\left[\left[\widehat{\pi}\left(e_{c-1}^{(r+1)}\right), \widehat{\pi}\left(e_c^{(1)}\right)\right], \widehat{\pi}\left(e_c^{(s+1)}\right)\right]\\
=&\left[ \widehat{\pi}\left(e_{c-1,c+1}^{(r+s+1)}\right), \widehat{\pi}\left(e_c^{(1)}\right)\right]=0.
\end{align*}
If $a<c-1$, 
$$\left[\widehat{\pi}\left(e_{a,c+1}^{(r+1)}\right), \widehat{\pi}\left(e_c^{(s+1)}\right)\right]\\
=\left[\left[\widehat{\pi}\left(e_{a,c-1}^{(r+1)}\right), \widehat{\pi}\left(e_{c-1,c+1}^{(1)}\right)\right], \widehat{\pi}\left(e_c^{(s+1)}\right)\right]=0,$$
since $\left[\widehat{\pi}\left(e_{a,c-1}^{(r+1)}\right), \widehat{\pi}\left(e_c^{(s+1)}\right)\right]=0$ by Case 1, and $\left[\widehat{\pi}\left(e_{c-1,c+1}^{(1)}\right), \widehat{\pi}\left(e_c^{(s+1)}\right)\right]=0$ as we proved above.

Subcase 3-3: $a=c$ and $b>c+1$. If $b=c+2$, a similar argument as in Subcase 3-2 shows that
\begin{align*}
\left[\widehat{\pi}\left(e_{c,c+2}^{(r+1)}\right),\widehat{\pi}\left(e_{c}^{(s+1)}\right)\right]
=&\left[\left[\widehat{\pi}\left(e_{c}^{(1)}\right),\widehat{\pi}\left(e_{c+1}^{(s+1)}\right)\right],\widehat{\pi}\left(e_{c}^{(s+1)}\right)\right]\\
=&-\left[\widehat{\pi}\left(e_{c}^{(1)}\right),\left[\widehat{\pi}\left(e_{c}^{(1)}\right),\widehat{\pi}\left(e_{c+1}^{(r+s+1)}\right)\right]\right]=0.
\end{align*}
For $b>c+2$, it follows from Case 1 that
\begin{align}\label{eq:inj:ecbec}
    \left[\widehat{\pi}\left(e_{cb}^{(r+1)}\right),\widehat{\pi}\left(e_{c}^{(s+1)}\right)\right]
    =&\left[\left[\widehat{\pi}\left(e_{c,c+2}^{(r+1)}\right),
    \widehat{\pi}\left(e_{c+2,b}^{(1)}\right)\right],
    \widehat{\pi}\left(e_{c}^{(s+1)}\right)\right]\nonumber\\
    =&\left[\left[\widehat{\pi}\left(e_{c,c+2}^{(r+1)}\right),
    \widehat{\pi}\left(e_{c}^{(s+1)}\right)\right],
    \widehat{\pi}\left(e_{c+2,b}^{(1)}\right)\right]=0.
\end{align}

Subcase 3-4: $a<c$ and $b>c+1$.
\begin{align*}
    \left[\widehat{\pi}\left(e_{ab}^{(r+1)}\right),\widehat{\pi}\left(e_{c}^{(s+1)}\right)\right]
    =&\left[\widehat{\pi}\left(e_{a,c+1}^{(r+1)}\right),
    \left[\widehat{\pi}\left(e_{c+1,b}^{(1)}\right),
    \widehat{\pi}\left(e_{c}^{(s+1)}\right)\right]\right]\\
    =&-\left[\widehat{\pi}\left(e_{a,c+1}^{(r+1)}\right),
    \widehat{\pi}\left(e_{c,b}^{(s+1)}\right)\right]\\
    =&-\left[\left[\widehat{\pi}\left(e_{a,c}^{(r+1)}\right),
    \widehat{\pi}\left(e_{c,b}^{(s+1)}\right)\right],
    \widehat{\pi}\left(e_{c}^{(1)}\right)\right]\\
    =&-  \left[\widehat{\pi}\left(e_{ab}^{(r+s+1)}\right),\widehat{\pi}\left(e_{c}^{(1)}\right)\right].
\end{align*}
Thus, we have $ \left[\widehat{\pi}\left(e_{ab}^{(r+1)}\right),\widehat{\pi}\left(e_{c}^{(1)}\right)\right]=0$. Consequently,  $ \left[\widehat{\pi}\left(e_{ab}^{(r+1)}\right),\widehat{\pi}\left(e_{c}^{(s+1)}\right)\right]=0$.
\medskip

{\bf Case 4:} Let $b=d$ and $d>c+1$. Since
\begin{align*}
\left[\widehat{\pi}\left(e_{ad}^{(r+1)}\right), \widehat{\pi}\left(e_{cd}^{(s+1)}\right)\right]
=&\left[\widehat{\pi}\left(e_{ad}^{(r+1)}\right),
\left[\widehat{\pi}\left(e_{c}^{(1)}\right),\widehat{\pi}\left(e_{c+1,d}^{(s+1)}\right)\right]\right]\\
=&
\left[\widehat{\pi}\left(e_{c}^{(1)}\right),\left[\widehat{\pi}\left(e_{ad}^{(r+1)}\right),\widehat{\pi}\left(e_{c+1,d}^{(s+1)}\right)\right]\right].
\end{align*}
Then an induction on $d-c$ starting with \eqref{eq:inj:eabec} yields that
\begin{align}\label{eq:inj:eadecd}
\left[\widehat{\pi}\left(e_{ad}^{(r+1)}\right), \widehat{\pi}\left(e_{cd}^{(s+1)}\right)\right]=0.
\end{align}
\medskip

{\bf Case 5:} Let $a=c$ and $b,d>a+1$. We shall show that
\begin{equation}
\left[\widehat{\pi}\left(e_{ab}^{(r+1)}\right), \widehat{\pi}\left(e_{ad}^{(s+1)}\right)\right]=0.
\label{eq:inj:eabead}
\end{equation}
We may assume $b<d$ and deduce by \eqref{eq:inj:eabec} that
\begin{align*}
\left[\widehat{\pi}\left(e_{ab}^{(r+1)}\right), \widehat{\pi}\left(e_{ad}^{(s+1)}\right)\right]
=&\left[\left[\widehat{\pi}\left(e_{a,b-1}^{(1)}\right),\widehat{\pi}\left(e_{b-1,b}^{(r+1)}\right)\right]
\widehat{\pi}\left(e_{ad}^{(s+1)}\right)\right]\\
=&\left[\left[\widehat{\pi}\left(e_{a,b-1}^{(1)}\right),\widehat{\pi}\left(e_{ad}^{(s+1)}\right)\right],
\widehat{\pi}\left(e_{b-1,b}^{(r+1)}\right)\right].
\end{align*}
Hence, \eqref{eq:inj:eabead} follows by induction on $b-a$. The case where $b-a=1$ is verified in \eqref{eq:inj:ecbec}.
\medskip

{\bf Case 6:} Let $a=c$, $b=d$, and $b>a+1$. Since the elements $\widehat{\pi}\left(e_{a,b-1}^{(1)}\right)$ and $\widehat{\pi}\left(e_{b-1,b}^{(r+1)}\right)$ commute with $\widehat{\pi}\left(e_{ab}^{(s+1)}\right)$ by \eqref{eq:inj:eabec}, \eqref{eq:inj:ecbec} and Case 5, we conclude that
$$\left[\widehat{\pi}\left(e_{ab}^{(r+1)}\right),\widehat{\pi}\left(e_{ab}^{(s+1)}\right)\right]=0,$$
since $\left[\widehat{\pi}\left(e_{a,b-1}^{(1)}\right),\widehat{\pi}\left(e_{b-1,b}^{(r+1)}\right)\right]=
\widehat{\pi}\left(e_{ab}^{(r+1)}\right)$.
\medskip

{\bf Case 7:} Let $a<c<c+1<d<b$. In this situation, it follows from \eqref{eq:inj:eabec} that
 \begin{align*}
\left[\widehat{\pi}\left(e_{ab}^{(r+1)}\right), \widehat{\pi}\left(e_{cd}^{(s+1)}\right)\right]
=&\left[\widehat{\pi}\left(e_{ab}^{(r+1)}\right),
\left[\widehat{\pi}\left(e_{c,c+1}^{(1)}\right),\widehat{\pi}\left(e_{c+1,d}^{(s+1)}\right)\right]\right]\\
=&
\left[\widehat{\pi}\left(e_{c,c+1}^{(1)}\right),\left[\widehat{\pi}\left(e_{ab}^{(r+1)}\right),\widehat{\pi}\left(e_{c+1,d}^{(s+1)}\right)\right]\right].
\end{align*}
Then it follows that $\left[\widehat{\pi}\left(e_{ab}^{(r+1)}\right), \widehat{\pi}\left(e_{cd}^{(s+1)}\right)\right]=0$ by induction on $d-c$.
\medskip

{\bf Case 8:} Let $a<c<b<d$. By Case 2, we have
\begin{align*}
\left[\widehat{\pi}\left(e_{ab}^{(r+1)}\right), \widehat{\pi}\left(e_{cd}^{(s+1)}\right)\right]
=&\left[\left[\widehat{\pi}\left(e_{ac}^{(r+1)}\right), \widehat{\pi}\left(e_{cb}^{(1)}\right)\right], \widehat{\pi}\left(e_{cd}^{(s+1)}\right)\right]\\
=&\left[\left[\widehat{\pi}\left(e_{ac}^{(r+1)}\right), \widehat{\pi}\left(e_{cd}^{(s+1)}\right)\right], \widehat{\pi}\left(e_{cb}^{(1)}\right)\right]
=\left[\widehat{\pi}\left(e_{ad}^{(r+s+1)}\right), \widehat{\pi}\left(e_{cb}^{(1)}\right)\right],
\end{align*}
which vanishes by Case 7.
\end{proof}

\begin{lemma}
\label{lem:grYD1}
The identity \eqref{eq:grhat:eet} holds in $\mathrm{gr}^{\prime}\YD(\mathfrak{q}_n)$.
\end{lemma}
\begin{proof}
The proof is similar to Lemma~\ref{lem:grYD1}. According to the equations \eqref{eq:Dr:eaeb}, \eqref{eq:Dr:eaeta}, \eqref{eq:Dr:etaeta}, \eqref{eq:Dr:eaea1}
 and \eqref{eq:Dr:eaeta1}, we know that the following relations hold in $\mathrm{gr}^{\prime}\YD(\mathfrak{q}_n)$.
 \begin{align}
&\left[\widehat{\pi}\left(e_a^{(r)}\right), \widehat{\pi}\left(\bar{e}_a^{(s)}\right)\right]=\left[\widehat{\pi}\left(\bar{e}_a^{(r)}\right), \widehat{\pi}\left(e_a^{(s)}\right)\right]=0, \label{eq:inj:eaeat}\\
&\left[\widehat{\pi}\left(e_a^{(r)}\right), \widehat{\pi}\left(\bar{e}_b^{(s)}\right)\right]=\left[\widehat{\pi}\left(\bar{e}_a^{(r)}\right), \widehat{\pi}\left(e_b^{(s)}\right)\right]=0,\qquad \text{if } |a-b|>1, \label{eq:inj:eaebt}\\
&\left[\widehat{\pi}\left(e_a^{(r)}\right), \widehat{\pi}\left(e_{a+1}^{(s)}\right)\right]=(-1)^r\left[\widehat{\pi}\left(\bar{e}_a^{(r)}\right), \widehat{\pi}\left(\bar{e}_{a+1}^{(s)}\right)\right], \label{eq:inj:eatea1t}\\
&\left[\widehat{\pi}\left(e_a^{(r)}\right), \widehat{\pi}\left(\bar{e}_{a+1}^{(s)}\right)\right]=(-1)^{r+1}\left[\widehat{\pi}\left(\bar{e}_a^{(r)}\right), \widehat{\pi}\left(e_{a+1}^{(s)}\right)\right]. \label{eq:inj:eaea1t}
 \end{align}
 
 {\bf Case 1:}  $a<b<c<d$. The identity
$\left[\widehat{\pi}\left(\bar{e}_{ab}^{(r+1)}\right), 
\widehat{\pi}\left(e_{cd}^{(s+1)}\right)\right]
=0$ in $\mathrm{gr}^{\prime}\YD(\mathfrak{q}_n)$ follows from the identity  \eqref{eq:inj:eaebt} and the definition of $\bar{e}_{a,b}^{(r)}$.

{\bf Case 2:} $a<b=c<d$. If $a=b-1<b=c<b+1=d$, we compare the coefficients of $u^{-r}v^{-s}$ in \eqref{eq:Dr:etcircle} and obtain the following relation in $\YD(\mathfrak{q}_n)$:
$$\left[\bar{e}_{b-1}^{(r+1)},e_b^{(s)}\right]
-\left[\bar{e}_{b-1}^{(r)},e_b^{(s+1)}\right]
=\bar{e}_{b-1}^{(r)}e_b^{(s)}+(-1)^re_{b-1}^{(r)}\bar{e}_b^{(s)}.$$
It yields that
$$\left[\widehat{\pi}\left(\bar{e}_{b-1}^{(r+1)}\right),\widehat{\pi}\left(e_b^{(s+1)}\right)\right]
=\left[\widehat{\pi}\left(\bar{e}_{b-1}^{(r+2)}\right),\widehat{\pi}\left(e_b^{(s)}\right)\right]
=\left[\widehat{\pi}\left(\bar{e}_{b-1}^{(r+s+1)}\right),\widehat{\pi}\left(e_b^{(1)}\right)\right]
=\widehat{\pi}\left(\bar{e}_{b-1,b+1}^{(r+s+1)}\right).$$
If $a<a+1<b=c<b+1=d$. By Case 9, we compute that
\begin{align*}
\left[\widehat{\pi}\left(\bar{e}_{ab}^{(r+t+1)}\right),
\widehat{\pi}\left(e_{b}^{(s+1)}\right)\right]
=&\left[\left[\widehat{\pi}\left(\bar{e}_{a,b-1}^{(r+1)}\right),
\widehat{\pi}\left(e_{b-1}^{(t+1)}\right)\right],
\widehat{\pi}\left(e_b^{(s+1)}\right)\right]\\
=&\left[\widehat{\pi}\left(\bar{e}_{a,b-1}^{(r+1)}\right),\left[
\widehat{\pi}\left(e_{b-1}^{(t+1)}\right),
\widehat{\pi}\left(e_b^{(s+1)}\right)\right]\right]\\
=&\left[\widehat{\pi}\left(\bar{e}_{a,b-1}^{(r+1)}\right),
\widehat{\pi}\left(e_{b-1,b+1}^{(s+t+1)}\right)\right].
\end{align*}
It follows that
\begin{align*}
\left[\widehat{\pi}\left(\bar{e}_{ab}^{(r+1)}\right),
\widehat{\pi}\left(e_{b}^{(s+1)}\right)\right]
=&\left[\widehat{\pi}\left(\bar{e}_{a,b-1}^{(r+1)}\right),
\widehat{\pi}\left(e_{b-1,b+1}^{(s+1)}\right)\right]\\
=&\left[\widehat{\pi}\left(\bar{e}_{ab}^{(r+s+1)}\right),
\widehat{\pi}\left(e_{b}^{(1)}\right)\right]
=\widehat{\pi}\left(\bar{e}_{a,b+1}^{(r+s+1)}\right).
\end{align*}

In general, if $b+1<d$, we conclude by induction on $d-b$ that
\begin{align}\label{eq:inj:beabebd}
\left[\widehat{\pi}\left(\bar{e}_{ab}^{(r+1)}\right),
\widehat{\pi}\left(e_{bd}^{(s+1)}\right)\right]
=\widehat{\pi}\left(\bar{e}_{ad}^{(r+s+1)}\right).
\end{align}

{\bf Case 3: } $d=c+1$ and $c<b$. We have to show that
\begin{equation}\label{eq:inj:beabec}
\left[\widehat{\pi}\left(\bar{e}_{ab}^{(r+1)}\right),\widehat{\pi}\left(e_c^{(s+1)}\right)\right]
=0,
\end{equation}
provided that $a\leqslant c$ and $c<b$. We consider the following subcases. 

Subcase 3-1: $a=c$ and $b=c+1$. In this situation, we have
\begin{equation}
\left[\widehat{\pi}\left(\bar{e}_{c}^{(r+1)}\right),
\widehat{\pi}\left(e_{c}^{(s+1)}\right)\right]=0,
\label{eq:inj:becec}
\end{equation}
by comparing the coefficients of $u^{-r}v^{-s}$ on both sides of \eqref{eq:Dr:eaeta}.

Subcase 3-2: $a<c$ and $b=c+1$. Deduce from 
$$\left[e_{a}^{(r)}, \bar{e}_{a+1}^{(s)}\right]=(-1)^{r+1}\left[\bar{e}_{a}^{(r)}, e_{a+1}^{(s)}\right],$$
we know that 
$$\bar{e}_{ab}^{(r)}:=\left[\bar{e}_{a,b-1}^{(r)}, e_{b-1}^{(1)}\right]=(-1)^{r+1}\left[e_{a,b-1}^{(r)}, \bar{e}_{b-1}^{(1)}\right].$$

The Serre relations \eqref{eq:Dr:eserre3} and \eqref{eq:Dr:eserre4} imply 
$$\left[\left[e_{c-1}^{(r+1)}, \bar{e}_c^{(1)}\right], e_c^{(1)}\right]=\left[\left[e_{c-1}^{(r+1)}, e_c^{(1)}\right], \bar{e}_c^{(1)}\right]=0.$$
Hence,
$$\left[\widehat{\pi}\left(\bar{e}_{c-1,c+1}^{(r+1)}\right),\widehat{\pi}\left(e_c^{(1)}\right)\right]=(-1)^{r+2}\left[\left[\widehat{\pi}\left(e_{c-1}^{(r+1)}\right), \widehat{\pi}\left(\bar{e}_c^{(1)}\right)\right], \widehat{\pi}\left(e_c^{(1)}\right)\right]=0,$$
for $r\geqslant0$. If $s>0$, it follows from Case 10 and \eqref{eq:inj:becec} that 
\begin{align*}
\left[\widehat{\pi}\left(\bar{e}_{c-1,c+1}^{(r+1)}\right), \widehat{\pi}\left(e_c^{(s+1)}\right)\right]
=&(-1)^{r+2}\left[\left[\widehat{\pi}\left(e_{c-1}^{(r+1)}\right), \widehat{\pi}\left(\bar{e}_c^{(1)}\right)\right], \widehat{\pi}\left(e_c^{(s+1)}\right)\right]\\
=&(-1)^{r+2}\left[ \widehat{\pi}\left(e_{c-1,c+1}^{(r+s+1)}\right), \widehat{\pi}\left(\bar{e}_c^{(1)}\right)\right]\\
=&(-1)^{s}\left[ \widehat{\pi}\left(\bar{e}_{c-1,c+1}^{(r+s+1)}\right), \widehat{\pi}\left(e_c^{(1)}\right)\right]=0.
\end{align*}
If $a<c-1$, 
$$\left[\widehat{\pi}\left(\bar{e}_{a,c+1}^{(r+1)}\right), \widehat{\pi}\left(e_c^{(s+1)}\right)\right]\\
=(-1)^{r+2}\left[\left[\widehat{\pi}\left(e_{a,c-1}^{(r+1)}\right), \widehat{\pi}\left(\bar{e}_{c-1,c+1}^{(1)}\right)\right], \widehat{\pi}\left(e_c^{(s+1)}\right)\right]=0,$$
since $\left[\widehat{\pi}\left(e_{a,c-1}^{(r+1)}\right), \widehat{\pi}\left(e_c^{(s+1)}\right)\right]=0$ by Case 1, and $\left[\widehat{\pi}\left(\bar{e}_{c-1,c+1}^{(1)}\right), \widehat{\pi}\left(e_c^{(s+1)}\right)\right]=0$ as we proved above.

Subcase 3-3: $a=c$ and $b>c+1$. If $b=c+2$, a similar argument as in Subcase 11-2 shows that
\begin{align*}
\left[\widehat{\pi}\left(\bar{e}_{c,c+2}^{(r+1)}\right),\widehat{\pi}\left(e_{c}^{(s+1)}\right)\right]
=&\left[\left[\widehat{\pi}\left(e_{c}^{(1)}\right),\widehat{\pi}\left(\bar{e}_{c+1}^{(r+1)}\right)\right],\widehat{\pi}\left(e_{c}^{(s+1)}\right)\right]\\
=&-\left[\widehat{\pi}\left(e_{c}^{(1)}\right),\left[\widehat{\pi}\left(\bar{e}_{c}^{(1)}\right),\widehat{\pi}\left(e_{c+1}^{(r+s+1)}\right)\right]\right]=0.
\end{align*}
For $b>c+2$, it follows from Case 8 that
\begin{align}\label{eq:inj:becbec}
    \left[\widehat{\pi}\left(\bar{e}_{cb}^{(r+1)}\right),\widehat{\pi}\left(e_{c}^{(s+1)}\right)\right]
    =&\left[\left[\widehat{\pi}\left(\bar{e}_{c,c+2}^{(r+1)}\right),
    \widehat{\pi}\left(e_{c+2,b}^{(1)}\right)\right],
    \widehat{\pi}\left(e_{c}^{(s+1)}\right)\right]\nonumber\\
    =&\left[\left[\widehat{\pi}\left(\bar{e}_{c,c+2}^{(r+1)}\right),
    \widehat{\pi}\left(e_{c}^{(s+1)}\right)\right],
    \widehat{\pi}\left(e_{c+2,b}^{(1)}\right)\right]=0.
\end{align}

Subcase 3-4: $a<c$ and $b>c+1$.
\begin{align*}
    \left[\widehat{\pi}\left(\bar{e}_{ab}^{(r+1)}\right),\widehat{\pi}\left(e_{c}^{(s+1)}\right)\right]
    =&\left[\widehat{\pi}\left(\bar{e}_{a,c+1}^{(r+1)}\right),
    \left[\widehat{\pi}\left(e_{c+1,b}^{(1)}\right),
    \widehat{\pi}\left(e_{c}^{(s+1)}\right)\right]\right]\\
    =&-\left[\widehat{\pi}\left(\bar{e}_{a,c+1}^{(r+1)}\right),
    \widehat{\pi}\left(e_{cb}^{(s+1)}\right)\right]\\
    =&-\left[\left[\widehat{\pi}\left(\bar{e}_{ac}^{(r+1)}\right),
    \widehat{\pi}\left(e_{cb}^{(s+1)}\right)\right],
    \widehat{\pi}\left(e_{c}^{(1)}\right)\right]\\
    =&-  \left[\widehat{\pi}\left(\bar{e}_{ab}^{(r+s+1)}\right),\widehat{\pi}\left(e_{c}^{(1)}\right)\right].
\end{align*}
Thus, we have $ \left[\widehat{\pi}\left(\bar{e}_{ab}^{(r+1)}\right),\widehat{\pi}\left(e_{c}^{(1)}\right)\right]=0$. Consequently,  $ \left[\widehat{\pi}\left(\bar{e}_{ab}^{(r+1)}\right),\widehat{\pi}\left(e_{c}^{(s+1)}\right)\right]=0$.
\medskip

{\bf Case 4:} Let $b=d$ and $d>c+1$. Since
\begin{align*}
\left[\widehat{\pi}\left(\bar{e}_{ad}^{(r+1)}\right), \widehat{\pi}\left(e_{cd}^{(s+1)}\right)\right]
=&\left[\widehat{\pi}\left(\bar{e}_{ad}^{(r+1)}\right),
\left[\widehat{\pi}\left(e_{c}^{(1)}\right),\widehat{\pi}\left(e_{c+1,d}^{(s+1)}\right)\right]\right]\\
=&
\left[\widehat{\pi}\left(e_{c}^{(1)}\right),\left[\widehat{\pi}\left(\bar{e}_{ad}^{(r+1)}\right),\widehat{\pi}\left(e_{c+1,d}^{(s+1)}\right)\right]\right].
\end{align*}
Then an induction on $d-c$ starting with \eqref{eq:inj:becbec} yields that
\begin{align}\label{eq:inj:beadecd}
\left[\widehat{\pi}\left(\bar{e}_{ad}^{(r+1)}\right), \widehat{\pi}\left(e_{cd}^{(s+1)}\right)\right]=0.
\end{align}
\medskip

{\bf Case 5:} Let $a=c$ and $b,d>a+1$. We shall show that
\begin{equation}
\left[\widehat{\pi}\left(\bar{e}_{ab}^{(r+1)}\right), \widehat{\pi}\left(e_{ad}^{(s+1)}\right)\right]=0.
\label{eq:inj:beabead}
\end{equation}
We may assume $b<d$ and deduce by \eqref{eq:inj:beabebd} that
\begin{align*}
\left[\widehat{\pi}\left(\bar{e}_{ab}^{(r+1)}\right), \widehat{\pi}\left(e_{ad}^{(s+1)}\right)\right]
=&\left[\left[\widehat{\pi}\left(\bar{e}_{a,b-1}^{(1)}\right),\widehat{\pi}\left(e_{b-1,b}^{(r+1)}\right)\right]
\widehat{\pi}\left(e_{ad}^{(s+1)}\right)\right]\\
=&\left[\left[\widehat{\pi}\left(\bar{e}_{a,b-1}^{(1)}\right),\widehat{\pi}\left(e_{ad}^{(s+1)}\right)\right],
\widehat{\pi}\left(e_{b-1,b}^{(r+1)}\right)\right].
\end{align*}
Hence, \eqref{eq:inj:beabead} follows by induction on $b-a$. The case where $b-a=1$ is verified in \eqref{eq:inj:becbec}.
\medskip

{\bf Case 6:} Let $a=c$, $b=d$, and $b>a+1$. Since $\widehat{\pi}\left(e_{a,b-1}^{(1)}\right)$ and $\widehat{\pi}\left(e_{b-1,b}^{(r+1)}\right)$ commute with $\widehat{\pi}\left(\bar{e}_{ab}^{(s+1)}\right)$ by \eqref{eq:inj:beabec}, \eqref{eq:inj:becbec} and Case 5, we conclude that
$$\left[\widehat{\pi}\left(\bar{e}_{ab}^{(r+1)}\right),\widehat{\pi}\left(e_{ab}^{(s+1)}\right)\right]=0,$$
since $\left[\widehat{\pi}\left(e_{a,b-1}^{(1)}\right),\widehat{\pi}\left(\bar{e}_{b-1,b}^{(r+1)}\right)\right]=
\widehat{\pi}\left(\bar{e}_{ab}^{(r+1)}\right)$.
\medskip

{\bf Case 7:} Let $a<c<c+1<d<b$. In this situation, it follows from \eqref{eq:inj:beabec} that
 \begin{align*}
\left[\widehat{\pi}\left(\bar{e}_{ab}^{(r+1)}\right), \widehat{\pi}\left(e_{cd}^{(s+1)}\right)\right]
=&\left[\widehat{\pi}\left(\bar{e}_{ab}^{(r+1)}\right),
\left[\widehat{\pi}\left(e_{c,c+1}^{(1)}\right),\widehat{\pi}\left(e_{c+1,d}^{(s+1)}\right)\right]\right]\\
=&
\left[\widehat{\pi}\left(e_{c,c+1}^{(1)}\right),\left[\widehat{\pi}\left(\bar{e}_{ab}^{(r+1)}\right),\widehat{\pi}\left(e_{c+1,d}^{(s+1)}\right)\right]\right].
\end{align*}
Then it follows that $\left[\widehat{\pi}\left(\bar{e}_{ab}^{(r+1)}\right), \widehat{\pi}\left(e_{cd}^{(s+1)}\right)\right]=0$ by induction on $d-c$.
\medskip

{\bf Case 8:} Let $a<c<b<d$. By Case 10, we have
\begin{align*}
\left[\widehat{\pi}\left(\bar{e}_{ab}^{(r+1)}\right), \widehat{\pi}\left(e_{cd}^{(s+1)}\right)\right]
=&\left[\left[\widehat{\pi}\left(\bar{e}_{ac}^{(r+1)}\right), \widehat{\pi}\left(e_{cb}^{(1)}\right)\right], \widehat{\pi}\left(e_{cd}^{(s+1)}\right)\right]\\
=&\left[\left[\widehat{\pi}\left(\bar{e}_{ac}^{(r+1)}\right), \widehat{\pi}\left(e_{cd}^{(s+1)}\right)\right], \widehat{\pi}\left(e_{cb}^{(1)}\right)\right]
=\left[\widehat{\pi}\left(\bar{e}_{ad}^{(r+s+1)}\right), \widehat{\pi}\left(e_{cb}^{(1)}\right)\right],
\end{align*}
which vanishes by Case 7.
\end{proof}

\begin{remark}
    It follows from \eqref{eq:Dr:haea}-\eqref{eq:Dr:haeta} and \eqref{eq:Dr:hafa}-\eqref{eq:Dr:hafta} that
\begin{align*}
    &\left[H_a^1(u), E_a^2(v)\right]=\tilde{J}^1\left[H_a^1(u), E_a^2(v)\right]\tilde{J}^2,\\
    &\left[H_a^1(u), F_a^2(v)\right]=\tilde{J}^1\left[H_a^1(u), F_a^2(v)\right]\tilde{J}^2,
\end{align*}
where $\tilde{J}=\sum\limits_{i\in I_{n|n}}E_{i,-i}$.
\end{remark}

\begin{remark}
Let $\Y^+(\mathfrak{q}_n)$ (resp. $\Y^-(\mathfrak{q}_n)$) be the subalgebra of $\Y(\mathfrak{q}_n)$ generated by the coefficients of $e_b(u)$ and $\bar{e}_b(u)$ (resp. $f_b(u)$ and $\bar{f}_b(u)$) for $b=1,2,\ldots,n$. Let $\Y^0(\mathfrak{q}_n)$ be the subalgebra of $\Y(\mathfrak{q}_n)$ generated by the coefficients of $h_a(u)$ and $\bar{h}_a(u)$ for $a=1,2,\ldots,n$. The relations \eqref{eq:Dr:haea}-\eqref{eq:Dr:etafta} show that the multiplication gives a surjective homomorphism
$$\Y^-(\mathfrak{q}_n)\otimes\Y^0(\mathfrak{q}_n)\otimes\Y^+(\mathfrak{q}_n)\rightarrow \Y(\mathfrak{q}_n).$$
It is indeed an isomorphism by the PBW theorem (see \cite[Corollary~2.4]{N99}). Moreover, it follows from \eqref{eq:Dr:hahb}-\eqref{eq:Dr:htahta} that the diagonal part $\Y^0(\mathfrak{q}_n)$ is isomorphic to $\Y(\mathfrak{q}_1)\otimes \Y(\mathfrak{q}_1)\otimes\cdots\otimes\Y(\mathfrak{q}_1)$, that is, the tensor product of $n$ copies of $\Y(\mathfrak{q}_1)$. We observe that $\Y^0(\mathfrak{q}_n)$ has a nontrivial odd part. 
\end{remark}

\section{The Center }\label{se:qCenter}

As an application of the Drinfeld presentation of $\Y(\mathfrak{q}_n)$ given in Theorem~\ref{thm:isomorphism}, we deduce an alternative expression of the generators of the center of $\Y(\mathfrak{q}_n)$ given in \cite[Theorem~3.4]{N99} in terms of Drinfeld generators.

We first review the generators of the center of $\Y(\mathfrak{q}_n)$ given in \cite{N99}. We set 
\begin{align}
\Lambda=&\left(\mathrm{id}\otimes\tau\right)(\mathsf{P})=\sum\limits_{i,j\in I_{n|n}}(-1)^{|i||j|}\mathsf{E}_{ij}\otimes\mathsf{E}_{ij},\label{eq:ct:lmd}\\
\Theta=&\left(\mathrm{id}\otimes\tau\right)(\mathsf{Q})=\sum\limits_{i,j\in I_{n|n}}
(-1)^{|i||j|+|i|+|j|}\mathsf{E}_{ij}\otimes\mathsf{E}_{-i,-j},\label{eq:ct:tht}
\end{align}
where $\tau$ is the anti-automorphism of the associative superalgebra $\mathrm{End}\left(\mathbb{C}^{n|n}\right)$ defined by $$\tau:\ \mathsf{E}_{ij}\mapsto(-1)^{|i||j|+|i|}\mathsf{E}_{ji},\quad i,j\in I_{n|n}.$$ It is shown in \cite{N99} that there is a series $z(u)=1+\sum\limits_{r\geqslant1}z_ru^{-r}\in\Y(\mathfrak{q}_n)[[u^{-1}]]$ such that
\begin{equation}
(\Lambda\otimes1)T^1(u)\tau_2\left(\widetilde{T}^2(u)\right)=\Lambda\otimes z(u),\label{eq:ct:TT}
\end{equation}
where $\widetilde{T}(u)$ is the inverse of $T(u)$, and $\tau_2$ denotes the action of $\tau$ on the second tensor factor. Moreover, $z(u)=z(-u)$ and elements $z_2, z_4, \ldots$ are free generators of the center of $\Y(\mathfrak{q}_n)$.

It is known from \cite{N22} that the images of $\Lambda$ and $\Theta$ acting on $\left(\mathbb{C}^{n|n}\right)$ are both one-dimensional\footnote{Here, $\Lambda$ and $\Theta$ correspond to $K$ and $-L$ in \cite[Section VI]{N22}, respectively.}. These operators satisfy the following relation:
\begin{equation}
\Lambda^2=\Theta^2=0,\qquad \Lambda\Theta=\Theta\Lambda=0.\label{eq:ct:ltrelation}
\end{equation}
Moreover, we have the following lemma.
\begin{lemma}
\label{lem:ct:lmdmt}
If $A(u)=\sum\limits_{i,j\in I_{n|n}}\mathsf{E}_{ij}\otimes a_{ij}(u)\in\mathrm{End}(\mathbb{C}^{n|n})\otimes\Y(\mathfrak{q}_n)$ satisfies $a_{ij}(u)=a_{-i,-j}(-u)$, then
$$(\Lambda\otimes1)\tau_2(A^2(u))=(\Lambda\otimes1) A^1(u),\qquad \tau_2(A^2(u))(\Theta\otimes1)=A^1(-u)(\Theta\otimes1),$$
and 
$$(\Lambda\otimes1)A^1(u)(\Lambda\otimes1)=\mathrm{str}(A(u)),$$
where $\mathrm{str}(A(u))=\sum\limits_{i\in I_{n|n}}(-1)^{|i|}a_{ii}(u).$
\end{lemma}

\begin{proof}
The proof can be completed by a direct computation. We omit the details here. 
\end{proof}

We first express the central series using the supertrace. Similar formulas for the Yangian $\Y(\mathfrak{gl}_n)$ can be found in \cite[Section~1.9]{M07}, for the quantum affine queer superalgebra have recently been provided in \cite{LMZ25}.

\begin{theorem}
Let $T(u)$ be the generator matrix of $\Y(\mathfrak{q}_n)$, and $z(u)$ be the central series given in \eqref{eq:ct:TT}. Then
$$z(u)=1-\mathrm{str}\left(T(u)\partial \widetilde{T}(u)\right),$$
where $\widetilde{T}(u)$ is the inverse of $T(u)$, and $\partial\widetilde{T}(u)$ is the formal derivative of $\widetilde{T}(u)$ with respect to $u$.
\end{theorem}
\begin{proof}
Since $T(u)$ and $\widetilde{T}(u)$ satisfy the relation \eqref{eq:TRT}, we have
\begin{align*}
T^1(u)\widetilde{T}^2(v)-\widetilde{T}^2(v)T^1(u)
=&\frac{1}{u-v}\left(T^1(u)\mathsf{P}\widetilde{T}^2(v)-\widetilde{T}^2(v)\mathsf{P}T^1(u)\right)\\
&+\frac{1}{u+v}\left(T^1(u)\mathsf{Q}\widetilde{T}^2(v)-\widetilde{T}^2(v)\mathsf{Q}T^1(u)\right)\\
=&\frac{1}{u-v}\left(T^1(u)\widetilde{T}^1(v)\mathsf{P}-\mathsf{P}\widetilde{T}^1(v)T^1(u)\right)\\
&+\frac{1}{u+v}\left(T^1(u)\widetilde{T}^1(-v)\mathsf{Q}-\mathsf{Q}\widetilde{T}^1(-v)T^1(u)\right).
\end{align*}
This implies that
\begin{align*}
T^1(u)\widetilde{T}^2(u)-\widetilde{T}^2(u)T^1(u)
=&\left(-T^1(u)\partial\widetilde{T}^1(u)\mathsf{P}+\mathsf{P}\partial\widetilde{T}^1(u)T^1(u)\right)\\
&+\frac{1}{2u}\left(T^1(u)\widetilde{T}^1(-u)\mathsf{Q}-\mathsf{Q}\widetilde{T}^1(-u)T^1(u)\right).
\end{align*}
Applying $\tau$ to the second tensor factor of the above identity, we derived that
\begin{align*}
T^1(u)\tau_2\left(\widetilde{T}^2(u)\right)-\tau_2\left(\widetilde{T}^2(u)\right)T^1(u)
=&\left(-T^1(u)\partial\widetilde{T}^1(u)\Lambda+\Lambda\partial\widetilde{T}^1(u)T^1(u)\right)\\
&+\frac{1}{2u}\left(T^1(u)\widetilde{T}^1(-u)\Theta-\Theta\widetilde{T}^1(-u)T^1(u)\right).
\end{align*}
Note that $\Lambda^2=0$ and $\Lambda\Theta=0$. By left multiplying $(\Lambda\otimes1)$ on both sides of the above equality, we obtain
\begin{align*}
(\Lambda\otimes1)T^1(u)\tau_2\left(\widetilde{T}^2(u)\right)
=&(\Lambda\otimes1)\tau_2\left(\widetilde{T}^2(u)\right)T^1(u)
-(\Lambda\otimes1)T^1(u)\partial\widetilde{T}^1(u)(\Lambda\otimes1)\\
&+(\Lambda\otimes1)T^1(u)\widetilde{T}^1(-u)(\Theta\otimes1).
\end{align*}
We further deduce by Lemma~\ref{lem:ct:lmdmt} that
\begin{align*}
(\Lambda\otimes1)\tau_2\left(\widetilde{T}^2(u)\right)T^1(u)
=&(\Lambda\otimes1)\widetilde{T}^1(u)T^1(u)=\Lambda\otimes1,\\
(\Lambda\otimes1)T^1(u)\partial\widetilde{T}^1(u)(\Lambda\otimes1)
=&\Lambda\otimes\mathrm{str}(T(u)\partial T(u)),\\
(\Lambda\otimes1)T^1(u)\widetilde{T}^1(-u)(\Theta\otimes1)
=&(\Lambda\otimes1)T^1(u)\tau_2\left(\widetilde{T}^2(u)\right)(\Theta\otimes1)
=(\Lambda\otimes z(u))(\Theta\otimes1)=0.
\end{align*}
Consequently, 
$$\Lambda\otimes z(u)=(\Lambda\otimes1)T^1(u)\tau_2\left(\widetilde{T}^2(u)\right)
=\Lambda\otimes\left(1-\mathrm{str}\left(T(u)\partial T(u)\right)\right).$$
This completes the proof.
\end{proof}

Applying the evaluation homomorphism of the Lemma \ref{lem:embeddingandevaluation} to the series $z(u)$, we have
\begin{corollary}
    We have the formula
    \begin{align}
    \mathrm{ev}(z(u))=1+\sum\limits_{k=1}^{\infty}\mathrm{str}\,\mathsf{G}^{k+1}u^{-k-1},
\end{align}
where $\mathsf{G}=\sum\limits_{i,j\in I_{n|n}}(-1)^{|j|}\mathsf{E}_{ij}\otimes\mathsf{g}_{ji}$ as in Section \ref{se:Yqdef}.
\end{corollary}
\begin{proof}
Since $\mathrm{ev}(T(u))=1-\mathsf{G}u^{-1}$, then
$$\mathrm{ev}(\widetilde{T}(u))=(1-\mathsf{G}u^{-1})^{-1}=\sum\limits_{k=0}^{\infty}\mathsf{G}^ku^{-k}.$$
Therefore,
 \begin{align*}
     \mathrm{ev}(T(u)\partial \widetilde{T}(u))=(1-\mathsf{G}u^{-1})\sum\limits_{k=1}^{\infty}(-k)\mathsf{G}^{k+1}u^{-k-2}
     =-\sum\limits_{k=1}^{\infty}\mathsf{G}^{k+1}u^{-k-1}.
 \end{align*}
  Then, the corollary follows from the above theorem.
\end{proof}

Next, we express the central series $z(u)$ in terms of Gauss generators. Recall that the generator matrix $T(u)$ of $\Y(\mathfrak{q}_n)$ and its inverse $\widetilde{T}(u)$ can be written in the block forms \eqref{eq:qnTu} and \eqref{eq:Ttildeblock}. We also write $\Lambda$ in block form:
$$\Lambda=\sum\limits_{a,b=1}^n\mathsf{E}_{ab}\otimes\mathsf{E}_{ab}\otimes \Lambda_0,$$
where $\Lambda_0=\sum\limits_{i,j\in I_{1|1}}(-1)^{|i||j|}\varepsilon_{ij}\otimes\varepsilon_{ij}$. 
Then, the identity \eqref{eq:ct:TT} is equivalent to
\begin{align}
\sum\limits_{p=1}^n(\Lambda_0\otimes1)T_{pb}^1(u)\tau_2\left(\widetilde{T}_{cp}^2(u)\right)=&\delta_{bc}\Lambda_0\otimes z(u).
\label{eq:ct:blcTT}
\end{align}

According to Theorem~\ref{thm:embedding}, there is an algebra homomorphism $\psi_1:\Y\left(\mathfrak{q}_{n-1}\right)\rightarrow\Y\left(\mathfrak{q}_n\right)$.
We denote the generator matrix of $\Y\left(\mathfrak{q}_{n-1}\right)$ by $\mathcal{T}(u)=\sum\limits_{a,b=2}^n\mathsf{E}_{ab}\otimes\mathcal{T}_{ab}(u)$. Then
\begin{equation}
\begin{aligned}
T_{11}(u)=&H_1(u), &
T_{ab}(u)=&F_{a1}(u)H_1(u)E_{1a}(u)+\psi_1\left(\mathcal{T}_{ab}(u)\right),\\
T_{1b}(u)=&H_1(u)E_{1b}(u),&
T_{a1}(u)=&F_{a1}(u)H_1(u),
\end{aligned}
\label{eq:ct:THEF1}
\end{equation}
for $2\leqslant a,b\leqslant n$. Moreover, let $\widetilde{\mathcal{T}}(u)=\sum\limits_{a,b=2}^n\mathsf{E}_{ab}\otimes\widetilde{\mathcal{T}}_{ab}(u)$ be the inverse of $\mathcal{T}(u)$. We have the following identities:
\begin{equation}
\begin{aligned}
\widetilde{T}_{11}(u)=&\widetilde{H}_1(u)+\sum\limits_{r,s=2}^nE_{1r}(u)\psi_1\left(\widetilde{\mathcal{T}}_{rs}(u)\right)F_{s1}(u),&
\widetilde{T}_{ab}(u)=&\psi_1\left(\widetilde{\mathcal{T}}_{ab}(u)\right),\\
\widetilde{T}_{1b}(u)=&-\sum\limits_{p=2}^nE_{1p}(u)\psi_1\left(\widetilde{\mathcal{T}}_{pb}(u)\right),&
\widetilde{T}_{a1}(u)=&-\sum\limits_{p=2}^n\psi_1\left(\widetilde{\mathcal{T}}_{ap}(u)\right)F_{p1}(u).
\end{aligned}
\label{eq:ct:THEF2}
\end{equation}
Moreover, we have the following lemma:
\begin{lemma}
Let $n\geqslant2$. The following identities hold in $\mathrm{End}\left(\mathbb{C}^{1|1}\right)\otimes \mathrm{End}\left(\mathbb{C}^{1|1}\right)\otimes\Y(\mathfrak{q}_n)[[u^{-1},v^{-1}]]$:
\begin{align}
\left[H_1^1(u),F_{r1}^2(v)\right]
=&\left(K(u,v)F_{r1}^1(u)-F_{r1}^2(v)K(u,v)\right)H_1^1(u),
\label{eq:ct:HFr}\\
\left[F_{r1}^1(u), \widetilde{T}_{nr}^2(v)\right]
=&-\sum\limits_{p=1}^n\widetilde{T}_{np}^2(v)K(u,v)F_{p1}^1(u),
\label{eq:ct:TFr}\\
\left[E_{1n}^1(u),F_{r1}^2(v)\right]
=&\widetilde{H}_1^1(u)K(u,v)\psi_1\left(\mathcal{T}_{rn}^1(u)\right)
-\psi_1\left(\mathcal{T}_{rn}^2(v)\right)K(u,v)\widetilde{H}_1^2(v),
\label{eq:ct:EFr}
\end{align}
where $r\geqslant2$.
\end{lemma}
\begin{proof}
The relations \eqref{eq:ct:HFr} and \eqref{eq:ct:TFr}  can be directly derived from \eqref{re:BTWT} and \eqref{q2:eq:HaHa}, respectively.  

For \eqref{eq:ct:EFr},  first by \eqref{re:BTT}, we have
\begin{align}
    \left[T_{1n}^1(u), H_1^2(v)\right]=K(u,v)T_{1n}^1(u)H_1^2(v)-T_{1n}^2(v)H_1^1(u)K(u,v).\label{eq:ct:TnH}
\end{align}
Also, by \eqref{re:BTT}, we have
\begin{align*}
    \left[E_{1n}^1(u),\ F_{r1}^2(v)\right]=&\widetilde{H}_1^1(u)\left(K(u,v)T_{rn}^1(u)-\left[H_1^1(u), F_{r1}^2(v)\right]E_{1n}^1(u)\right)\\
    &-\widetilde{H}_1^1(u)\left(T_{rn}^2(v)T_{11}^1(u)K(u,v)+F_{r1}^2(v)\left[T_{1n}^1(u), H_{1}^2(v)\right]\right)\widetilde{H}_1^2(v) \\
    =&\widetilde{H}_1^1(u)K(u,v)\psi_1\left(\mathcal{T}_{rn}^1(u)\right)
-\psi_1\left(\mathcal{T}_{rn}^2(v)\right)K(u,v)\widetilde{H}_1^2(v)\\
&+\widetilde{H}_1^1(u)K(u,v)F_{r1}^1(u)H_1^1(u)E_{1n}^1(u)-\widetilde{H}_1^1(u)\left[ H_1^1(u), F_{r1}^2(v)\right]E_{1n}^1(u)\\
&-\widetilde{H}_1^1(u)\left(F_{r1}^2(u)H_1^2(u)E_{1n}^2(u)H_1^1(u)K(u,v)+\left[ T_{1n}^1(u), H_{1}^2(v)\right]\right)\widetilde{H}_{1}^2(v)\\
=&\widetilde{H}_1^1(u)K(u,v)\psi_1\left(\mathcal{T}_{rn}^1(u)\right)
-\psi_1\left(\mathcal{T}_{rn}^2(v)\right)K(u,v)\widetilde{H}_1^2(v),
\end{align*}
where the last equation follows from \eqref{eq:ct:TFr} and \eqref{eq:ct:TnH}.
\end{proof}

\begin{lemma}
Let $n\geqslant2$. The following identities hold in $\mathrm{End}\left(\mathbb{C}^{1|1}\right)\otimes \mathrm{End}\left(\mathbb{C}^{1|1}\right)\otimes\Y(\mathfrak{q}_n)[[u^{-1},v^{-1}]]$:
\begin{align}
(\Lambda_0\otimes1)F_{r1}^1(u)H_1^1(u)
=&(\Lambda_0\otimes1)H_1^1(u)\tau_2\left(F_{r1}^2(u)\right),
\label{eq:ct:FrHuu}\\
\tau\left(\widetilde{T}_{nr}(u)F_{r1}(u)\right)
=&\tau\left(F_{r1}(u)\right)\tau\left(\widetilde{T}_{nr}(u)\right),
\label{eq:ct:FrTuu}
\end{align}
where $r\geqslant 2$.
\end{lemma}
\begin{proof}
We consider the tensor matrix $\tilde{J}=\varepsilon_{1,-1}+\varepsilon_{-1,1}\in\mathrm{End}(\mathbb{C}^{1|1})$. It can be verified that
$$(\widetilde{J}\otimes 1)P=P(1\otimes\widetilde{J}),\text{ and }
(\widetilde{J}\otimes 1)Q=Q(1\otimes\widetilde{J}),$$
which yields $(\widetilde{J}\otimes 1)K(u,v)=K(u,v)(1\otimes\widetilde{J})$.

It follows from \eqref{eq:ct:HFr} that
$$(\widetilde{J}\otimes1\otimes1)\left[H_1^1(u),F_{r1}^2(v)\right](1\otimes\widetilde{J}\otimes1)
=\left[H_1^1(u),F_{r1}^2(v)\right],$$
which is equivalent to 
$$\left[H_{1;ij}(u),F_{r1;kl}(v)\right]=-\left[H_{1;-i,j}(u),F_{r1;k,-l}(v)\right],$$
for all $i,j,k,l\in I_{1|1}$. In particular, we obtain
$$\sum\limits_{p\in I_{1|1}}(-1)^{(|i|+|p|)(|j|+|p|)}F_{r1;ip}(v)H_{1;pj}(u)
=\sum\limits_{p\in I_{1|1}}H_{pj}(u)F_{r1;ip}(v).
$$
Now, we compute that
\begin{align*}
&(\Lambda_0\otimes1)F_{r1}^1(u)H_1^1(u)\\
=&\sum\limits_{r,s,i,j,p}
\left(\varepsilon_{rs}\otimes\varepsilon_{rs}\otimes(-1)^{|r||s|}\right)
\left(\varepsilon_{ij}\otimes1\otimes (-1)^{(|i|+|p|)(|j|+|p|)}
F_{r1;ip}(u)H_{1;pj}(u)\right)\\
=&\sum\limits_{r,s,i,j,p}
\left(\varepsilon_{rs}\otimes\varepsilon_{rs}\otimes(-1)^{|r||s|}\right)
\left(\varepsilon_{ij}\otimes1\otimes H_{1;pj}(u)F_{r1;ip}(u)\right)\\
=&\sum\limits_{r,i,j,p}
\varepsilon_{rj}\otimes\varepsilon_{ri}\otimes (-1)^{|i||j|+|j||r|+|i|}H_{1;pj}(u)F_{r1;ip}(u)\\
=&\sum\limits_{r,i,j,p}
\left(\varepsilon_{rs}\otimes\varepsilon_{rs}\otimes(-1)^{|r||s|}\right)
\left(\varepsilon_{pj}\otimes1\otimes H_{1;pj}(u)\right)
\left(1\otimes\varepsilon_{pi}\otimes(-1)^{|i||p|+|i|}F_{r1;ip}(u)\right)\\
=&(\Lambda_0\otimes1)H_1^1(u)\tau_2\left(F_{r1}^2(u)\right).
\end{align*}
We complete verifying the identity \eqref{eq:ct:FrHuu}.

It follows from \eqref{eq:ct:TFr} that
\begin{align*}
\left[F_{r1}^1(u),\widetilde{T}_{nr}^2(v)\right]
=&\tilde{J}^1\left[F_{r1}^1(u),\widetilde{T}_{nr}^2(v)\right]\tilde{J}^2.
\end{align*}
Equivalently,
\begin{align*}
\left[F_{r1;ij}(u),\widetilde{T}_{nr;kl}(v)\right]
=&-\left[F_{r1;-i,j}(u),\widetilde{T}_{nr;k,-l}(v)\right].
\end{align*}
Since, the left hand side of equation \eqref{eq:ct:FrTuu} is equal to 
\begin{align*}
    \tau\left(\widetilde{T}_{nr}(u)F_{r1}(u)\right)
=&\sum\limits_{i,j,k}(-1)^{(|i|+|j|)(|j|+|k|)}\tau\left(\varepsilon_{ik}\otimes\widetilde{T}_{nr;ij}(u)F_{r1;jk}(u)\right)\\
=&\sum\limits_{i,j,k}(-1)^{|i|+|j|+(|i|+|k|)|j|}\varepsilon_{ki}\otimes\widetilde{T}_{nr;ij}(u)F_{r1;jk}(u).
\end{align*}
And, the right hand side of equation \eqref{eq:ct:FrTuu} is equal to
\begin{align*}
    \tau\left(F_{r1}(u)\right)\tau\left(\widetilde{T}_{nr}(u)\right)=\sum\limits_{i,j,k}(-1)^{|i|+|i||k|}\varepsilon_{ki}\otimes F_{r1;jk}(u)\widetilde{T}_{nr;ij}(u).
\end{align*}
Thus, 
\begin{align*}
  \tau\left(F_{r1}(u)\right)\tau\left(\widetilde{T}_{nr}(u)\right)-\tau\left(\widetilde{T}_{nr}(u)F_{r1}(u)\right)=\sum\limits_{i,k}(-1)^{|i|+|i||k|}\varepsilon_{ki}\otimes\sum\limits_{j}\left[F_{r1;jk}(u),\, \widetilde{T}_{nr;ij}(u) \right]=0.
\end{align*}
We complete verifying the identity \eqref{eq:ct:FrTuu}.
\end{proof}

\begin{theorem}
\label{thm:ct}
The central series $z(u)$ of $\Y(\mathfrak{q}_n)$ can be expressed in terms of the Gauss generators as follows:
\begin{equation}
\label{eq:center}
z(u)=\prod\limits_{a=1}^{n}\left(1-\mathrm{str}\left(H_a(u)\partial \widetilde{H}_a(u)\right)\right).
\end{equation}
where $\partial \widetilde{H}_a(u)$ is the formal derivative of $\widetilde{H}_a(u)$ with respect to $u$.
\end{theorem}

\begin{proof}
By \eqref{eq:ct:blcTT}, the central series $z(u)$ satisfies the identity:
$$\Lambda_0\otimes z(u)=\sum\limits_{p=1}^n(\Lambda_0\otimes1)T_{pn}^1(u)\tau_2\left(\widetilde{T}_{np}^2(u)\right).$$
Let $\psi_1:\Y(\mathfrak{q}_{n-1})\rightarrow\Y(\mathfrak{q}_n)$ be the embedding homomorphism given in Theorem~\ref{thm:embedding} and $z_{n-1}(u)$ be the central series of $\Y(\mathfrak{q}_{n-1})$. Then
$$\Lambda_0\otimes\psi_1(z_{n-1}(u))=\sum\limits_{r=2}^n(\Lambda_0\otimes1)\psi_1\left(\mathcal{T}_{rn}^1(u)\tau_2\left(\widetilde{\mathcal{T}}_{nr}^2(u)\right)\right).$$ 
According to \eqref{eq:ct:THEF1} and \eqref{eq:ct:THEF2}, we deduce
\begin{align*}
\Lambda_0\otimes z(u)
=&\Lambda_0\otimes\psi_1\left(z_{n-1}(u)\right)\\
&+\sum\limits_{r=2}^n\left(\Lambda_0\otimes1\right)
\left(F_{r1}^1(u)H_1^1(u)E_{1n}^1(u)\tau_2\left(\widetilde{T}_{nr}^2(u)\right)
-H_1^1(u)E_{1n}^1(u)\tau_2\left(\widetilde{T}_{nr}^2(u)F_{r1}^2(u)\right)\right).
\end{align*}
By \eqref{eq:ct:FrHuu} and \eqref{eq:ct:FrTuu}, we obtain
\begin{align*}
\Lambda_0\otimes z(u)
=&\Lambda_0\otimes\psi_1\left(z_{n-1}(u)\right)\\
&-\sum\limits_{r=2}^n\left(\Lambda_0\otimes1\right)H_1^1(u)
\left[E_{1n}^1(u),\tau_2(F_{r1}^2(u))\right]\tau_2\left(\widetilde{T}_{nr}^2(u)\right).
\end{align*}

Note that $\left[H_1^1(u),\psi_1\left(\widetilde{\mathcal{T}}_{nr}(v)\right)\right]=0$ for $r\geqslant2$. The identity \eqref{eq:ct:EFr} can be written as
\begin{align*}
\left[E_{1n}^1(u), F_{r1}^2(v)\right]
=&\frac{1}{u-v}
\left(\widetilde{H}_1^1(u)P\psi_1\left(\mathcal{T}_{rn}^1(u)\right)
-\widetilde{H}_1^1(v)P\psi_1\left(\mathcal{T}_{rn}^1(v)\right)\right)\\
&+\frac{1}{u+v}
\left(\widetilde{H}_1^1(u)Q\psi_1\left(\mathcal{T}_{rn}^1(u)\right)
-\widetilde{H}_1^1(-v)Q\psi_1\left(\mathcal{T}_{rn}^1(-v)\right)\right).
\end{align*}
Applying $\tau$ to the second tensor factor, we deduce that
\begin{align*}
\left[E_{1n}^1(u), \tau_2\left(F_{r1}^2(v)\right)\right]
=&\frac{1}{u-v}
\left(\widetilde{H}_1^1(u)\Lambda_0\psi_1\left(\mathcal{T}_{rn}^1(u)\right)
-\widetilde{H}_1^1(v)\Lambda_0\psi_1\left(\mathcal{T}_{rn}^1(v)\right)\right)\\
&+\frac{1}{u+v}
\left(\widetilde{H}_1^1(u)\Theta_0\psi_1\left(\mathcal{T}_{rn}^1(u)\right)
-\widetilde{H}_1^1(-v)\Theta_0\psi_1\left(\mathcal{T}_{rn}^1(-v)\right)\right),
\end{align*}
where $\Theta_0=\sum\limits_{i,j\in I_{1|1}}
(-1)^{|i||j|+|i|+|j|}\varepsilon_{ij}\otimes\varepsilon_{-i,-j}$. 
Let $v\rightarrow u$, then we have
\begin{align*}
\left[E_{1n}^1(u), \tau_2\left(F_{r1}^2(u)\right)\right]
=&\partial\widetilde{H}_1^1(u)\Lambda_0\psi_1\left(\mathcal{T}_{rn}^1(u)\right)
+\widetilde{H}_1^1(u)\Lambda_0\psi_1\left(\partial\mathcal{T}_{rn}^1(u)\right)\\
&+\frac{1}{2u}
\left(\widetilde{H}_1^1(u)\Theta_0\psi_1\left(\mathcal{T}_{rn}^1(u)\right)
-\widetilde{H}_1^1(-u)\Theta_0\psi_1\left(\mathcal{T}_{rn}^1(-u)\right)\right).
\end{align*}

Hence,
\begin{align*}
\Lambda_0\otimes z(u)
=&\Lambda_0\otimes\psi_1\left(z_{n-1}(u)\right)\\
&-\sum\limits_{r=2}^n\left(\Lambda_0\otimes1\right)H_1^1(u)\partial\widetilde{H}_1^1(u)(\Lambda_0\otimes1)
\psi_1\left(\mathcal{T}_{rn}^1(u)\tau_2(\widetilde{\mathcal{T}}_{nr}^2(u))\right)\\
&-\sum\limits_{r=2}^n\left(\Lambda_0\otimes1\right)H_1^1(u)\widetilde{H}_1^1(u)(\Lambda_0\otimes1)
\psi_1\left(\mathcal{T}_{rn}^1(u)\tau_2(\widetilde{\mathcal{T}}_{nr}^2(u))\right)\\
&-\frac{1}{2u}\sum\limits_{r=2}^n\left(\Lambda_0\otimes1\right)H_1^1(u)\widetilde{H}_1^1(u)(\Theta_0\otimes1)
\psi_1\left(\mathcal{T}_{rn}^1(u)\tau_2(\widetilde{\mathcal{T}}_{nr}^2(u))\right)\\
&+\frac{1}{2u}\sum\limits_{r=2}^n\left(\Lambda_0\otimes1\right)H_1^1(u)\widetilde{H}_1^1(-u)(\Theta_0\otimes1)
\psi_1\left(\mathcal{T}_{rn}^1(-u)\tau_2(\widetilde{\mathcal{T}}_{nr}^2(u))\right).
\end{align*}
Now, $\widetilde{H}_1(u)$ is the inverse of $H_1(u)$. It follows from Lemma~\ref{lem:ct:lmdmt} that
\begin{align*}
\Lambda_0\otimes z(u)
=&\left(1-\mathrm{str}\left(H_1(u)\partial \widetilde{H}_1(u)\right)\right)\Lambda_0\otimes\psi_1\left(z_{n-1}(u)\right).
\end{align*}
Consequently, the identity \eqref{eq:center} can be obtained by induction on $n$.
\end{proof}

\begin{remark}
The subalgebra of $\Y(\mathfrak{q}_n)$ generated by the coefficients of all entries of $H_a(u)$ is isomorphic to $\Y(\mathfrak{q}_1)$. Nazarov introduced the quantum Berezinian $C_a(u)=h_a(u)\widetilde{h}_a(-u)$ for $\Y(\mathfrak{q}_1)$, see \cite[Section VIII]{N22} for more details. Furthermore, the central series of $\Y(\mathfrak{q}_1)$ can be expressed as
$$z_a(u)=1-\mathrm{str}(H_a(u)\partial H_a(u))=C_a(u)C_a(-u)-D_a(u)D_a(-u),$$
where $C_a(u)-C_a(-u)=4uD_a(u)$. By Theorem~\ref{thm:ct}, the central series $z(u)$ of $\Y(\mathfrak{q}_n)$ can be expressed as:
$$z(u)=\prod\limits_{a=1}^n\left(C_a(u)C_a(-u)-D_a(u)D_a(-u)\right).$$
\end{remark}

\bigskip

The authors have no conflicts of interest to declare that are relevant to this article.

\end{document}